\documentclass[a4paper,reqno,10pt,authoryear]{amsart}
\usepackage{fancyhdr}
\usepackage{dsfont}
\usepackage[mathscr]{eucal}
\usepackage[cp1251]{inputenc}
\usepackage[english]{babel}
\usepackage{cite,enumerate,float,indentfirst}
\usepackage{graphicx}
\usepackage{xcolor}
\usepackage{mathrsfs,amsthm,amsmath,url}
\usepackage{latexsym,a4,amssymb,amsfonts}
\usepackage{epstopdf}
\usepackage{mathtools}
\usepackage{etoolbox}
\usepackage{color}

\definecolor{blue(pigment)}{rgb}{0.2, 0.2, 0.6}
\definecolor{ultramarine}{rgb}{0.07, 0.04, 0.56}
\definecolor{darkspringgreen}{rgb}{0.09, 0.45, 0.27}
\definecolor{hookersgreen}{rgb}{0.0, 0.44, 0.0}
\definecolor{plum(traditional)}{rgb}{0.56, 0.27, 0.52}
\definecolor{purple(html/css)}{rgb}{0.5, 0.0, 0.5}
\definecolor{magenta(dye)}{rgb}{0.79, 0.08, 0.48}

\usepackage{hyperref}

\hypersetup{colorlinks,
	linkcolor=hookersgreen,
	linktoc=blue,
	citecolor=ultramarine,
	urlcolor=black,
	filecolor=black
}
\usepackage{nameref}

\numberwithin{equation}{section}

\newtheorem{theorem}{Theorem}[section]
\newtheorem{lemma}[theorem]{Lemma}
\newtheorem{statement}[theorem]{Proposition}

\newtheorem{definition}[theorem]{Definition}
\newtheorem{corollary}[theorem]{Corollary}
\theoremstyle{remark}
\newtheorem*{remark}{Remark}

\newcommand{\R}{\mathbb{R}}
\newcommand{\C}{\mathbb{C}}
\newcommand{\J}{\mathbb{J}}
\newcommand{\K}{\mathbb{K}}

\newcommand{\T}{\mathbb{T}}

\newcommand{\TJ}{\T_\J}
\newcommand{\TK}{\T_\K}

\newcommand{\Cond}{{\bf (C0)}}
\newcommand{\CondTwo}{{\bf (C1)}}

\newcommand{\X}{{\bf X}}

\newcommand{\A}{{\bf A}}

\newcommand{\OO}{{\bf O}}

\newcommand{\II}{{\bf I}}
\newcommand{\RR}{{\bf R}}

\newcommand{\V}{{\bf V}}
\newcommand{\Lam}{{\bf \Lambda}}
\newcommand{\TTT}{{\bf T}}
\newcommand{\y}{{\bf y}}
\newcommand{\rr}{{\bf r}}

\newcommand{\W}{{\bf W}}

\newcommand{\Y}{{\bf Y}}

\newcommand{\ve}{\varepsilon}

\newcommand{\tTj}{\widetilde T_{n}^{(j)}}

\def\rr{\mathtt{r}}

\DeclareMathOperator{\Tr}{Tr}

\DeclareMathOperator{\E}{\mathbb{E}}

\DeclareMathOperator{\Pb}{\mathbb{P}}

\DeclareMathOperator{\one}{\mathds{1}}
\DeclareMathOperator{\re}{Re}
\DeclareMathOperator{\imag}{Im}

\DeclareMathOperator{\supp}{supp}

\DeclareMathOperator{\eqdef}{\,\,\mathrel{\overset{\makebox[0pt]{\mbox{\normalfont\tiny\sffamily def}}}{=}}\,\,}

\DeclareMathOperator{\tow}{\,\,\mathrel{\overset{\makebox[0pt]{\mbox{\normalfont\tiny\sffamily w}}}{\longrightarrow}}\,\,}

\begin{document}

\vspace{1in}

\title[Local Laws for non-Hermitian random matrices]{\bf On  Local laws for non-Hermitian random matrices\\ and their products}

\author[F. G{\"o}tze]{F. G{\"o}tze}
\address{Friedrich G{\"o}tze\\
 Faculty of Mathematics\\
 Bielefeld University \\
 Bielefeld, Germany
}
\email{goetze@math.uni-bielefeld.de}

\author[A. Naumov]{A. Naumov}
\address{Alexey A. Naumov\\
 National Research University Higher School of Economics, Moscow, Russia
 }
\email{anaumov@hse.ru}

\author[A. Tikhomirov]{A. Tikhomirov}
\address{Alexander N. Tikhomirov\\
 Department of Mathematics\\
 Komi Science Center of Ural Division of RAS \\
 Syktyvkar, Russia; and National Research University Higher School of Economics, Moscow, Russia
 }
\email{tikhomirov@dm.komisc.ru}



\date{\today}

\begin{abstract}
The aim of this paper is to prove a local version of the circular law for non-Hermitian random matrices and its generalization to the product of non-Hermitian random matrices under weak moment conditions. More precisely we assume that the entries \(X_{jk}^{(q)}\) of non-Hermitian random matrices \(\X^{(q)}, 1 \le j,k \le n, q = 1, \ldots, m, m \geq 1\) are i.i.d. r.v. with \(\E X_{jk}^{(q)} =0, \E |X_{jk}^{(q)}|^2 = 1\) and \(\E |X_{jk}^{(q)}|^{4+\delta} < \infty\) for some \(\delta > 0\). It is shown that the local law holds on the optimal scale \(n^{-1+2a}, 0 < a < 1/2\), up to some logarithmic factor.  We further develop a Stein type method to estimate the perturbation of the equations for the Stieltjes transform of the limiting distribution. We also generalize the recent results~\cite{Bourgade2014a},~\cite{TaoVu2015a} and~\cite{nemish2017}. 
 
\end{abstract}

\maketitle

\section{Introduction and main result}
One of the main questions of the Random matrix theory (RMT) is to investigate the limiting behaviour of spectra of random matrices from different ensembles.  In the current paper we shall study  the case of products of non-Hermitian random matrices. More precisely, we consider a set of random non-Hermitian matrices 
\begin{eqnarray*}
\X^{(q)} = \big[X_{jk}^{(q)}\big]_{j,k=1}^n, q = 1, \ldots, m, \quad  m \in \mathbb N.
\end{eqnarray*}
Assume that \(X_{jk}^{(q)}, 1\le j,k \le n, q = 1, \ldots, m\), are independent random variables (r.v.) with zero mean. Note that the distribution of \(X_{jk}^{(q)}\) may depend on \(n\). Denote by \((\lambda_1(\X), ... , \lambda_n(\X))\) -- the eigenvalues of the matrix 
\begin{eqnarray*}
\X \eqdef \frac{1}{n^\frac m2} \prod_{q=1}^m \X^{(q)}.
\end{eqnarray*} 
For any set \( B \in \mathfrak B(\C) \) we introduce the counting function of the eigenvalues in \(B\):
\begin{eqnarray*}
N_B \eqdef N_B(\X) \eqdef \# \{1 \le k \le n: \lambda_k(\X) \in B\}. 
\end{eqnarray*}
It is also convenient to denote by \( \mu_n(\cdot)\) --  the empirical spectral distribution of \(\X\): 
\begin{eqnarray*}
\mu_n(B) \eqdef \frac{1}{n} N_B, \quad B \in \mathfrak B(\C).
\end{eqnarray*}
We first assume that \( m = 1\). 
Denote
\begin{eqnarray}
p^{(1)}(z) \eqdef \frac{1}{\pi} \one[|z|\le 1], \quad z \in \C,
\end{eqnarray}
and let \(A(\cdot)\) be the Lebesgue measure on \(\C\).  By \(\tow  \) we denote weak convergence of probability measures. We first assume that \( m = 1\). Then the following result is the well-known {\it circular law}. 
\begin{theorem}[Macroscopic circular law]\label{th: circular law global regime} Let \(X_{jk}, 1 \le j, k \le n\) be i.i.d. complex r.v. with \(\E X_{jk} = 0, \E |X_{jk}|^2 = 1\). Then \(\mu_n \tow \mu^{(1)}\) \it{a.s.} as \(n\) tends to infinity, where 
\begin{eqnarray*}
d\mu^{(1)}(z) =  p^{(1)}(z) dA(z).
\end{eqnarray*}
\end{theorem}
The circular law was first proven by Ginibre~\cite{Ginibre1965} in 1965 in the case when \(X_{jk}\) are standard complex Gaussian r.v. His proof was based on the joint density of \((\lambda_1(\X), \ldots, \lambda_n(\X))\). If \(X_{jk}\) are complex (real) Gaussian r.v. we say that \(\X\) belongs to complex (resp. real) {\it Ginibre ensemble} of random matrices. Here we also refer to the book of M. Mehta~\cite{Mehta}. Later on the circular law was extended to more general classes of random entries by V. Girko~\cite{Girko1984}. Therefore the circular law is often referred to as the {\it Girko--Ginibre circular law}. It has been  further extended in a number of papers, for instance in~\cite{BaiSilv2010},~\cite{GotTikh2007},~\cite{GotTikh2010},~\cite{PanZhou2007},~\cite{TaoVu2007},~\cite{TaoVu2010a}. In particular,  F. G\"otze and A. Tikhomirov, see \cite{GotTikh2010}, established  the circular law under the assumption that \(\max_{j,k}\E |X_{jk}|^2 \log^{19+\eta}(1+|X_{jk}|) < \infty\) for any \(\eta > 0\). They also generalized it to the case of sparse random matrices as well. That is, let us define \(\X^\ve \eqdef (np)^{-1/2} [ X_{jk}^{(1)} \ve_{jk}]_{j,k=1}^n\), where \( \ve_{jk} \)  are i.i.d. Bernoulli r.v. with parameter \(p\). It follows from~\cite{GotTikh2010}, that one may take \( p \geq c \log n /n \) for some \(c > 0\). A result with optimal moment conditions, see Theorem~\ref{th: circular law global regime}, was established by T. Tao and V. Vu in~\cite{TaoVu2010a}. The progress made in \cite{GotTikh2010}, \cite{PanZhou2007}, \cite{TaoVu2007}, \cite{TaoVu2010a} was based on bounds for the least singular value of shifted matrices \( \X - z \II, z \in \C \), due to M. Rudelson and R. Vershynin, see e.g. \cite{RudVesh2008}. For a detailed account we refer the interested reader to the  overview~\cite{BordenaveChafai2014}.

In applications the case of the non-homogeneous circular law is of considerable interest, which means dropping the assumption of identical distribution of entries, while still assuming that \( \E |X_{jk}^{(1)}|^2 = \sigma_{jk}^2 \). In particular, the papers \cite{GotTikh2010}, \cite{TaoVu2010a} already deal with non i.i.d. entries but under the  additional assumption that all \( \sigma_{jk}^2 = 1\). An extended model  one would require some appropriate conditions on the matrix \( \Sigma \eqdef [\sigma_{jk}^2 ]_{j,k=1}^n\). For example, one may assume that \(\Sigma\) is doubly-stochastic. See e.g.~\cite{Adamczak2011}, \cite{Cook2017}, \cite{Erdos2017}.     

The circular law may be further generalized to the case of dependent r.v. The typical example here is the case of matrices from { \it Girko's elliptic ensemble}. Here the pairs \((X_{jk}, X_{kj}), 1 \le j < k < n\), are i.i.d. random vectors, and \(\E X_{jk} X_{kj} = \rho\), for some \(\rho: |\rho| \le 1\). The global limiting distribution for spectra of elliptic random matrices is given by a uniform law in the ellipsoid with the semi-axes equal to \(1+\rho\) and \(1-\rho\) resp. We refer the interested reader to the papers to~\cite{Girko1985},~\cite{Nau2013b},~\cite{ORourkeNguyen2012},\cite{GotNauTikh2013sing}. In the case \(\rho = 1\) we get Wigner's semicircle law, \cite{Wigner1955}. In the case \(\rho = 0\) and under additional assumption that \(\X\) belongs to Ginibre's ensemble we again arrive at the circular law. For other models of non-Hermitian random matrices with dependent entries, see for instance~\cite{Adamczak2011}, \cite{Bordenave2012}, \cite{Adamczak2016}.

The main emphasis of the current paper is concerned with the generalization of the circular law to the case of arbitrary \(m \geq 1\). We denote
\begin{eqnarray}
p^{(m)}(z) \eqdef \frac{1}{\pi m} |z|^{\frac2m - 2} \one[|z|\le 1].
\end{eqnarray}
It is straightforward to check that \( p^{(m)}(z)\) is the density of \(m\)-th power of a uniform distribution on the unit circle. 
We state the following theorem in the macroscopic scale.  
\begin{theorem}[Products of random matrices, macroscopic regime]\label{th: products global regime} 
	Let \(m \in \mathbb N\) and \(X_{jk}^{(q)}, 1 \le j, k \le n, q = 1, \ldots, m\), be i.i.d. complex r.v. with \(\E X_{jk}^{(q)} = 0, \E |X_{jk}^{(q)}|^2 = 1\). Then \(\mu_n \tow \mu^{(m)}\) \it{in probability} as \(n\) tends to infinity, where 
	\begin{eqnarray*}
	d\mu^{(m)}(A) = p^{(m)}(z)dA(z).
    \end{eqnarray*}	
\end{theorem}
We refer here to the results of F. G\"otze and A. Tikhomirov~\cite{GotTikh2011} and S. O'Rourke and A. Soshnikov~\cite{Soshnikov2011}. For product of Girko's elliptic random matrices, see~\cite{Soshnikov2015} and~\cite{GotNauTikh2015prod}. 

The circular law and its generalisation to the product of random matrices are valid, in particular, for all circles \(B(z_0, r)\) with centre at \(z_0\) and  finite radius \(r > 0\) independent of \(n\). Such sets  typically contain a  macroscopically large number of eigenvalues, which means a number of order \(n\). In particular, the statement of Theorem~1.1--1.2 may be formulated as follows:
\begin{eqnarray}\label{for circular}
\frac{1}{nr^2}\E N_{B(z_0,r)}  = \frac{1}{r^2} \int_{B(z_0,r)} p^{(m)}(z) dA(z) + \frac{R_n}{r^2},
\end{eqnarray}
where 
\begin{eqnarray}\label{asymptotics}
\lim_{n \rightarrow \infty} R_n = 0. 
\end{eqnarray}
(similar statement  may be formulated for \( N_{B(z_0,r)} \)). Unfortunately for  smaller radius, when \(r\) tends to zero as \(n\) goes to infinity, the number of eigenvalues cease to be macroscopically large.  In this case it is essential to describe the second term in~\eqref{for circular}  more precisely, rather then~\eqref{asymptotics}. We say that the { \it local law} holds if the second term in~\eqref{for circular} tends to zero as \(r = r(n)\) tends to zero. The series of the results in that direction was recently proved by P. Bourgade, H.-T. Yau  and J. Yin~\cite{Bourgade2014a},~\cite{Bourgade2014b},~\cite{Yin2014c} and T. Tao and V. Vu~\cite{TaoVu2015a} in the case of \( m = 1\). They derived the local version of Theorem~\ref{th: circular law global regime} up to the optimal scale \(n^{-1+2a}, a > 0\).  In ~\cite{Bourgade2014a},~\cite{Bourgade2014b},~\cite{Yin2014c} the local circular law was proved under the assumption of sub-exponential tails for the distribution of entries (or assuming finite moments of all orders).  In \cite{TaoVu2015a} it was proved under similar assumptions by means of the so-called {\it fourth moment theorem}, which requires that the first four moments of \( X_{jk} \) match the corresponding moments of the standard Gaussian distribution. We also refer to the recent results~\cite{Erdos2017} and~\cite{xi2017}. The general case of \(m \ge 1\) was proved by Y. Nemish~\cite{nemish2017} who obtained a local version of Theorem~\ref{th: products global regime} under sub-exponential assumptions. For a more detailed discussion of these result, see  the next section after Theorem~\ref{th:main res}.

The aim of the current paper is to relax the above assumptions and prove local versions of Theorem~\ref{th: circular law global regime} and   Theorem~\ref{th: products global regime} under weak moment condition. More precisely we assume that \(4+\delta\) moments are finite for some \(\delta>0\). See the following section~\ref{sec: main result} for precise statements. This work continues the previous results of authors~\cite{GotzeNauTikh2015a},~\cite{GotzeNauTikh2015b}, where {\it the local semicircle law} for Hermitian random matrices was proved under similar moment conditions.

We continue to use Stein type  methods for the estimation of perturbations of the  equation for Stieltjes transforms of the limiting distribution, since it turn out to be very flexible and useful. In this context we provide a general result, i.e. Lemma~\ref{lem: T_n general lemma}, which may be of independent interest. In particular, as a consequence of this lemma one may derive   among others a Rosenthal type inequality for moments of linear forms (e.g.~\cite{Rosenthal1970}[Theorem~3] and~\cite{JohnSchecttmanZinn1985}[Inequality~(A)]), inequality for moments of quadratic forms (e.g.~\cite{GineLatalaZinn2000}[Proposition~2.4] or~\cite{GotTikh2003}[Lemma A.1]) with precise values of all constants involved. We also jointly apply the {\it additive descent} method introduced by  L.~Erd{\"o}s, B.~Schlein, H.-T.~Yau and et al., see \cite{ErdosSchleinYau2009}, \cite{ErdosSchleinYau2009b}, \cite{ErdosSchleinYau2010}, \cite{ErdKnowYauYin2012},~\cite{ErdKnowYauYin2013}, \cite{ErdKnowYauYin2013a}, \cite{LeeYin2014} among others, together with  multiplicative descent methods introduced in~\cite{Schlein2014} and further developed in~\cite{GotzeNauTikh2015a},~\cite{GotzeNauTikh2016a}. See Lemma~\ref{discent} for details. 

We finish this section discussing some related results. In particular, we have already mentioned the local semicircle law. Significant progress in studying the local semicircular law was made in a series of papers by L.~Erd{\"o}s, B.~Schlein, H.-T.~Yau and et al., \cite{ErdosSchleinYau2009}, \cite{ErdosSchleinYau2009b}, \cite{ErdosSchleinYau2010}, \cite{ErdKnowYauYin2012},~\cite{ErdKnowYauYin2013}, \cite{ErdKnowYauYin2013a}, \cite{LeeYin2014}. We also refer to the more recent results~\cite{Schlein2014},~\cite{GotzeNauTikh2015a},~\cite{GotzeNauTikh2015b}. An extension to the {\it elliptic} random matrix ensembles, which generalizes both ensembles considered above would be of interest.  This applies as well to local versions of the elliptic law and its extension to  products of such matrices. See~\cite{Soshnikov2015} and~\cite{GotNauTikh2015prod} for the limiting behaviour in the macroscopic regime.  In particular, it would be interesting to study the so-called {\it weak non-Hermicity limit}, i.e. the case \(\rho\) tends to one, see~\cite{Fyodorov1997} and recent result~\cite{AkCikVenker2016}.

\subsection{Notations}
\label{notations} 
Throughout the paper we will use the following notations. We assume that all random variables are defined on  a common probability space \((\Omega, \mathcal F, \Pb)\)  writing \(\E\) for the mathematical expectation with respect to \(\Pb\). We denote by \(\R\) and \(\C\) the set of all real and complex numbers. We also introduce \(\C^{+} \eqdef \{z \in \C: \imag  z \geq 0\}\).
\begin{enumerate}
	\item We denote by \(\one[A]\) the indicator function of the set \(A\).
	\item By \(C\) and \(c\) we denote some positive constants. If we write that \(C\) depends on \(\delta\) we mean that \(C = C(\delta, \mu_{4+\delta})\).
	\item For an arbitrary square matrix \(\A\) taking values in \(\C^{n \times n}\) (or \(\R^{n \times n}\)) we define the operator norm by \(\|\A\| \eqdef \sup_{x \in \R^n: \|x\| = 1} \|\A x\|_2\), where \(\|x\|_2 \eqdef (\sum_{j = 1}^n |x_j|^2)^{1/2}\). We use the Hilbert-Schmidt (Frobenius) norm given by \(\|\A\|_2 \eqdef \Tr^{1/2} \A \A^{*} = (\sum_{j,k = 1}^n |\A_{jk}|^2)^{1/2}\). 
	\item For a vector \( x = (x_1, \ldots, x_n)^\mathsf{T} \) we denote \( |x| \eqdef \max_{ 1 \le k \le n} |x_k|\). 
	\item For an arbitrary function from \(L_1(\C)\)-space we denote \(\|f\|_{L_1} \eqdef \int_{\C} |f(z)| \, dz\). 
	\item For an arbitrary function \(f\) we denote \( \|f\|\eqdef \sup_{z \in \C} |f(z)|\).   
	\item Define the Laplace operator in two dimensions as \(\Delta \eqdef \frac{\partial^2}{\partial x^2} + \frac{\partial^2}{\partial y^2}\).  
	\item We write \( f \sim g\) if there exist positive constants \(c_1, c_2\), such that \(c_1 |g| \le |f| \le c_2 |g| \). We shall often write\( f \lesssim g\) which mean that there exists positive constant \(c\) such that \(|f| \le c |g|\). 
\end{enumerate}

\section{Main result}\label{sec: main result}
Without loss of generality we will assume in what follows that \(\X^{(q)}\) are real non-symmetric matrices. Our results proven below apply to the case of complex matrices as well. Here we may additionally assume for simplicity that \(\re X_{jk}^{(q)}\) and \(\imag X_{jk}^{(q)}\) are independent r.v. for all \(1 \le j, k \le n, q = 1, \ldots, m\). 
Otherwise  one needs to extend the moment inequalities for linear and quadratic forms in complex r.v. (see~\cite{GotzeNauTikh2015a}[Theorem A.1-A.2]) to the case of dependent real and imaginary parts, the details of which we omit.  

We will often refer to the following conditions. 
\begin{definition}[Conditions \(\Cond\)]
We say that conditions \(\Cond\) hold if:
\begin{itemize}
	\item \quad \(X_{jk}^{(q)}, 1\le j, k\le n, q =1, \ldots, m\), are independent real random variables;
	\item \quad \(\E X_{jk}^{(q)} = a_{jk}^{(q)}, \, \E |X_{jk}^{(q)}|^2 = [\sigma_{jk}^{(q)}]^2 \);  
	\item \quad \(\max_{j,k,q, n}\E|X_{jk}^{(q)}|^{4+\delta} \eqdef \mu_{4+\delta} < \infty\) for some \(\delta > 0\) independent of \(n\); 
	\item \quad \(|a_{jk}^{(q)}| \le n^{-1 - \delta_0}, \, |1 - [\sigma_{jk}^{(q)}]^2| \le n^{-1-\delta_0}\) for some \( \delta_0 > 0 \) independent of \(n\).
\end{itemize}    	
\end{definition}
\begin{definition}[Conditions \(\CondTwo\)] We say that conditions \(\CondTwo\) hold if:
\begin{itemize}
	\item \quad \(\Cond\) hold;
	\item \quad There exists \( \phi = \phi(\delta) > 0 \)  such that \( |X_{jk}^{(q)}| \le D n^{1/2 - \phi} \) for all \( 1 \le j,k \le n\) and some \( D > 0\).  
\end{itemize}
Here one may take \( 0 < \phi \leq \delta/(2(4+\delta))\). 	
\end{definition}

Let \(f(z)\) be a smooth non-negative function with compact support, such that \(\|f\| \le C, \|f'\| \le n^C\) for some constant \(C\) independent of \(n\). Following\cite{Bourgade2014a}, we define for any \(a \in (0, 1/2)\) and \(z_0 \in \C\) the function \(f_{z_0}(z) \eqdef n^{2a} f((z-z_0)n^a)\) ( \(f_{z_0} \) is a smoothed delta-function at the point \(z_0\)). The main result of the current paper is the following theorem which provides a local version of  Theorem~\ref{th: circular law global regime} and   Theorem~\ref{th: products global regime}  under weak moment conditions \( \CondTwo\). 
\begin{theorem}[Local regime]\label{th:main res}
Assume that the conditions \(\CondTwo\) hold. Let \(z_0: ||z_0|-1| \geq \tau > 0\). Then for any \(Q > 0\) there exists constant \(c = c(\delta, \delta_0) > 0\) such that with probability at least \(1 - n^{-Q}\)
\begin{eqnarray}\label{eq: main result 0} 
\bigg|\frac{1}{n} \sum_{j=1}^n f_{z_0}(\lambda_j) - \int f_{z_0} (z) d\mu^{(m)}(z) \bigg| \le \frac{q(n) }{n^{1-2a}} \|\Delta f\|_{L^1},
\end{eqnarray}
where \(q(n) \le c \log^5 n\). 
\end{theorem}
An immediate corollary of the main theorem is the following statement. 
\begin{corollary}\label{cor 2.4} Assume that the conditions \( \Cond\) hold. Then the inequality~\eqref{eq: main result 0} holds with probability at least \(1 - n^{-c(\delta)}\), where \(c(\delta)\) is some positive constant. 
\end{corollary}
\begin{proof}[Proof of Corollary~\ref{cor 2.4}] Let \( \widehat \X^{(q)}\) be \(\X^{(q)}\) with \( X_{jk}^{(q)}\) replaced by \( X_{jk}^{(q)} \one[|X_{jk}^{(q)}| \le D n^{1/2-\phi'} ]\), where \(0 < \phi' < \phi\). Applying Markov's inequality we obtain
\begin{eqnarray*}
\Pb ( \X^{(q)} \neq \widehat \X^{(q)} ) \le \sum_{j,k =1}^n \Pb (|X_{jk}^{(q)}| \geq D n^{1/2-\phi'}) \le n^{-c(\delta)}. 
\end{eqnarray*}
This inequality implies the statement of Corollary~\ref{cor 2.4}.
\end{proof}

\begin{remark} 
It still remains one challenging open problem, namely extending the bounds to weaken the moment condition to \(\delta = 0\). Furthermore, it is not unlikely  that the power of the logarithmic factor in the upper bound for \( q(n) \) may be reduced.

It seems that the bound \(n^{-c(\delta)}\) of Corollary~\ref{cor 2.4} can not be improved in general. The main difficulty here is to estimate the least singular value, see~\eqref{smallest sing value}. The required bound should be faster than any polynomial. The proof of such bound is based on the result~\cite{RudVesh2008} and requires to control the largest singular value, see~\eqref{omega n def}. Unfortunately, this requires to assume high finite moments of matrix entries. Another way is to assume that the matrix entries have absolutely continuous and bounded densities, see~\cite{Erdos2017}.    	
	
It is possible to consider the case when \(z\) is near the edge of the unit circle and extend the results~\cite{Bourgade2014b}, \cite{Yin2014c}, but this topic leaves the scope of the current paper.  
\end{remark}

We finish this section comparing our result with~\cite{Bourgade2014a}  in the case \( m =1\) and~\cite{nemish2017} for \( m > 1\).  In these papers the authors assume instead of   condition \( (3)\) in \(\Cond\)  that the uniform sub-exponential decay condition is satisfied:
\begin{eqnarray*}
\exists\, \theta > 0: \, \max_{1 \le q \le m} \max_{1 \le j,k \le n} \Pb( |X_{jk}^{(q)}| \geq t)  \le \theta^{-1} e^{-t^\theta}. 
\end{eqnarray*}
They also extended the latter to the case of finite moments of all orders. Another difference is in the upper bound for \( q(n)\) in~\eqref{eq: main result 0}. It was proved that \( q(n) \le n^\ve \) for any small \(\ve > 0 \).  In the case \( m = 1\) in  the paper~\cite{Erdos2017} conditions \((1)\) and \((2)\) were replaced by assumption that \(X_{jk}^{(1)}\) may be non-i.i.d. and \( c_1  \le \E |X_{jk}^{(1)}|^2 \le c_2  \) for some \(c_1, c_2 > 0\), but one needs to assume that \( \E |X_{jk}^{(q)}|^l \le \mu_l < \infty \) for all \(l \in \mathbb N\) and \( X_{jk}^{(q)}\) have bounded density.

\section{Proof of Theorem~\ref{th:main res}}
\subsection{Linearization} We linearise the problem considering the following block matrix (see e.g.~\cite{Burda2010}): 
\begin{eqnarray*}
\W \eqdef \frac{1}{\sqrt n} \begin{bmatrix}
\OO& \X^{(1)} & \OO & \ldots & \OO\\
\OO & \OO &  \X^{(2)} & \ldots & \OO \\
\OO & \OO & \OO & \ldots & \X^{(m-1)} \\
\X^{(m)} &  \OO &  \OO & \ldots & \OO
\end{bmatrix}.
\end{eqnarray*}
It is straightforward to check that the eigenvalues of \( \W^m \) are \(\lambda_1(\X), \ldots, \lambda_n(\X)\) with multiplicity \(m\). Hence, the following identity holds: 
\begin{eqnarray}\label{linearization}
\frac{1}{n} \sum_{j=1}^n f_{z_0}(\lambda_j(\X)) - \int f_{z_0}(z) \mu^{(m)}(dz) &=& \frac{1}{nm} \sum_{j=1}^{nm} f_{z_0}(\lambda_j^m (\W)) - \int f_{z_0}(z^m) d\mu^{(1)}(z) \nonumber \\
&=& \frac{1}{nm} \sum_{j=1}^{nm} \widetilde f(\lambda_j(\W)) - \int \widetilde f(z) d\mu^{(1)}(z),
\end{eqnarray}
where \( \widetilde f(z) \eqdef f_{z_0}(z^m)\). 

Let us consider a r.v. \(\zeta\) uniformly distributed in the unit circle and independent of all other r.v. Then for any \(r>0\) the eigenvalues of \( \W - r \zeta \II \) are
\begin{eqnarray*}
\lambda_j(\W) - r \zeta. 
\end{eqnarray*}
We denote the counting measure of \(  \lambda_j(\W) - r \zeta, j = 1, \ldots, nm\), by \(\mu_n(r,\cdot)\). It follows that \( \mu_n^{(0)} = \mu_n\).  Since \(\| f'\| \le n^C\) we get the following bound
\begin{eqnarray*}
&&\frac{1}{n} \sum_{j=1}^n f_{z_0}(\lambda_j(\X)) - \int f_{z_0}(z) \mu^{(m)}(dz) = \\
&&\qquad\qquad\qquad= \frac{1}{nm} \sum_{j=1}^{nm} \widetilde f(\lambda_j(\W)-r\zeta) - \int \widetilde f(z) d\mu^{(1)}(z) + R_n(r),
\end{eqnarray*}
where \( |R_n(r)| \le r  n^C\). Choosing \(r\) small enough the term \(R_n(r)\) will be negligible. In what follows we assume that \( r \eqdef n^{-c\log n}\).  

Together with the eigenvalues of \(\W - r \zeta \II \) we will be also interested as well in the singular values of shifted matrices \(\W(z, r)\eqdef \W -r \zeta \II - z \II \), \(z \in \C\). Let \(s_j(z,r) \eqdef s_j(\W(z,r)), j = 1, \ldots, nm\),  be the singular values of \(\W(z,r)\) arranged in the non-increasing order, i.e. 
\begin{eqnarray*}
	s_1(z, r) \geq s_2(z, r) \geq \ldots \geq s_{nm}(z, r).
\end{eqnarray*}
We shall consider as well the following matrix
\begin{eqnarray*}
\V(z,r)\eqdef\begin{bmatrix}
\OO             & \W(z,r) \\
\W^{*}(z, r) & \OO
\end{bmatrix}.
\end{eqnarray*}
It is easy to check that \(\pm s_j(z,r), j = 1, \ldots, nm\), are the eigenvalues of \(\V(z,r)\). Introduce the empirical spectral distribution (ESD) of \(\V(z,r)\): 
\begin{eqnarray*}\label{F_n def lin} 
F_n(z, x, r) \eqdef \frac{1}{2nm} \sum_{j=1}^{nm} \one[s_j(z,r) \le x] + \frac{1}{2nm} \sum_{j=1}^{nm} \one[-s_j(z,r) \le x]. 
\end{eqnarray*}

\subsection{The logarithmic potential approach}
A common tool to deal with non-Hermitian random matrices is the logarithmic potential, which is defined as follows. Let \(\nu\) be an arbitrary (probability) measure on \(\C \). Then the logarithmic potential of \(\nu\) is given by
\begin{eqnarray*}
	U_{\nu}(z)\eqdef  -\int_{\C} \log|z - w| d\nu(w).
\end{eqnarray*}
For any \( f \in C_0^2 (\C ) \) we have
\begin{eqnarray}\label{green}
\int f(z) d\nu(z) = \frac{1}{2\pi} \int \Delta f(z) U_\nu (z) \, dA(z).
\end{eqnarray}
Applying~\eqref{green}, we obtain
\begin{eqnarray*}
\frac{1}{n} \sum_{j=1}^n f_{z_0}(\lambda_j(\X)) - \int f_{z_0}(z) d\mu^{(m)}(z) = \frac{1}{2\pi} \int \Delta \widetilde f(z) [U_n^{(r)}(z) - U_{\mu^{(1)}}(z)]\, dA(z) + R_n(r),
\end{eqnarray*} 
where \(U_n^{(r)}, U_{\mu^{(1)}}\) are the logarithmic potentials of \( \mu_n^{(r)}, \mu^{(1)}\) respectively. 
 
We observe that \(\widetilde f(z) = 0\) for all  \(z \in \mathcal M \eqdef \{ z: |z^m - z_0| < C n^{-a} \} \). 
For any \(z \in \mathcal M\) we introduce the following event
\begin{eqnarray}\label{omega n def}
\Omega_n \eqdef  \Omega_n(z) \eqdef \{\omega \in \Omega: s_{nm}(z,r) \geq n^{-c \log n}, \| \W \| \le K \}
\end{eqnarray}
for some large \(K\). It follows from~\cite{RudVesh2008} (see also~\cite{GotTikh2011}[Lemma 5.1], \cite{RS2011}[Theorem 31]) and~\cite{Vu2007} (see also~\cite{GotzeNauTikh2015b}[Lemma A.1]) that 
\begin{eqnarray}\label{smallest sing value}
\Pb(\Omega_n^c) \le n^{-c\log n}.
\end{eqnarray} 
We rewrite \(U_n^{(r)}(z)\) as follows
\begin{eqnarray*}
U_n^{(r)} &=& - \frac{1}{nm}\sum_{j=1}^{nm} \log |\lambda_j(\W) - r \zeta - z| \one[\Omega_n] -  \frac{1}{nm} \sum_{j=1}^{nm} \log |\lambda_j(\W) - r \zeta - z| \one[\Omega_n^c] \\
&\eqdef& \overline U_n^{(r)}(z) + \widehat U_n^{(r)}(z). 
\end{eqnarray*}
Let us investigate the difference \(  \overline U_n^{(r)}(z) - U_{\mu^{(1)}}(z) \). Following Girko~\cite{Girko1984} we use his {\it hermitization trick} and  rewrite \(U_{n}^{(r)}\) as the logarithmic moment of \(F_n(z,x,r)\): 
\begin{align*}
U_{n}^{(r)}(z) = - \frac{1}{nm}\log |\det \W(z,r)|  = - \frac{1}{2nm} \log |\det \V(z,r)| =  - \int_{-\infty}^\infty \log |x| \, d  F_n(z, x,r).
\end{align*} 
Moreover, it was proved in~\cite{GotTikh2010} that there exists a distribution function \(G(z,x)\) such that
\begin{eqnarray*}
U_{\mu^{(1)}} = - \int_{-\infty}^\infty \log |x| dG(z,x).
\end{eqnarray*}
These equations imply
\begin{eqnarray}\label{bound 10}
|\overline U_n^{(r)}(z) - U_{\mu^{(1)}}(z) | \le \mathcal I_1 + \mathcal I_2 + \mathcal I_3,
\end{eqnarray}
where
\begin{eqnarray*}
\mathcal I_1 &\eqdef& \bigg|\int_{|x| \le n^{-c \log n}} \log |x| \, d G(z,x) \bigg|, \quad \mathcal I_2 \eqdef \bigg|\int_{n^{-c\log n} \le |x| \le K} \log |x| \, d (F_n(z, x,r) - G(z,x,r) \bigg|, \\
\mathcal I_3 &\eqdef& \bigg|\int_{|x| \geq K} \log |x| \, d G(z,x) \bigg|.	
\end{eqnarray*}
We recall some properties of the limiting distribution and introduce additional notations.  Let us denote 
\(\alpha \eqdef \sqrt{1+8|z|^2}\). Define 
\begin{eqnarray*}
	w^2_{1,2} \eqdef \frac{(\alpha\pm3)^3}{8(\alpha\pm1)},
\end{eqnarray*}
and \(\lambda_{+} \eqdef |w_1|, \lambda_{-}\eqdef |w_2|\). Moreover, let
\begin{eqnarray}\label{G support}
\mathbb J(z) \eqdef 
\begin{cases}
x \in \R: x \in [-\lambda_{+}, -\lambda_{-}] \cup [\lambda_{-}, \lambda_{+}], & \text{ if } |z| > 1, \\
x \in \R: x \in [-\lambda_{+}, \lambda_{+}],  & \text{ if } |z| < 1.
\end{cases}
\end{eqnarray}
It is known (e.g.~\cite{GotTikh2010}) that \( \mathbb J(z)\) is the support of \(G(z,x)\).  Moreover, \(G(z,x)\) has an absolutely continues symmetric density \(g(z,x)\), which is bounded and at the endpoints \( \pm \lambda_\pm \) of the support \(\J(z)\) it behaves as follows \(g(z, x)\sim \sqrt{\gamma(x)}\), where
\begin{eqnarray*}
	\gamma(u) \eqdef 
	\begin{cases}
		\min(||u|- \lambda_{+}|, ||u| - \lambda_{-}|), & \text{ if } |z| > 1, \\
		||u|-\lambda_{+}|, & \text{ if } |z| < 1.
	\end{cases}
\end{eqnarray*}
Returning to \( \mathcal I_1 \) and \(\mathcal I_3\) we may conclude that
\begin{eqnarray}\label{bound 11}
	\mathcal I_1 \lesssim n^{-1} \quad \text{ and } \quad \mathcal I_3 = 0.
\end{eqnarray}
Let us  consider the second term \(\mathcal I_2\).  Applying integration by parts we obtain
\begin{eqnarray}\label{bound 12}
	\mathcal I_2 \lesssim \Delta_n^{*}(z,r) \, \log^2 n,
\end{eqnarray} 
where \(\Delta_n^{*}(z,r) \eqdef \sup_{x \in \R } |F_n(z,x,r) - G(z,x)|\).  It is easy to check that 
\begin{eqnarray*}
\Delta_n^{*}(z,r) \le \Delta_n^{*}(z,0) + C r. 
\end{eqnarray*}
We proceed by application of the smoothing inequality of Corollary~\ref{smoothing1}.  Let us denote \(\mathbb J_\ve \eqdef \{ x \in \mathbb J: \gamma(x) \geq \ve \} \) and introduce the following region in \( \C_{+}\): 
\begin{eqnarray}\label{D def}
\mathcal D(z) \eqdef \{w = u + i v \in \C^{+}: u \in \mathbb J_{\ve /2}(z), v_0/\sqrt{\gamma(u)} \le v \le V\}, 
\end{eqnarray} 
where 
\begin{eqnarray}\label{v_0}
v_0 \eqdef A_0 \, n^{-1}\log^2 n
\end{eqnarray} 
and \( V \geq 1, A_0 > 0 \) are some constants defined later in section~\ref{sec: general principe}. Denote the  Stieltjes transform of \(F_n(z,x) \eqdef F_n(z,x,0)\) by \(m_n(z,w)\).  It is known that under conditions of Theorem~\ref{th:main res} the Stieltjes transform \(m_n(z,w)\) converges a.s. to the Stieltjes transform \(s(z,w)\), which is a solution of the following cubic equation
\begin{eqnarray}\label{eq: s(z,w) equation}
s(z,w)=-\frac{w+s(z,w)}{(w+s(z,w))^2-|z|^2},
\end{eqnarray}
see, for instance,~\cite{GotTikh2010}. Moreover, \(s(z,w)\) is the Stieltjes transform of the distribution function \(G(z,x)\).  For  detailed properties of \(s(z,w)\) we refer to~\cite{Bourgade2014a}[Lemma 4.1, 4.2]. Let us denote 
\begin{eqnarray*}
	\Lambda_n(z, u+iv) \eqdef m_n(z, u + i v) - s(z, u + iv), 
\end{eqnarray*} 
We may conclude from Theorem~\ref{lem: all v} below that  there exists \( C > 0\) such that
\begin{eqnarray}\label{lambda bound}
\Pb \bigg( \bigcap_{z \in \mathcal M} \bigcap_{w \in \mathcal D} \bigg \{|\Lambda_n(z,w)| \le \frac{C \log^2 n}{nv} \bigg \} \bigg) \geq 1 - n^{-Q}.
\end{eqnarray} 
Applying the smoothing inequality, Corollary~\ref{smoothing1}, to \( \Delta_n^{*} \) we get the following bound
\begin{eqnarray}\label{eq: smoothing inequality 0}
\Delta_n^{*}(z,0) \le C_1\int_{-\infty}^\infty |\Lambda_n(z, u + i V)|\, du  + C_2 \sup_{x \in \mathbb J_{\varepsilon/2}} \bigg|  \int_{v'}^V \Lambda_n(z, x + i v) \, dv \bigg| +  C_3 \, v +  C_4 \, \varepsilon^\frac32,
\end{eqnarray}
\( v' = v_0/\sqrt{\gamma(x)}\). The proof of this inequality repeats the proof of its analogue in the case of the semi-circular law (see ~\cite{GotTikh2003}[Corollary~2.3]).  For the readers convenience we include the arguments in the appendix.  Let us take in this inequality \(\varepsilon \eqdef (2 v_0 a)^{2/3}\). Then \(C_3 v_0 + C_4 \varepsilon^{3/2} \le C n^{-1} \log^2 n\). 
It follows from~\eqref{bound 10},~\eqref{bound 11},~\eqref{bound 12} and~\eqref{eq: smoothing inequality 0} that 
\begin{eqnarray}\label{bound 13}
	\int |\Delta \widetilde f(z)| |\overline U_n^r(z) - U_\mu(z)|\, dA(z) &\le&  C_1 \log^2 n \int |\Delta \widetilde f(z)| \int_{-\infty}^\infty |\Lambda_n(z, u + i V)|\, du \, dA(z) \nonumber \\
	&+& C_2 \log^2 n \int |\Delta \widetilde f(z)|  \sup_{x \in \mathbb J_{\varepsilon/2}} \bigg|  \int_{v'}^V \Lambda_n(z, x + i v) \, dv \bigg|  \, dA(z) \nonumber \\
	&+& C_3 n^{-1} \log^4 n.
\end{eqnarray}
Inequality~\eqref{lambda bound} implies that with probability at least \(1 - n^{-Q} \)
\begin{eqnarray*}\label{eq: second integral 0}
\sup_{z \in \mathcal M} \sup_{x \in \mathbb J_{\varepsilon/2}} \bigg| \int_{v'}^V \Lambda_n(z, x + i v) \, dv \bigg| \lesssim   n^{-1} \log^3 n.
\end{eqnarray*}
Hence, 
\begin{eqnarray}\label{bound 14}
\int |\Delta \widetilde f(z)|  \sup_{x \in \mathbb J_{\varepsilon/2}} \bigg|  \int_{v'}^V \Lambda_n(z, x + i v) \, dv \bigg|  \, dA(z) \lesssim \|\Delta \widetilde f\|_{L_1}\, n^{-1} \log^3 n
\end{eqnarray}
with probability at least \(1 - n^{-Q} \).
We conclude from Lemma~\ref{lem: large v values} that
\begin{eqnarray*}\label{eq: estimate 2 assumption}
\E^\frac{1}{p}|\Lambda_n(z, u + i V)|^p \le \frac{Cp |s(z, u+iV)|^\frac{p+1}{p}}{n},
\end{eqnarray*}
which holds for all \(w = u + i V, u \in \R\). Hence,
\begin{eqnarray*}
	\E^\frac1p \bigg[ \int_{-\infty}^\infty |\Lambda_n(z, u + i V)|\, du \bigg ]^p &\le&
	\int_{-\infty}^\infty \E^\frac{1}{p}|\Lambda_n(z, u + i V)|^p\, du \\
	&\le& \frac{C p}{n} \int_{-\infty}^\infty \int_{-\infty}^\infty \frac{d u \,d G(z, x)}{((x-u)^2 + V^2)^\frac{p+1}{p}} \lesssim n^{-1} \log^2n.
\end{eqnarray*}
It is straightforward to check that
\begin{eqnarray*}
\E \bigg[ \int |\Delta \widetilde f(z)| \int_{-\infty}^\infty |\Lambda_n(z, u + i V)|\, du \, dA(z) \bigg]^p &\le& \|\Delta \widetilde f\|_{L_1}^p  \sup_{z \in \mathcal M} \E \bigg[ \int_{-\infty}^\infty |\Lambda_n(z, u + i V)|\, du \bigg ]^p \nonumber \\
&\le & \|\Delta \widetilde f\|_{L_1}^p  \, n^{-p} \log^{2p} n.
\end{eqnarray*}
Markov's inequality implies that with probability at least \(1 - n^{-Q}\)
\begin{eqnarray}\label{bound 15}
\int |\Delta \widetilde f(z)| \int_{-\infty}^\infty |\Lambda_n(z, u + i V)|\, du \, dA(z) \lesssim \|\Delta \widetilde f\|_{L_1} \, n^{-1} \log^2 n.
\end{eqnarray}
Combining now~\eqref{bound 13},~\eqref{bound 14} and \eqref{bound 15} we conclude that with probability at least \(1 - n^{-Q}\)
\begin{eqnarray}\label{bound 20}
\int |\Delta \widetilde f(z)| |\overline U_n^{(r)}(z) - U_{\mu^{(1)}}(z)|\, dA(z)  \lesssim \|\Delta \widetilde f\|_{L_1}  \, n^{-1} \log^5 n.
\end{eqnarray}
It remains to estimate
\begin{eqnarray*}
\int |\Delta \widetilde f(z)| |\widehat U_n^{(r)}(z)|\, dA(z).
\end{eqnarray*}
Let us consider \(\widehat U_n^{(r)}(z) \). We get
\begin{eqnarray*}
\E |\widehat U_n^{(r)}(z)|^p &\le&   \frac{1}{nm}\sum_{j=1}^{nm} \E^\frac12 \bigg[ \log^{2p} | \lambda_j - r \zeta - z |\bigg]\, \Pb^\frac12(\Omega_n^c).
\end{eqnarray*}
We fix \(j = 1, \ldots, nm\) and write
\begin{eqnarray*}
\frac{1}{2\pi} \int_{|\zeta|\le 1}\log^{2p} | \lambda_j - r \zeta - z |\, d\zeta  \le J_1 + J_2 + J_3,
\end{eqnarray*}
where 
\begin{eqnarray*}
J_1 & = &\frac{1}{2\pi}  \int_{|\zeta| \le 1, |\lambda_j - r \zeta - z | \le \ve} \log^{2p} | \lambda_j - r \zeta - z | \, d \zeta, \\
J_2 & = & \frac{1}{2\pi} \int_{|\zeta| \le 1, \ve < |\lambda_j - r \zeta - z | \le 1/\ve} \log^{2p} | \lambda_j - r \zeta - z | \, d \zeta, \\
J_3 & = & \frac{1}{2\pi} \int_{|\zeta| \le 1, |\lambda_j - r \zeta - z | \geq 1/\ve} \log^{2p} | \lambda_j - r \zeta - z | \, d \zeta,
\end{eqnarray*}
It is easy to see that
\begin{eqnarray*}
J_2 \le \log^p (1/\ve). 
\end{eqnarray*}
To estimate \( J_1\) we first note that for any \( b > 0\), the function \( - u^b \log u \) is not decreasing on the interval \(0 < u < e^{-1/b}\). Hence, for any \( 0 < u \le \ve < e^{-1/b}\) we obtain
\begin{eqnarray*}
- \log u \le \ve^b u^{-b} \log (1/\ve).
\end{eqnarray*} 
We take \(b\) such that \(b p = 1\). Then
\begin{eqnarray*}
J_1 \le \frac{1}{2\pi r^2} \ve^{bp}  \log^p (1/\ve) \int_{|\zeta| \le \ve} |\zeta|^{-bp} \, d\zeta \le \log^p (1/\ve) \ve^2 r^{-2}. 
\end{eqnarray*}
Choosing \( \ve = r \) we arrive at the inequality
\begin{eqnarray*}
J_1 \le \log^p (1/\ve). 
\end{eqnarray*}
It remains to estimate \(J_3\). It is straightforward to check that \( \log^{2p} u \le \ve^{2} u^{2} \log^{2p} \ve  \) for \( u \geq 1/\ve\) and \( p\) of order \( \log n\) (we recall that \(\ve = n^{-c \log n}\)).  Hence,
\begin{eqnarray*}
\frac{1}{nm} \sum_{j=1}^{nm} \E J_3 \le  m \,n \, r^{2} (2 + |z|^2) \log^{2p}\ve. 
\end{eqnarray*}
These bounds together imply that for \( p\) of order \( \log n\)
\begin{eqnarray*}
\sup_{z \in \mathcal M} \E |\widehat U_n^{(r)}(z)|^p \le n^{-c\log n}.
\end{eqnarray*}
Repeating the same arguments as in the proof of~\eqref{bound 15} we conclude the estimate 
\begin{eqnarray}\label{bound 21}
\int |\Delta \widetilde f(z)| |\widehat U_n^{(r)}(z)|\, dA(z) \lesssim \frac{\| \Delta \widetilde f \|_{L_1}}{n},
\end{eqnarray}
which holds with probability at least \(1 - n^{-Q}\). Combining~\eqref{bound 20} and~\eqref{bound 21} we come to the following bound
\begin{eqnarray*}
\bigg | \frac{1}{2\pi} \int \Delta \widetilde f(z) [U_n^{(r)}(z) - U_{\mu^{(1)}}(z)]\, dA(z) \bigg| \lesssim \frac{q(n) \|\Delta f\|_{L_1}}{n^{1-2a}}, 
\end{eqnarray*}
which holds with probability at least \(1 - n^{-Q}\). The last inequality implies the claim of Theorem~\ref{th:main res}.

\section{Local law for shifted matrices} 

The following theorem provides the estimate for \(\Lambda_n(z, w) \) up to the optimal scale \(v_0\) (see definition~\eqref{v_0}).  The proof of this result will be given later on in section. 
\begin{theorem}[Local law for eigenvalues of \(\V(z)\)]\label{lem: all v}
	Assume that \(\Cond\) hold. Let \( Q > 0\) be an arbitrary number. There exists constant \( C > 0\) such that
	\begin{eqnarray*}
	\Pb \bigg( \bigcap_{z \in \mathcal M} \bigcap_{w \in \mathcal D} \bigg \{|\Lambda_n(z,w)| \le \frac{C \log^2 n}{nv} \bigg \} \bigg) \geq 1 - n^{-Q}.
	\end{eqnarray*} 
\end{theorem}

By standard truncation arguments (see~\cite{GotzeNauTikh2016a}[Lemmas D.1-D.3]) in what follows we may assume that  conditions \(\CondTwo\) hold and  \( a_{jk}^{(q)} = 0\), for all \(j,k = 1, \ldots, n, q = 1, \ldots, m\). For simplicity we will also assume that \( X_{jk}^{(q)}, j,k=1, \ldots, n, q = 1, \ldots, m \) are i.i.d. r.v. In this case one may also show (see~\cite{GotzeNauTikh2016a}[Lemmas D.1-D.3]) that it is possible to assume that \( [\sigma_{jk}^{(q)}]^2 = 1 \) for all \(j,k=1, \ldots, n, q = 1, \ldots, m\). The proof in the non i.i.d. case in the same. One needs to add additional \(\varepsilon_j\) term in~\eqref{rj} which will be small due to the assumption that \( |1 - [\sigma_{jk}^{(q)}]^2| \le n^{-1 - \ve}\).
 
\subsection{Bound for the distance between Stieltjes transforms}
We start from the general lemma, which is motivated by the additive descent approach introduced and further developed by  L.~Erd{\"o}s, B.~Schlein, H.-T.~Yau and et al., see \cite{ErdosSchleinYau2009}, \cite{ErdosSchleinYau2009b}, \cite{ErdosSchleinYau2010}, \cite{ErdKnowYauYin2012},~\cite{ErdKnowYauYin2013}, \cite{ErdKnowYauYin2013a}, \cite{LeeYin2014} among others.  
Recall that
\begin{eqnarray}\label{eq: lambda_n def}
\Lambda_n(z,w) \eqdef m_n(z,w) - s(z,w).
\end{eqnarray}
For  \(w=u+iv \in \C^{+}\) we define
\begin{eqnarray*}
\RR(z,w) \eqdef (\V(z) - w \II)^{-1}.
\end{eqnarray*} 
It is easy to see that \( m_n(z,w) = \frac{1}{2nm} \Tr \RR(z,w) \).  Denote \(j_\alpha \eqdef (\alpha-1)n + j\). Introduce the following partial traces of resolvent  \( m_n^{(\alpha)} (z,w) \eqdef \frac{1}{n} \sum_{j=1}^n \RR_{j_\alpha j_\alpha} \) and 
\begin{eqnarray*}
	\Lam \eqdef ( \Lambda_n^{(1)}, \ldots, \Lambda_n^{(2m)})^\mathsf{T}, \quad \Lambda_n^{(\alpha)} = m_n^{(\alpha)}(z,w) - s(z,w).
\end{eqnarray*}
It is easy to check that
\begin{eqnarray*}
\Lambda_n = \frac{1}{m} \sum_{\alpha = 1}^{m} \Lambda_n^{(\alpha)} = \frac{1}{m} \sum_{\alpha=1}^{m} \Lambda_n^{(m+\alpha)}.
\end{eqnarray*}
Moreover, 
\begin{eqnarray*}
|\Lambda_n(z,w)| \le | \Lam(z,w) |.
\end{eqnarray*}
Let \( w = u + i v \in \mathbb D \)
\begin{eqnarray}\label{I(v) def 0}
I(v) \eqdef I_{\tau}(z, u + i v) \eqdef  \prod_{k=0}^{K_v} \one\big[|\Lam_n(z, u + i v s^k )| \le \tau \imag s(z,u+ i v s^k) \big],
\end{eqnarray}
where \(  K_v \eqdef \min\{l: v s^l \geq V \}\) and \(s \geq 1\). The exact value of \(s\) will be defined later in section~\ref{sec: general principe}. Let \(C\) be a positive constant. We take  \(\tau, A_1\) sufficiently small and  \(A_0\) sufficiently large such that
\begin{eqnarray}
\frac{C p}{nv} \le \tau \imag s(z, w)
\end{eqnarray}
for any \(w = u + i v \in \mathcal D\) and \(1 \le p \le A_1 \log n\). The exact values of \(\tau, A_0, A_1\) will be defined later in section~\ref{sec: general principe}. 

The next lemma is crucial for the proof of Theorem~\ref{lem: all v}. 
\begin{lemma}\label{discent}
Let \( w \in \mathcal D\)  and \(\tau\) be some fixed number.  Assume that for all \( v \geq v_0/\sqrt{\gamma(u)} \) and \(1 \le p \le A_1 \log n \)  
	\begin{eqnarray}\label{cond 1}
	\E [|\Lam_n(z,u+iv) |^p I(v)] \le \frac{C^p p^{2p} }{(nv)^p}, 
	\end{eqnarray}
	and 
	\begin{eqnarray}\label{cond 2}
	\Pb \bigg( |\Lam_n(z,u+iV)| \geq \tau \imag s(z, u + i V)  \bigg) \le \frac{C}{n^Q}, 
	\end{eqnarray}
	Then for any \( v_0/\sqrt{\gamma(u)} \le v \le V\) 
	\begin{eqnarray}\label{result 1}
	\Pb \bigg( |\Lam_n(z,u+iv)| \geq \tau \imag s(z, u + i v) \bigg) \le \frac{C}{n^Q}.
	\end{eqnarray}	
\end{lemma}
\begin{proof}
Let \( \kappa = \kappa_n \) be such that
\begin{eqnarray}\label{eq: kappa change}
| \Lam_n(z,u + i v) - \Lam_n(z,u+ i (v + \kappa)) | \le \frac{\tau}{2}  \imag s(z, u+i v)  .
\end{eqnarray}
It easy to check that one may take, for example,  \( \kappa_n = n^{-3}  \).  Denote \(v' = v_0/\sqrt{\gamma(u)}\). We split \([v', V]\) into \(N = (V-v')/\kappa\) intervals and denote \(v_k = v' + k \kappa\). Assume that we have already proved~\eqref{result 1} for all \(v_{k} \le v \le V\) and prove it for any \(v\) up to \(v_{k-1}\). For example, for \(v = v_N = V\) it follows from~\eqref{cond 2}. 
	
We fix \(v: v_{k-1} \le v < v_{k}\). Taking \( p = A_1 \log n\) and \( K: K^{-p} \le C n^{-Q} \) we get
\begin{eqnarray*}
\Pb \bigg ( |\Lam_n(z, u+i v_{k})| \geq \frac{K C p^2}{nv_{k}} \bigg ) &\le& \E  \one \bigg [ |\Lam_n(z, u+iv_k))| \geq \frac{K C p^2}{nv_k} \bigg ] I(v_k)  \\
&+&\sum_{l=0}^{K_v} \Pb \bigg ( |\Lam_n(z, u + i v_k s^l)| \geq  \tau \imag s(z, u+ i v_k s^l)   \bigg )\\
&\le& \left( \frac{C  K p^2 }{nv_k} \right)^{-p} \E [|\Lam_n(z,u+iv_k)|^p I(v_k) ] + \frac{C}{n^Q}  \le  \frac{C}{n^Q}.
\end{eqnarray*}
Here we also used~\eqref{cond 1}. 
Since \( v_k \geq v \geq v' \) we get that \( \frac{K C p^2}{nv_k} \le \frac{\tau}{2} \imag s(z, u+ i v) \). Hence, using~\eqref{eq: kappa change} we obtain 
\begin{eqnarray*}
\Pb \bigg( |\Lam_n(z, u + i v)| \geq \tau \imag s(z, u + iv) \bigg) \le \frac{C}{n^Q}.
\end{eqnarray*} 
\end{proof}
It follows from Lemma that we need to check conditions~\eqref{cond 1}--\eqref{cond 2}.

\subsection{Stieltjes transform and self-consistent equations}
In this section we investigate  \(m_n(z, w)\) and show that it satisfies a cubic equation (see~\eqref{mn1} below), which is a perturbation of the corresponding equation \eqref{eq: s(z,w) equation} for \(s(z,w)\).

Let \( \RR^{(j_\alpha)}\) ( resp. \(\RR^{(j_\alpha, j_\alpha)}\) ) be the resolvent matrix of \(\V(z)\) with \(j_\alpha\)-th row and column deleted (resp. \(j_\alpha\)- and \(j_{m+\alpha}\)-th row and column deleted).  Applying Schur's inverse formula, we may write, for all \(j  = 1, \ldots, n\) and \(\alpha = 1, \ldots, m\), that 
\begin{eqnarray}\label{rj}
\RR_{j_\alpha j_\alpha}= -\frac{1}{-w-m_n^{([\alpha+1]+m)}+ \frac{|z|^2}{w + m_n^{([\alpha-1])}(z,w)} }(1 - \ve_{j_\alpha}\RR_{j_\alpha j_\alpha}),
\end{eqnarray}
where \( m_n^{(\alpha)} (z,w) \eqdef \frac{1}{n} \sum_{j=1}^n \RR_{j_\alpha j_\alpha} \) and 
\begin{eqnarray*}
\ve_{j_\alpha} \eqdef \widetilde \ve_{j_\alpha} + \frac{|z|^2}{w + m_n^{([\alpha-1])}(z,w)} \RR_{j_{\alpha +m}, j_{\alpha + m}}^{(j_\alpha)} \widehat \ve_{j_\alpha}. 
\end{eqnarray*}
Here \(\widetilde \ve_{j_\alpha} \eqdef \widetilde \ve_{j1}+\ldots+ \widetilde \ve_{j4}\),
\begin{eqnarray*}
\widetilde \ve_{j_\alpha, 1}& \eqdef& \frac1n \sum_{k\in \T} \RR_{k_{[\alpha +1 ] }, k_{[\alpha +1] + m }}-\frac1n \sum_{k \in \T} \RR_{k_{[\alpha +1 ] + m }, k_{[\alpha + 1 ] + m }}^{(j_\alpha) } ,\\
\widetilde \ve_{j_\alpha, 2}&\eqdef& -\frac1n\sum_{l\ne k \in \T} X_{jk}^{(\alpha)} X_{jl}^{(\alpha)}  \RR_{l_{[\alpha +1 ] + m }, k_{[\alpha +1 ] + m }}^{(j_\alpha) }, \\
\widetilde \ve_{j_\alpha, 3}& \eqdef&-\frac1n\sum_{l \in \T } ([X_{jl}^{(\alpha)}]^2-1) \RR_{l_{[\alpha +1 ] + m }, l_{[\alpha +1] + m }}^{(j_\alpha) }, \\
\widetilde \ve_{j_\alpha, 4}&\eqdef& \frac{z+\overline z}{\sqrt n}\sum_{l \in \T } X_{jl}^{(\alpha)} \RR_{j_{[\alpha +1 ] + m },l_{[\alpha +1] + m }}^{(j_\alpha) }.
\end{eqnarray*}
and \(\widehat \ve_{j} \eqdef \widehat \ve_{j,1}+\ldots+\widehat \ve_{j,3}\), where 
\begin{eqnarray*}
\widehat \ve_{j_\alpha,1}&\eqdef& \frac1n \sum_{l \in \T} \RR_{l_{[\alpha-1] }, l_{[\alpha-1]}} - \frac1n \sum_{l \in \T} \RR_{l_{[\alpha-1] }, l_{[\alpha-1]}}^{(j_\alpha, j_\alpha)}, \\
\widehat \ve_{j_\alpha,2} &\eqdef&-\frac1n\sum_{k\ne l\in\T}X_{kj}^{([\alpha-1])}X_{lj}^{([\alpha-1])}\RR_{k_{[\alpha-1] }, l_{[\alpha-1]}}^{(j_\alpha, j_\alpha)},\\
\widehat \ve_{j_\alpha,3} &\eqdef&-\frac1n\sum_{l\in\T}([X_{lj}^{([\alpha-1])}] ^2-1)\RR_{l_{[\alpha-1] }, l_{[\alpha-1]}}^{(j_\alpha, j_\alpha)}.
\end{eqnarray*}
Summing up equality~\eqref{rj} in \(j=1,\ldots,n\) for fixed \( \alpha = 1, \ldots, m\) we get
\begin{eqnarray}\label{mn1}
m_n^{(\alpha)}(w,z)&=&-\frac{1}{w+ m_n^{([\alpha+1]+m)}(z,w)-\frac{|z|^2}{w + m_n^{([\alpha-1])}(z,w) }}(1-T_n^{(\alpha)}), \\
\label{mn2}
m_n^{(m+\alpha)}(w,z)&=&-\frac{1}{w+m_n^{([\alpha-1])}(z,w)-\frac{|z|^2}{w + m_n^{([\alpha+1]+m)}(z,w) }}(1-T_n^{(m+\alpha)}),
\end{eqnarray}
where
\begin{eqnarray*}
	T_n^{(\alpha)}\eqdef\frac1n\sum_{j=1}^n\ve_{j_\alpha} \RR_{j_\alpha, j_\alpha}
\end{eqnarray*}
for \(\alpha = 1, \ldots, 2m\).
It follows from the form of equations~\eqref{mn1}-\eqref{mn2} and \eqref{eq: s(z,w) equation} that to bound the distance between \(m_n^{(\alpha)}(z,w)\) and \( s(z,w)\) it is crucial to estimate the perturbation \( T_n^{(\alpha)}\). 
Introduce the following block-matrix
\begin{eqnarray}\label{bmatrix}
\A \eqdef \begin{bmatrix}
\A_{11} & \A_{12} \\
\A_{12}^\mathsf{T} & \A_{11}^\mathsf{T}
\end{bmatrix},
\end{eqnarray}
where
\begin{eqnarray*}
\A_{11} \eqdef \begin{bmatrix}
a & 0 & 0 & \ldots & 0 & b \\
b & a & 0 & \ldots & 0 & 0\\
0 & b & a & \ldots & 0 & 0\\
\ldots \\
0 & 0 & 0 & \ldots & a & 0 \\
0 & 0 & 0 & \ldots & b & a 
\end{bmatrix}, \quad 
\A_{12} \eqdef \begin{bmatrix}
0 & 1 & 0 & \ldots & 0 & 0 \\
0 & 0 & 1 & \ldots & 0 & 0\\
0 & 0 & 0 & \ldots & 0 & 0\\
\ldots \\
0 & 0 & 0 & \ldots & 0 & 1 \\
1 & 0 & 0 & \ldots & 0 & 0 
\end{bmatrix}. 
\end{eqnarray*}
Here, \( a \eqdef - s^{-2}(z,w), b \eqdef \frac{|z|^2}{(w + s(z,w))^2} \). Substituting \(s(z,w)\) from the both sides of equations~\eqref{mn1}-\eqref{mn2} we come to the following linear system: 
\begin{eqnarray}\label{linear system}
\A \Lam_n = \rr_n + s^{-1} \TTT_n,
\end{eqnarray}
where
\begin{eqnarray}\label{rr def}
\|\rr_n\| \le | \Lam_n |^2 \left( \frac{|z|^2}{|w+s|^2}   \bigg( \sum_{\alpha = 1}^{m} \frac{1}{|w+m_n^{(\alpha)}|^2 } \bigg)^\frac12 \bigg(1 + \frac{| \Lam_n|}{|s|} \bigg) +  \frac{1}{|s|}\bigg( 1 + \frac{|z|^2}{|w+s|^2}\bigg)  \right).
\end{eqnarray}
Permuting the rows and columns of \( \A \) we may come to the matrix from~\cite{nemish2017}[Equation~5.18]. 
\subsection{Validity of condition~\eqref{cond 1}}
Define
\begin{eqnarray}
\mathcal A(z,v, q) &\eqdef& \max_{(\J, \K) \in \mathcal J_{1}} \max_{\alpha = 1, \ldots, m} \max_{l_\alpha \in \T_{J_\alpha}} \E^\frac{1}{q} \imag^{q} \RR_{l_\alpha, l_\alpha}^{(\J, \K)}(z,w)I(v), \\
\label{E def}
\mathcal E(q) &\eqdef& \imag s(z,w) + \mathcal A(z,v, q)\bigg(1 + \sum_{k=0}^{K_v} \frac{1}{s^k} \frac{\mathcal A(z,v_k, q)}{\imag s(z,v_k)}\bigg).
\end{eqnarray}
See~\eqref{eq: set def} for the definitions of \( \J, \K, \mathcal J_1, \T_{J_\alpha} \). The next lemma shows that the condition~\eqref{cond 1} holds. 
\begin{lemma}\label{lem: main cond 1}
Assume that for all \(w \in \mathcal D\) and \( 1 \le p \le A_1 \log n\)
\begin{align}\label{T_n bound general}
\E [|\TTT_n(z, iv)|^p I(v)] \lesssim \frac{C^p p^{2p} \mathcal E^p(\kappa p)}{(nv)^p}
\end{align}
and
\begin{eqnarray}\label{e_p bound}
\mathcal A^p(\kappa p) \le C^p \imag^p s(z,w)
\end{eqnarray}
for some \( \kappa > 0\). Then
\begin{eqnarray*}
\E |\Lam_n(z,iv)|^p I(v) \lesssim \bigg(\frac{C p^2}{nv}\bigg)^p. 
\end{eqnarray*}
\end{lemma}
We apply Stein's method and 'leave one out' idea to estimate \(\TTT_n\). Here we follow the ideas introduced in~\cite{GotTikh2003} and further developed in~\cite{GotzeNauTikh2015a}. It is clear that the bound for \(\TTT_n\) requires estimation of the moments of \(\RR_{jj}\). We do it section~\ref{sec: general principe}, where we also introduce the general principle to estimate the moments of so-called \(k\)-descent function (see definition~\ref{def: k descent functions}). Moreover, in this section we show that~\eqref{e_p bound} holds. 
\begin{proof}[Proof of Lemma~\ref{lem: main cond 1}]
We may rewrite~\eqref{linear system} as follows
\begin{eqnarray*}
\Lam = \A^{-1} \rr + s^{-1}\A^{-1} \TTT.
\end{eqnarray*}
It follows that
\begin{eqnarray}\label{main inequality}
\| \Lam_n \| \le \| \A^{-1}\| \| \rr \| + |s|^{-1} \| \A^{-1}\| \|\TTT\|
\end{eqnarray}
We may write 
\begin{eqnarray}
|w + m_n^{(\alpha)}(z,w)| I(v) \geq (|w + s(z,w)| - |\Lambda_n^{(\alpha)}|)  I(v) \geq (1 - \tau) |w + s(z,w)| I(v)   
\end{eqnarray}
Moreover, using the definition of \( I(v)\) we obtain
\begin{eqnarray*}\label{eq: lambda squared bound}
|\Lam_n|^2 I(v) \le \tau |\Lam_n| I(v). 
\end{eqnarray*}
Taking into account the last two inequalities and definition~\eqref{rr def} of \(\rr\) we obtain 
\begin{eqnarray*}
\|\rr \|I(v) \lesssim \tau \imag s(z,w)|\Lam_n| I(v).
\end{eqnarray*}
It follows from~\cite{nemish2017}[Proposition~5.5] then \( \| \A^{-1} \| \lesssim \imag^{-1} s(z,w) \). Taking expectation of the both sides of~\eqref{main inequality} and applying~\eqref{T_n bound general}, \eqref{e_p bound} we get the claim of this lemma. 
\end{proof}

\subsection{Validity of condition~\eqref{cond 2}}

This conditions in a consequence of the  the next lemma. 
\begin{lemma}\label{lem: large v values}
For any \(w = u + i V, u \in \R, V \geq 1\) and all \(p \geq 1\) the following bound holds
\begin{eqnarray*}
\E^\frac{1}{p}|\Lambda_n(z, u + i V)|^p \le \frac{Cp |s(z, u+iV)|^\frac{p+1}{p}}{n}.
\end{eqnarray*}
\end{lemma}
\begin{proof}
The proof is similar to the proof of the analogous inequality in~\cite{GotzeNauTikh2015b}[Inequality~2.8] in the semi-circle law case.
\end{proof}
\subsection{Proof of Theorem~\ref{lem: all v}}
\begin{proof}[Proof of Theorem~\ref{lem: all v}]
The proof is the direct corollary of Lemmas~\ref{discent} and~\ref{lem: main cond 1}. Indeed, taking \(p = A_1 \log n\) we may write 	
\begin{eqnarray*}
	\Pb \bigg ( |\Lam_n(z, u+i v)| \geq \frac{K C p^2}{nv} \bigg ) &\le& \E  \one \bigg [ |\Lam_n(z, u + i v)| \geq \frac{K C p^2}{nv} \bigg ] I(v)  \\
	&+&\sum_{k=0}^{K_v} \Pb \bigg ( |\Lam_n(z, u + i v s^ k)| \geq  \tau \imag s(z, u+ i v s^k)   \bigg )\\
	&\le& \left( \frac{C  K p^2 }{nv} \right)^{-p} \E [|\Lam_n|^p I(v) ] + \frac{C}{n^Q}  \le  \frac{C}{n^Q}.
\end{eqnarray*}
This inequality implies the claim of the theorem.
\end{proof}

\section{Bound for functions with \(k\)-descent property}\label{sec: general principe}
 As it was already mentioned that the estimation of \(\TTT_n\) requires to bound the high moments of \( \RR_{jj}(z,u+iv)\) and \(\imag \RR_{jj}(z, u + iv) \) for \( j = 1, \ldots, nm\) up to the optimal value \(v_0\) of \(v\).  Here we are going to apply {\it multiplicative descent method} introduced in~\cite{Schlein2014} and further developed in the series of papers~\cite{GotzeNauTikh2015a}, \cite{GotzeNauTikh2016a} by the authors. This method requires the small number of steps, usually of the logarithmic order. One may compare with {\it additive descent method}, Lemma~\ref{discent}, where one needs to make polynomial number of steps.  
 \subsection{Class of descent function }
 We start from rather general definition and proposition which are essential for {\it multiplicative descent}. Let us introduce the following class of functions. 
 \begin{definition}\label{def: k descent functions}
 Let \(k\ge1\). We say that a function \(f(w), w = u + i v \in \C^{+}\), satisfies the \(k\)-descent property if for any \(v>0\)
 \begin{eqnarray*}
 \left|\frac{\partial }{\partial v}\log f(u+iv)\right|\le \frac kv
 \end{eqnarray*}
 \end{definition}
 
 Let us denote by \(\mathcal Descent(k) \eqdef \{f:\C^{+}\rightarrow\C:\ f\text{  satisfies }k\text{-descent property}\}\). The following statement collects the main properties of \(k\)-descent functions.
 \begin{statement}\label{lem: descent property}
 	The following statements hold:
 	\begin{enumerate}
 		\item If \(f\in \mathcal Descent(k)\) then \(f^{-1}\in\mathcal Descent(k)\);
 		\item If \(f\in \mathcal Descent(k)\), \(g\in\mathcal Descent(l)\) then \(fg\in\mathcal Descent(k+l)\).
 		\item For any \(f\in\mathcal Descent(k)\) and for any \(s\ge1\)
 		\begin{eqnarray*}
 		|f(u+iv/s)|\le s^k|f(u+iv)| \quad \text{and} \quad |f(u+iv)|\le s^k|f(u+iv/s)|
 		\end{eqnarray*}
 	\end{enumerate}
 \end{statement}
 \begin{proof}
 	The proof of \((1)\) and \((2)\) are trivial. To prove \((3)\) it is enough to mention that
 	\begin{eqnarray*}
 	|\log f(u+iv/s) - \log f(u+iv)| \le k \log s.
 	\end{eqnarray*}
 \end{proof}
 
 It is easy to check that \(|\RR_{jj}(z,w)|\), \(\imag \RR_{jj}(z,w)\) are examples of functions with \(1\)-descent property w.r.t. \(w\). 

\subsection{Bound for moments of some functions of the resolvent matrix}

Recall the definition of \(I_\tau(w)\) for \( w = u + iv \in \mathcal D\)  
\begin{eqnarray*}
I(w) \eqdef I_{\tau}(z, u + i v) \eqdef  \prod_{k=0}^{K_v} \one\big[|\Lam_n(z, u + i v s^k )| \le \tau \imag s(z,u+ i v s^k) \big],
\end{eqnarray*}
where \(  K_v \eqdef \min\{l: v s^l \geq V \}\). Here \(V \geq 1\) and \(s \geq 1\) are some constants  defined later.   It is easy to see that
\begin{eqnarray}\label{I(v) property}
I(u+i v) \le I(u+ i s_0 v). 
\end{eqnarray}
In what follows for simplicity we shall often omit \(w = u + iv\) from all notations and write only imaginary part \(v\). 
\begin{lemma}\label{main lemma 0}
Let \(V\) be some fixed number. There exist a positive constant \(C_0\) depending on \(V, z \) and positive constants \(A_0, A_1, \tau\) depending on \(C_0\) such that
\begin{eqnarray}\label{eq: resolvent bound 0}
\max_{j = 1, \ldots, nm } \E \left|\RR_{jj}(z, u + i v )\right|^p I(u+i v) &\le& C_0^p, \\
\label{eq: resolvent bound 1}
\max_{j = 1, \ldots, nm } \E \imag^p \RR_{jj}(z, u+ iv ) I(u+iv) &\le& C_0^p\imag^p s(z, u + iv).
\end{eqnarray}
for all \(u + i v \in \mathcal D\) and \(1 \le p \le A_1 \log n\).
\end{lemma}

\begin{lemma}\label{l: bound for some functions}
	Let \(V\) be some fixed number. There exist a positive constant \(H_0\) depending on \(V, \) and positive constant \(\tau\) depending on \(H_0\) such that the following inequalities hold:  
	\begin{eqnarray}
		\max \left ( \frac{1}{|w + m_n^{(\alpha)}(z,w)|},  \Big|w + m_n^{([\alpha+1]+m)}(z,w) - \frac{|z|^2}{w + m_n^{([\alpha-1])}(z,w)}\Big|^{-1}  \right ) I(v) \le H_0
	\end{eqnarray}
	for all \(w = u + i v \in \mathcal D\) and \( \alpha = 1, \ldots, m\). 
\end{lemma}

\begin{remark}
The statement of Lemma remains valid if one replaces \( m_n^{(\alpha)}(z,w) \) by \( m_n^{(\alpha, \J, \K)}(z,w) \). 
\end{remark}

\begin{proof}[Proof of Lemma] 
Assume that \( I(v) = 1 \). In the opposite case the claim is trivial. Then
\begin{eqnarray*}
\frac{1}{|w + m_n^{(\alpha)}(z,w)|} \le \frac{1}{|w + s(z,w)|} + \frac{|\Lambda_n^{(\alpha)}|}{|w + s(z,w)||w + m_n^{(\alpha)}(z,w)|} \le c + \frac{c \, \tau }{|w + m_n^{(\alpha)}(z,w)|}, 
\end{eqnarray*}	
here \( c \) depends on \( z \) and \(V\). The last inequality implies 
\begin{eqnarray*}
\frac{1}{|w + m_n^{(\alpha)}(z,w)|} \le \frac{c}{1 - c \, \tau }. 
\end{eqnarray*}	
If we take, say, \( H_0 \geq 2c \), then we may find sufficiently small \( \tau\), such that the r.h.s. of the previous inequality is bounded by \( H_0 \).   Similarly we may prove.
\end{proof}	
\begin{proof}[Proof of Lemma~\ref{main lemma 0}] 	
The proof of is more involved then the proof of the previous lemma. The general idea how to prove these results follows the idea of~\cite{Schlein2014} about multiplicative descent approach developed in~\cite{GotzeNauTikh2015a}. We briefly discuss these ideas on the bound~\eqref{eq: resolvent bound 0} for \(\E \left|\RR_{jj}(z, v)\right|^p I(v) \) (the same will be true for~\eqref{eq: resolvent bound 1}).

We prove below in Lemma~\ref{main lemma} that the bound
\begin{align}\label{eq: step 1 res}
\max_{j \in \T} \E |\RR_{jj}(v)|^p I(v) \le C_0^p
\end{align}
holds for all \(w \in \mathcal D\) and \(1 \le p \le A_1 (nv)^{(1-2\alpha)/2}\). We are interested in \(p\) of the order \(\log n\). Denote \(v_1 \eqdef n^{-1} \log^{2/(1-2\alpha)} n\) and take \(p = A_1 \log n\) (It is sufficient to consider only such values of \(p\). For all \(1 \le q \le p\) we may apply the Lyapunov inequality). It is easy to see that~\eqref{eq: step 1 res} holds for all \( v: v_1 \leq v\leq V\) with \( p = A_1 \log n\).  Let us fix \(v: v_0 \le v \le v_1\) and let \(l_0 \eqdef \min \{l \geq 1: v s_0^l \geq v_1\}\). We take \(s \eqdef s_0^{l_0}\). It is clear that \( s \lesssim \log^\frac{1+2\alpha}{1-2\alpha} n\). Applying Proposition~\ref{lem: descent property},~\eqref{eq: step 1 res} and~\eqref{I(v) property} we may show that for all \(v \geq v_0\)
\begin{eqnarray*}
\max_{j \in \T} \E |\RR_{jj}(v)|^p I(v) \le C_0^p \log^{\left(\frac{1+2\alpha}{1-2\alpha}\right) p} n.
\end{eqnarray*}
It remains to remove the log factor on the right hand side of the previous inequality. To this aim we shall adopt the moment matching technique which has been successfully  used  recently by Lee and in Yin in~\cite{LeeYin2014}(see Lemma~5.2 and Lemma~5.3). We denote by \(Y_{jk}, 1 \le j \le k \le n\) a triangular set of random variables such that \(|Y_{jk}| \le D\), for some \(D\) chosen later, and
\begin{eqnarray*}
\E X_{jk}^s = \E Y_{jk}^s \, \text { for } \, s = 1, ... , 4.
\end{eqnarray*}
It follows from~\cite{LeeYin2014}[Lemma~5.2] that such a set of random variables exists. Let us denote \(\W^{\y}: = \frac{1}{\sqrt n} \Y, \RR^\y: = (\W^\y - z \II)^{-1}\) and \(m_n^\y(z): = \frac{1}{n} \Tr\RR^\y(z)\). Then, repeating the proof of \cite{GotzeNauTikh2016a}[Lemma~3.5] we show that for all \(v \geq v_0\)  and \(5 \le p \le A_1 \log n\) there exist positive constants \(C_1, C_2\) such that
\begin{eqnarray}\label{compare}
\E|\RR_{jj}(v)|^p I(v) \le C_1^p + C_2 \E|\RR_{jj}^\y(v)|^p I(v).
\end{eqnarray}
It is easy to see that \(Y_{jk}\) are sub-Gaussian random variables. Repeating the proof of Lemma~\ref{main lemma} below for sub-Gaussian random variables one may show that
\begin{eqnarray*}
\E|\RR_{jj}^\y(v)|^p I(v) \le C_0^p
\end{eqnarray*}
for all \(w \in \mathbb D\) and \(1 \le p \le A_1 nv\). Here one needs to replace Lemmas~\ref{e1}--\ref{e4-e5} by the Hanson-Wright inequality (see, for example, \cite{GotzeNauTikh2016a}[Lemma A.4--A.7]). For details see the proof of the corresponding result in~\cite{GotzeNauTikh2016a}[Lemma~4.1]).
\end{proof}

\begin{lemma}\label{main lemma}
Let \(V\) be some fixed number. There exist a positive constant \(C_0\) depending on \(V, z\) and positive constants \(A_0, A_1\) depending on \(C_0\) such that
\begin{eqnarray*}\label{eq: resolvent bound}
\max_{j,k = 1, \ldots, nm } \E \left|\RR_{jk}(z, u + i v)\right|^p I(u+i v) \le C_0^p
\end{eqnarray*}
for all \(A_0 n^{-1} \le v \le V, u \in \J_\ve\) and \(1 \le p \le A_1(nv)^\frac{1-2\alpha}{2}\). 
\end{lemma}

\begin{remark}
In Lemma~\ref{main lemma} we bound the off-diagonal entries as well. We use the bound for off diagonal entries to show that~\eqref{compare} holds. See \cite{GotzeNauTikh2016a}[Lemma~3.5] for details. 	
\end{remark}

Let us define \( J_\alpha, \alpha = 1, \ldots, m\) as an arbitrary subsets of \(\T\). Here \(J_\alpha\) will correspond to the indices of rows deleted from \(\X^{(\alpha)}\). Similarly we define  \( K_\alpha, \alpha = 1, \ldots, m\) as the indices of columns deleted from \(\X^{(\alpha)}\).  Moreover, let \( |J_\alpha \setminus K_\alpha|  = 0\) or \(1\).  For a particular choice of these sets we define
\begin{eqnarray}\label{eq: set def}
\J \eqdef \{ J_\alpha \subset \T , \alpha = 1, \ldots, m \}, \quad \K \eqdef \{ K_\alpha \subset \T , \alpha = 1, \ldots, m \}. 
\end{eqnarray}
Handling now all possible \( J_\alpha, K_\alpha, \alpha = 1, \ldots, m \) we define  \( \mathcal J_L\) as follows
\begin{eqnarray*}
\mathcal J_L \eqdef \{\J, \K: |\J| \le L,  |\J \setminus \K|  \le 1\}, \, L \geq 0.
\end{eqnarray*}
We also define \( T_{J_\alpha} \eqdef (\alpha -1) m + \T \setminus J_{\alpha} \). Similarly we may define \( T_{K_\alpha}\). Let \( \X^{(\alpha, \J_\alpha, \K_\alpha)} \) be a sub-matrix of \(\X^{(\alpha)}\) with entries \( X_{jk}^{(\alpha)}, j \in \T_{J_\alpha}, k \in \T_{K_\alpha} \). Then we may define \( \W^{(\J, \K)}\) as \( \W \) with all \(\X^{\alpha}\) replaced by \( \X^{(\alpha, \J_\alpha, \K_\alpha)} \). Similarly we define \( \V^{(\J, \K)}, \RR^{(\J, \K)} \) and all other quantities. 
Denote
\begin{eqnarray*}
I_{\tau}^{(\J,\K)}(v) \eqdef \prod_{k=1}^{K_v} \one \big[ |\Lambda_n^{(\J, \K)}(u+ i s_0^k v)| \le \tau \imag s(z, u+i s_0^k v)\big ].
\end{eqnarray*}
It is easy to see that
\begin{eqnarray}\label{eq: ind property}
I_{\tau}^{(\J,\K)}(v) \le I_{\tau_1}^{(\J,\K)}(s_0 v)
\end{eqnarray}
for any \(\tau_1: \tau_1 \geq \tau\).

\begin{lemma}\label{lemma step for resolvent}
Assume that the conditions \(\CondTwo\) hold. Let \(C_0\) and \(s_0\) be arbitrary numbers such that \(C_0 \geq \max(1/, H_0), s_0 \geq 2^{1/\kappa}\). There exist sufficiently large \(A_0\) and small \(A_1\) depending on \(C_0, s_0, V\) only such that the following statement holds. Fix some \(\tilde v: v_0 s_0 /\sqrt{\gamma(u)} \leq \tilde v \le V\). Suppose that for some integer \(K > 0\), all \(u, v',q\) such that  \(\tilde v \leq v' \leq V,\, u \in \J_\ve, 1 \le q \le A_1 (n v')^\frac{1-2\alpha}{2}\)
\begin{eqnarray}\label{main condition 1 0}
\max_{(\J, \K) \in \mathcal J_{K+1}} \max_{\alpha, \beta = 1, \ldots, m} \max_{l_\alpha \in \T_{J_\alpha}, k_\beta \in \T_{K_\beta} }\E |\RR_{l_\alpha k_\beta}^{(\J, \K)}(v')|^q I_{2\tau}^{(\J, \K)}(v) \le C_0^q.
\end{eqnarray}
Then for all \(u,v, q\) such that \(\tilde v/s_0 \leq v \le V, u \in \J_\ve\), \(1 \le q \le A_1 (nv)^\frac{1-2\alpha}{2} \)
\begin{eqnarray*}
\max_{(\J, \K) \in \mathcal J_{K}} \max_{\alpha, \beta = 1, \ldots, m} \max_{l_\alpha \in \T_{J_\alpha}, k_\beta \in \T_{K_\beta} }\E |\RR_{l_\alpha k_\beta}^{(\J, \K)}(v)|^q I_{2\tau}^{(\J, \K)}(v) \le C_0^q.
\end{eqnarray*}
\end{lemma}

\begin{proof}
Let us fix \( \alpha = 1, \ldots, m\). Without loss of generality we assume that \( \J = \K \). We fix some \(j \in \T \setminus J_\alpha \) and denote 
\begin{eqnarray*}
\widetilde \J = \{ J_\beta, \beta \neq \alpha, J_\alpha \cup \{j\} \}.
\end{eqnarray*}
Similarly we define \(\widetilde \K \). We first consider the diagonal entries. Applying Schur's inverse formula, we may write
\begin{eqnarray*}
\RR_{j_\alpha j_\alpha}^{(\J, \K)}= -\frac{1}{-w-m_n^{([\alpha+1]+m, \J, \K)}+ \frac{|z|^2}{w + m_n^{([\alpha-1], \J, \K)}(z,w)} }(1 - \ve_{j_\alpha}^{(\J, \K)}\RR_{j_\alpha j_\alpha}^{(\J, \K)}),
\end{eqnarray*}
where
\begin{eqnarray*}
	\ve_{j_\alpha}^{(\J, \K)} \eqdef \widetilde \ve_{j_\alpha}^{(\J, \K)} + \frac{|z|^2}{w + m_n^{([\alpha-1], \J, \K)}(z,w)} \RR_{j_{\alpha +m}, j_{\alpha + m}}^{(\widetilde \J, \K)} \widehat \ve_{j_\alpha}^{(\J, \K)}. 
\end{eqnarray*}
Here \(\widetilde \ve_{j_\alpha}^{(\J, \K)} \eqdef \widetilde \ve_{j1}^{(\J, \K)}+\ldots+ \widetilde \ve_{j4}^{(\J, \K)}\),
\begin{eqnarray*}
	\widetilde \ve_{j_\alpha, 1}^{(\J, \K)}& \eqdef& \frac1n \sum_{k\in \T\setminus K_{[\alpha+1]}} \RR_{k_{[\alpha +1 ] +m }, k_{[\alpha +1] + m }}^{(\J, \K)}-\frac1n \sum_{k\in \T\setminus K_{[\alpha+1]}} \RR_{k_{[\alpha +1 ] + m }, k_{[\alpha + 1 ] + m }}^{(\widetilde \J, \K)} ,\\
	\widetilde \ve_{j_\alpha, 2}^{(\J, \K)}&\eqdef& -\frac1n\sum_{l\ne k \in  \T\setminus K_{[\alpha+1]}} X_{jk}^{(\alpha)} X_{jl}^{(\alpha)}  \RR_{l_{[\alpha +1 ] + m }, k_{[\alpha +1 ] + m }}^{(\widetilde \J, \K)}, \\
	\widetilde \ve_{j_\alpha, 3}^{(\J, \K)}& \eqdef&-\frac1n\sum_{l\in \T\setminus K_{[\alpha+1]} } ([X_{jl}^{(\alpha)}]^2-1) \RR_{l_{[\alpha +1 ] + m }, l_{[\alpha +1] + m }}^{(\widetilde \J, \K)}, \\
	\widetilde \ve_{j_\alpha, 4}^{(\J, \K)}&\eqdef& \frac{z+\overline z}{\sqrt n}\sum_{l\in \T\setminus K_{[\alpha+1]}} X_{jl}^{(\alpha)} \RR_{j_{[\alpha +1 ] + m },l_{[\alpha +1] + m }}^{(\widetilde \J, \K)}.
\end{eqnarray*}
and \(\widehat \ve_{j}^{(\J, \K)} \eqdef \widehat \ve_{j,1}^{(\J, \K)}+\ldots+\widehat \ve_{j,3}^{( \J, \K)}\), where 
\begin{eqnarray*}
	\widehat \ve_{j_\alpha,1}^{(\J, \K)}&\eqdef& \frac1n \sum_{l \in \T\setminus J_{[\alpha-1]}} \RR_{l_{[\alpha-1] }, l_{[\alpha-1]}}^{(\J, \K)} - \frac1n \sum_{l \in \T\setminus J_{[\alpha-1]}} \RR_{l_{[\alpha-1] }, l_{[\alpha-1]}}^{(\widetilde \J, \widetilde \K)}, \\
	\widehat \ve_{j_\alpha,2}^{(\J, \K)} &\eqdef&-\frac1n\sum_{k\ne l\in \T\setminus J_{[\alpha-1]}}X_{kj}^{([\alpha-1])}X_{lj}^{([\alpha-1])}\RR_{k_{[\alpha-1] }, l_{[\alpha-1]}}^{(\widetilde \J, \widetilde \K)},\\
	\widehat \ve_{j_\alpha,3}^{( \J, \K)} &\eqdef&-\frac1n\sum_{l\in \T\setminus J_{[\alpha-1]}}([X_{lj}^{([\alpha-1])}] ^2-1)\RR_{l_{[\alpha-1] }, l_{[\alpha-1]}}^{(\widetilde \J, \widetilde \K)}.
\end{eqnarray*}
We conclude from~\eqref{rj} and Lemma~\ref{l: bound for some functions} that there exist a positive constant \(H_0\) depending on \(u_0, V, z\) and positive constant \(A\) depending on \(H_0\) such that the following inequality holds:  
\begin{eqnarray*}
|\RR_{j_\alpha j_\alpha}^{(\J, \K)}(v)| I_{\tau}^{(\J, \K)}(v) \le H_0 \left (1 + | \varepsilon_{j_\alpha}^{(\J,\K)} \RR_{j_\alpha j_\alpha }^{(\J,\K)}| I_{\tau}^{(\J, \K)}(v) \right).
\end{eqnarray*}
Hence,
\begin{eqnarray*}
\E\big[|\RR_{j_\alpha j_\alpha }^{(\J,\K)}|^q I_{\tau}^{(\J, \K)}(v)\big] \le 2^q H_0^q \left (1 + \E^\frac12\big[|\varepsilon_{j_\alpha}^{(\J,\K)}|^{2q} I_{\tau}^{(\J, \K)}(v)\big] \E^\frac12\big[|\RR_{j_\alpha j_\alpha}^{(\J,\K)}|^{2q} I_{\tau}^{(\J, \K)}(v)\big] \right).
\end{eqnarray*}
It follows from Proposition~\ref{lem: descent property},~\eqref{eq: ind property} and~\eqref{main condition 1 0} that
\begin{eqnarray}\label{eq: est 1}
\E^\frac12\big[|\RR_{j_\alpha j_\alpha}^{(\J,\K)}(v)|^{2q} I_{\tau}^{(\J, \K)}(v)\big] \le s_0^q C_0^q. 
\end{eqnarray}
The Cauchy-Schwartz inequality and Lemma imply
\begin{eqnarray*}
\E [|\ve_{j}^{(\J, \K)} |^{2 q}  I_{\tau}^{(\J, \K)} ]\le  2^{2 q}\E[|\widetilde \ve_j^{(\J, \K)}|^{2 q} I_{\tau}^{(\J, \K)}] +2^{2 q} H_0^q |z|^{4q}  \E^\frac12[|\widehat \ve_{j}^{(\J, \K)}|^{4 q} I_{\tau}^{(\J, \K)}] \E^\frac12[|\RR_{j_{\alpha+m},j_{\alpha+m}}^{(\widetilde \J, \K)}|^{4 q} I_{\tau}^{(\J, \K)} ].
\end{eqnarray*}
Similarly to~\eqref{eq: est 1} 
\begin{eqnarray}
\E^\frac12[|\RR_{j_{\alpha+m},j_{\alpha+m}}^{(\widetilde \J, \K)}(v)|^{4 q} I_{\tau}^{(\J, \K)}(v) ] \le s_0^q \E^\frac12[|\RR_{j_{\alpha+m},j_{\alpha+m}}^{(\widetilde \J, \K)}(s_0 v)|^{4 q} I_{2\tau}^{(\widetilde \J, \K)}(s_0 v) ] \le (C_0 s_0)^{2q}. 
\end{eqnarray}
Applying~\eqref{eq: ind property} we obtain
\begin{eqnarray*}
\E[|\widetilde \ve_{j_\alpha}^{(\J, \K)}|^{2 q} I_{\tau}^{(\J, \K)}] \le \E[|\widetilde \ve_{j_\alpha}^{(\J, \K)}|^{2 q} I_{3\tau/2}^{(\widetilde \J, \K)}(s_0v)]
\end{eqnarray*}
It is easy to see from Lemmas~\ref{e1}--\ref{e4-e5} in the appendix that the moment bounds for \(\widetilde \ve_{j_\alpha,2}^{(\J, \K)}\) and  \(\widetilde \ve_{j_\alpha,4}^{(\J, \K)}\) depends on the moments of off-diagonal entries of resolvent which are non \(k\)-descent function. Here we may use 
\begin{eqnarray}\label{eq: resolvent identity}
\RR(w_1) - \RR(w_2) = (w_1 - w_2) \RR(w_1)\RR(w_2), \, w_1, w_2 \in \C^{+},
\end{eqnarray}
which gives us that
\begin{eqnarray*}
|\RR_{j_{\alpha'} k_{\beta'}}^{(\widetilde \J, \K)}(z,v)| \le |\RR_{j_{\alpha'} k_{\beta'}}^{(\widetilde \J, \K)}(z,s v)| + s |  \RR_{j_{\alpha'} j_{\alpha'}}^{(\widetilde \J, \K)}(z,s v)|^\frac12 | \RR_{k_{\beta'} k_{\beta'}}^{(\widetilde \J, \K)}(z,v) |^\frac12. 
\end{eqnarray*}
Now the desired bound follows from Proposition~\ref{lem: descent property}  and assumption~\eqref{main condition 1 0}. Lemmas~\ref{e1}--\ref{e4-e5} in the appendix  imply
\begin{align*}
\E[|\ve_{j_\alpha}^{(\J, \K)}|^{2 q} I_{3\tau/2}^{(\widetilde \J, \K)}(s_0 v)] &\le  \left(\frac{C C_0 s_0^2 q }{(nv)^\frac{1-2\alpha}{2} }\right)^{2 q}.
\end{align*}
Here, \(C\) depends on \(z \) as well.  Similarly, applying Lemmas~\ref{e1+n}--\ref{e3+n} from the appendix we may estimate
\begin{eqnarray*}
\E^\frac{1}{2}[|\widehat \ve_{j_\alpha}^{(\J, \K)}|^{4 q} I_{2\tau}^{(\widetilde \J, \widetilde \K)}(s_0 v)] \le \left(\frac{C C_0 s^2 q }{(nv)^\frac{1-2\alpha}{2} }\right)^{2 q}.
\end{eqnarray*}
The last two inequalities yield the following bound:
\begin{eqnarray*}
\E [|\ve_{j_\alpha}^{(\J, \K)}|^{ 2q} I_{\tau}^{(\J, \K)}]   \le \left(\frac{C C_0^2 s^3 q }{(nv)^\frac{1-2\alpha}{2} }\right)^{2 q}.
\end{eqnarray*}
Hence, choosing sufficiently large \(A_0\) and small \(A_1\) we may show that
\begin{eqnarray*}
\E [|\RR_{j_\alpha j_\alpha}^{(\J, \K)}(v)|^q I_{\tau}^{(\J, \K)}] \le C_0^q.
\end{eqnarray*}	
To deal with off-diagonal entries we use the following representation
\begin{eqnarray*}
\RR_{j_\alpha, k_\beta}^{(\J, \K)} &=& -\frac{1}{\sqrt n} \sum_{l \in \T\setminus K_{[\alpha+1]}}  X_{jl}^{(\alpha)}  \RR_{l_{[\alpha+1]+m}, k_\beta}^{(\widetilde \J, \K)}  \RR_{j_\alpha, j_\alpha}^{(\J, \K)} \\
& - & \frac{z}{\sqrt n} \sum_{l \in \T\setminus J_{[\alpha-1]}}  X_{lj}^{([\alpha-1])} \RR_{l_{[\alpha-1]}, k_{\beta}}^{(\widetilde \J, \widetilde \K)}     \RR_{j_{\alpha+m}, j_{\alpha+m}}^{(\widetilde \J, \K)}  \RR_{j_\alpha j_\alpha}^{(\J, \K)}.
\end{eqnarray*}
Applying now Rosenthal's inequality (e.g.~\cite{Rosenthal1970}[Theorem~3] and~\cite{JohnSchecttmanZinn1985}[Inequality~(A)]) and assumption~\eqref{main condition 1 0} we may show that one may choose  sufficiently large \(A_0\) and small \(A_1\) such that
\begin{eqnarray*}
	\E |\RR_{j_\alpha k_\beta}^{(\J, \K)}(v)|^q I_{\tau}^{(\J, \K)} \le C_0^q.
\end{eqnarray*}
\end{proof}

\begin{proof}[Proof of Lemma~\ref{main lemma 0}]
Let us choose some sufficiently large constant \(C_0 > \max(1/V, H_0)\) and fix \(s_0 \eqdef 2^\frac{2}{1-2\alpha}\). Here \(H_0\) is defined in Lemma~\ref{l: bound for some functions}. We also choose \(A_0\) and \(A_1\) as in Lemma~\ref{lemma step for resolvent}. We fix \(u \in \J_\ve\) . Let \(L\eqdef \big[\log_{s_0}\big(V \sqrt{\gamma(u)}/v_0 \big)\big]+1\). Since \(\|\RR^{(\J)}(V)\| \le V^{-1}\) we may write 
\begin{eqnarray*}
\max_{(\J, \K) \in \mathcal J_{L}} \max_{\alpha, \beta = 1, \ldots, m} \max_{l_\alpha \in \T_{J_\alpha}, k_\beta \in \T_{K_\beta} }\E |\RR_{l_\alpha k_\beta}^{(\J, \K)}(V)|^q I_{(L +1)\tau }^{(\J, \K)}(V) \le C_0^q
\end{eqnarray*}
for all  \(1 \le p \le A_1 (nV)^\frac{1-2\alpha}{2}\). Fix arbitrary \(v: V/s_0 \le  v \leq V\) and \(p: 1 \le p \le A_1 (nv)^\frac{1-2\alpha}{2}\). Lemma~\ref{lemma step for resolvent} yields that
\begin{eqnarray*}
\max_{(\J, \K) \in \mathcal J_{L-1}} \max_{\alpha, \beta = 1, \ldots, m} \max_{l_\alpha \in \T_{J_\alpha}, k_\beta \in \T_{K_\beta} }\E |\RR_{l_\alpha k_\beta}^{(\J, \K)}(v)|^q I_{L \tau }^{(\J, \K)}(v) \le C_0^q
\end{eqnarray*}
for \(1 \le p \le A_1 (n V/s_0)^\frac{1-2\alpha}{2}\), \(v \geq V/s_0\). We may repeat this procedure \(L\) times and finally obtain
\begin{eqnarray*}
\max_{l,k = 1, \ldots, nm}\E|\RR_{lk}(v)|^p I_\tau(v) \le C_0^p
\end{eqnarray*}
for \(1 \le p \le A_1 (n V /s_0^{L})^\frac{1-2\alpha}{2} \le A_1 (n \tilde v_0)\) and \(v \geq  v_0/\sqrt{\gamma(u)}\). 
\end{proof}

\section{Estimation of \(\TTT_n\)} \label{sec: T_n}
In this section we prove the following theorem. 
\begin{theorem}\label{th: T_n bound} For any \( w \in \mathcal D\) and all \(1 \le p \le A_1 \log n\) 
\begin{eqnarray*}
\max_{1 \le \alpha \le m} \E |T_n^{(\alpha)}|^p \le \frac{C^p p^{2p} \mathcal E^p(\kappa p)}{(nv)^p},
\end{eqnarray*}
where \(\mathcal E(q)\) is defined in~\eqref{E def} and \(\kappa\) is some positive constant depending on \(\delta\) only.
\end{theorem}
We shall proceed as in~\cite{GotzeNauTikh2015a} applying Stein's method. 

\subsection{Framework for moment bounds of some statistics of r.v.} 
We start from the following lemma, which provide a framework to estimate the moments of some statistics of independent random variables. 

Let \(X_1, \ldots , X_n\) be independent r.v. and denote
\begin{eqnarray*}
\mathfrak M \eqdef  \sigma\{X_1, \ldots , X_n\}, \quad \mathfrak M^{(j)} \eqdef  \sigma\{X_1, \ldots X_{j-1}, X_{j+1}, \ldots , X_n\}.
\end{eqnarray*}
For simplicity we introduce \(\E_j(\cdot) \eqdef \E(\cdot \big | \mathfrak M^{(j)})\). Assume that \(\xi_{j}, f_{j}, j = 1, \ldots , n\), are \(\mathfrak M\)-measurable r.v. and
\begin{eqnarray}\label{assumption 1}
\E_j(\xi_{j})   = 0.
\end{eqnarray}
We consider the following statistic:
\begin{eqnarray*}
T_n^{*} \eqdef \sum_{j=1}^n \xi_{j} f_{j} + \mathcal R,
\end{eqnarray*}
where \(\mathcal R\) is some \(\mathfrak M\) measurable function. Moreover, let \(\widehat f_{j}\) an arbitrary \(\mathfrak M^{(j)}\)-measurable r.v. and
\begin{eqnarray*}
\tTj \eqdef \E_j(T_n^{*}).
\end{eqnarray*}
\begin{lemma} \label{lem: T_n general lemma}
For all \(p \geq 2\) there exist some absolute constant \(C\) such that
\begin{eqnarray*}
\E|T_n^{*}|^p \le C^p \bigg(\mathcal A^p + p^\frac p2 \mathcal B^\frac p2 + p^p\mathcal C + p^p \mathcal D + \E |\mathcal R|^p \bigg), 
\end{eqnarray*}
where
\begin{align*}
&\mathcal A \eqdef \E^\frac1p\left( \sum_{j=1}^n \E_j|\xi_{j} (f_{j} - \widehat f_{j})|\right)^p,\\
&\mathcal B \eqdef \E^\frac 2p\left(\sum_{j=1}^n \E_j(|\xi_j(T_n^{*} - \tTj)|) |\widehat f_{j}|\right)^\frac p 2, \\
&\mathcal C \eqdef \sum_{j=1}^n \E |\xi_{j}||T_n^{*} - \tTj|^{p-1}|\widehat f_{j}|, \\
&\mathcal D \eqdef \sum_{j=1}^n \E |\xi_{j}||f - \widehat f_{j}| |T_n^{*} - \tTj|^{p-1}.
\end{align*}
\end{lemma}
\begin{remark}
We conclude the statement of the last lemma by several remarks. 
\begin{enumerate}
	\item It follows from the definition of \( \mathcal A, \mathcal B, \mathcal C, \mathcal D\) that instead of estimation of high moments of \(\xi_j\) one needs to estimate conditional expectation \(\E_j|\xi_j|^\alpha\) for some small \(\alpha\). Typically, \(\alpha \le 4\);
	\item Moreover, to get the desired bounds one needs to choose an appropriate approximation \(\widehat f_j\) of \(f_j\) and estimate \(T_n^{*} - \tTj\); 
	\item This lemma may be generalized as follows. We may assume that 
	\begin{align}\label{eq: T star}
		T_n^{*} \eqdef \sum_{\nu=1}^{m} \sum_{j=1}^n \xi_{j\nu} f_{j \nu} + \mathcal R,
	\end{align}
	where \(\xi_{j\nu}, f_{j\nu}, j = 1, \ldots , n, \nu = 1, \ldots , m\), are \(\mathfrak M\)-measurable r.v. such that
	\begin{eqnarray*}
	\E_j (\xi_{j\nu}) = 0, \quad \nu = 1, \ldots , m.
	\end{eqnarray*}
	Repeating the previous calculations we obtain
	\begin{eqnarray*}
	\E |T_n^{*}|^p \le  C^p \bigg(\sum_{\nu=1}^{m} (\mathcal A_{\nu}^p + p^\frac p 2 \mathcal B_{\nu}^\frac p2 + p^p \mathcal C_{\nu}) + \E |\mathcal R|^p \bigg),
    \end{eqnarray*}
	where \(\mathcal A_{\nu}, \mathcal B_{\nu}, \mathcal C_{\nu}\) are defined similarly to the corresponding quantities in Lemma~\ref{lem: T_n general lemma}.
\end{enumerate}
\end{remark}

\begin{proof}[Proof of Lemma~\ref{lem: T_n general lemma}] Let us introduce the following function:
\begin{eqnarray}\label{eq: varphi def}
\varphi(\zeta)\eqdef \zeta|\zeta|^{p-2}.
\end{eqnarray} 
In these  notations \(\E |T_n^{*}|^p\) may be rewritten as follows
\begin{eqnarray*}\label{eq: T_n representation}
\E|T_n^{*}|^p=\E T_n^{*}\varphi(T_n^{*})= \sum_{j=1}^n \xi_{j} f_{j} \varphi(T_n^{*}) + \mathcal R \varphi(T_n^{*}) = \sum_{l=1}^{4}\mathcal A_{l} ,
\end{eqnarray*}
where
\begin{eqnarray*}
\mathcal A_{1} &\eqdef& \sum_{j=1}^n \E \xi_{j} \widehat f_{j} \varphi(\tTj), \qquad\quad \mathcal A_{2} \eqdef \sum_{j=1}^n \E \xi_{j} \widehat f_{j} (\varphi(T_n^{*}) - \varphi(\tTj)),\\
\mathcal A_{3} &\eqdef& \sum_{j=1}^n \E \xi_{j} (f_{j} -\widehat f_{j}) \varphi(T_n^{*}), \quad \mathcal A_{4} \eqdef \mathcal R \varphi(T_n^{*}).
\end{eqnarray*}
It follows from~\eqref{assumption 1} that \(\mathcal A_{1} = 0\).  Applying the following useful inequality
\begin{eqnarray}\label{useful inequality}
(x + y)^q \le e x^q + (q+1)^q y^q, x, y > 0, \quad q \geq 1,
\end{eqnarray}
we estimate \( \mathcal A_3\) by the sums of the following terms 
\begin{eqnarray*}
\mathcal A_{31} &\eqdef& e \sum_{j=1}^n \E |\xi_{j}||f_{j} -\widehat f_{j}| |\tTj|^{p-1}, \\
\mathcal A_{32} &\eqdef& p^{p-1} \sum_{j=1}^n \E |\xi_{j}||f_{j} -\widehat f_{j}| |T_n^{*} - \tTj|^{p-1}.
\end{eqnarray*}
The term \( \mathcal A_{32} \) we remain unchanged. It will appear in the final bound. H\"older', Jensen' and Young's inequalities imply
\begin{eqnarray}
\mathcal A_{31} &\le& \E^\frac{p-1}{p} |T_n^*|^p \E^\frac1p \bigg(\sum_{j=1}^n \E_j |\xi_{j}||f_{j} -\widehat f_{j}|\bigg)^p \nonumber \\
\label{eq: a_3}
 &\le&  \rho \E |T_n^*|^p + \E \bigg(\sum_{j=1}^n \E_j |\xi_{j}||f_{j} -\widehat f_{j}|\bigg)^p.
\end{eqnarray}
It follows from the Taylor formula that
\begin{eqnarray*}
\mathcal A_{2} = \sum_{j=1}^n \E \xi_{j} \widehat f_{j} (T_n^{*} - \tTj) \varphi'(\tTj+ \theta(T_n^{*} - \tTj)) ,
\end{eqnarray*}
where \(\theta\) is a uniformly distributed on \([0,1]\) r.v., independent of \(X_j, j = 1, \ldots, n\). Taking absolute values and using~\eqref{useful inequality} we get
\begin{eqnarray*}
|\mathcal A_{2}| \le \mathcal A_{21} + \mathcal A_{22},
\end{eqnarray*}
where
\begin{eqnarray*}
\mathcal A_{21} &\eqdef& e p\sum_{j=1}^n \E \E_j(|\xi_j(T_n^{*} - \tTj)|)|\widehat f_{j}| |\tTj|^{p-2},\\
\mathcal A_{22} &\eqdef& p^{p-1}\sum_{j=1}^n \E\E_j(|\xi_{j}||T_n^{*} - \tTj|^{p-1})|\widehat f_{j}| .
\end{eqnarray*}
Applying H\"older's and Jensen's inequalities we obtain
\begin{eqnarray*}
\mathcal A_{21} \le C p \E^\frac 2p\left(\sum_{j=1}^n \E_j(|\xi_j(T_n^{*} - \tTj)|) |\widehat f_{j}|\right)^\frac p 2  \E^\frac{p-2}{p} | T_n^{*}|^p.
\end{eqnarray*}
Now Young's inequality implies 
\begin{eqnarray}\label{eq: a_21}
\mathcal A_{21} \le C^p p^\frac p2 \E\left(\sum_{j=1}^n \E_j(|\xi_j(T_n^{*} - \tTj)|) |\widehat f_{j}|\right)^\frac p 2 +  \rho\E | T_n^{*}|^p.
\end{eqnarray}
Finally, for the term \(\mathcal A_4\) we may write
\begin{eqnarray}\label{eq: a_4}
\mathcal A_{4} \le C^p \E|\mathcal R|^p + \rho\E |T_n^{*}|^p. 
\end{eqnarray}
Inequalities~\eqref{eq: a_3}, \eqref{eq: a_21} and~\eqref{eq: a_4} yield the claim of the lemma.
\end{proof}

\subsection{Proof of Theorem~\ref{th: T_n bound}}

\begin{proof}[Proof of Theorem~\ref{th: T_n bound}] 
 We consider the case \( \alpha = 1\) only. For simplicity we shall write \( T_n = T_n^{(1)} \). First we mention that \(T_n\) is of the kind ~\eqref{eq: T star}. Indeed, here
\begin{align*}
&\xi_{j1} = \widetilde \varepsilon_{j2} + \ldots + \widetilde \varepsilon_{j5}, \quad \xi_{j2} = \widehat \ve_{j2}+  \widehat \ve_{j3}, \\
&f_{j1} = \RR_{jj}, \quad f_{j2} = - \frac{|z|^2\RR_{jj}\RR_{j_{m+1},j_{m+1}}^{(j)}}{w + m_n^{(m)}(z,w)}, \\
&\mathcal R = \frac1n\sum_{j=1}^n  \widetilde \ve_{j1} f_{j1}  +\frac1n\sum_{j=1}^n \widehat \ve_{j1} f_{j2}. 
\end{align*}
We introduce the following smoothed version of \( I(v)\). We denote
\begin{eqnarray*}
h_{\alpha, \beta}(x) \eqdef \begin{cases}
1, & \text{ if } 0 < x < \imag s(z,w), \\
1 - \frac{x - \alpha \imag s(z,w)}{(\beta - \alpha) \imag s(z,w)}, & \text{ if } \alpha \imag s(z,w) < x < \beta \imag s(z,w), \\
0, &\text{ otherwise,}
\end{cases}
\end{eqnarray*}
and  write \( H(v) \eqdef \prod_{k=1}^{K_v} \prod_{\alpha = 1}^{m} h_{\tau, \frac32 \tau}(|\Lambda_n^{(\alpha)}(v)|) \). It is easy to see that
\begin{eqnarray}\label{eq: bound for H}
I(v) \le H(v) \le I_{3/2\tau}(v) \le I_{2\tau}^{(j)}(v).
\end{eqnarray}
To simplify all notations below we shall often omit the bottom index from \(I_\tau(v)\) and all its counterparts. We will also write \( I(v) \le I^{(j)}(v)\) having in mind that \(I_\tau(v) \le I_{\tau'}^{(j)}(v)\) for some fixed \(\tau' > \tau\). 

Applying the notation of \(H(v)\) we write
\begin{eqnarray*}
\E |T_n|^p I(v) \le \E |T_n|^p H^p(v) . 
\end{eqnarray*}
For simplicity we set \( T_{n,h}  \eqdef T_n H(v)\).  Recall the definition~\eqref{eq: varphi def} of  \(\varphi(\zeta)\). We rewrite the r.h.s. of the previous inequality as follows
\begin{eqnarray*}
 \E |T_{n,h}|^p  = \frac{1}{n} \sum_{j=1}^n \sum_{\nu = 1}^{2}\E \xi_{j \nu} f_{j \nu} H(v)  \varphi(T_{n,h}) + \E \mathcal R H(v) \varphi(T_{n,h}).  
\end{eqnarray*}
We denote 
\begin{eqnarray*}
	\mathcal{A} \eqdef \frac{1}{n} \sum_{j=1}^n\sum_{\nu = 1}^{2} \E \xi_{j \nu} f_{j \nu} H(v) \varphi(T_{n,h}). 
\end{eqnarray*}
In these notations \( \E |T_{n,h}|^p\) may be rewritten as follows
\begin{eqnarray*}
	\E |T_{n,h}|^p  = \mathcal A + \E \mathcal R H(v) \varphi(T_{n,h}).  
\end{eqnarray*}

\subsection{Estimate of \(\E \mathcal R H(v) \varphi(T_{n,h})\)} 
Simple calculations imply
\begin{eqnarray*}
	\mathcal R &=& \frac{1}{n^2} \sum_{j=1}^n\sum_{l=1}^n [\RR_{jl}]^2 - \frac{1}{n^2} \frac{|z|^2}{w + m_n^{(m)}(z,w)} \sum_{j=1}^n \sum_{l=1}^n [\RR_{l_m, j}]^2 \RR_{j_{m+1}, j_{m+1}}^{(j)}  \\
	&-& \frac{1}{n^2} \frac{|z|^2}{w + m_n^{(m)}(z,w)} \sum_{j=1}^n \sum_{l=1}^n [\RR_{l_m, j_{m+1}}]^2 \RR_{j j}. 
\end{eqnarray*}

Applying~\eqref{eq: bound for H}, Lemma~\ref{l: bound for some functions} and Lemma~\ref{lem: res inequalities} in the appendix we conclude that 
\begin{eqnarray*}
|\mathcal R|H(v)  \lesssim \frac{\imag s(z,w)}{nv} + \frac{|z|^2}{n^2 v } \sum_{j=1}^n \imag \RR_{j j} |\RR_{j_{m+1}, j_{m+1}}^{(j)}|I(v) + \frac{|z|^2}{n^2 v}  \sum_{j=1}^n  \imag \RR_{j_{m+1}, j_{m+1}} |\RR_{j j}| I(v).
\end{eqnarray*}
Using now H\"older's inequality and Young's inequality we come to the following inequality
\begin{eqnarray*}
\E \mathcal R H(v) \varphi(T_{n,h}) &\le&   C^p \E|\mathcal R|^p H(v)^p   +  \rho \E |T_{n,h}|^p \\
&\le&  C^p (1+|z|)^{2p}  \frac{\mathcal A^p(2p)}{(nv)^p} + \rho \E |T_{n,h}|^p.
\end{eqnarray*}

\subsection{Estimate of \(\mathcal A\)} The estimation of \( \mathcal A\) is more involved. 
Let us introduce conditional expectations \( \E_{j}(\cdot) \eqdef \E (\cdot \big | \mathfrak M^{(j)}) \)  (resp. \( \E_{j,j}(\cdot) \eqdef \E (\cdot \big | \mathfrak M^{(j, j)}) \)) with respect to \(\sigma\)-algebras \(\mathfrak M^{(j)}\) (resp. \(\mathfrak M^{(j, j)} \)).   
Here \(\mathfrak M^{(j)}\) (resp. \(\mathfrak M^{(j, j)} \)) is formed from all \(X_{lk}^{(\alpha)}, l, k = 1, \ldots, n, \alpha = 1, \ldots, m\), except \(X_{jk}^{(1)}, k = 1, \ldots, n \) (resp.    except \(X_{jk}^{(1)}, X_{lj}^{(1)}, k, l =1, \ldots, n \)).  
Moreover, we introduce the following notations
\begin{eqnarray*}
\widetilde T_{n,h} ^{(j)} &\eqdef&  \E ( T_{n,h} \big | \mathfrak M^{(j)} ), \quad \widetilde T_{n,h} ^{(j,j)} \eqdef  \E ( T_{n,h} \big | \mathfrak M^{(j, j)} ),  \\
 \widetilde \Lam_n^{(j)} &\eqdef& \E ( \Lam_n \big | \mathfrak M^{(j)} ),  \quad\,\,\, \widetilde \Lam_n^{(j,j)} \eqdef \E ( \Lam_n \big | \mathfrak M^{(j, j)} ). 
\end{eqnarray*}
We rewrite \(\mathcal A\) as follows \( \mathcal A = \mathcal A_1 + \ldots + \mathcal A_5\), where
\begin{eqnarray*}
\mathcal A_1 &\eqdef& \frac{1}{n} \sum_{j=1}^n \sum_{\nu = 1}^{2} \E \xi_{j \nu} \widetilde  f_{j \nu} \widetilde H^{(j,j)}(v) \varphi(\widetilde T_{n,h}^{(j,j)}), \\
\mathcal A_2 &\eqdef& \frac{1}{n} \sum_{j=1}^n \sum_{\nu = 1}^{2} \E \xi_{j \nu} \widetilde f_{j \nu} [H(v) - \widetilde H^{(j,j)}(v) ] \varphi(\widetilde T_{n,h}^{(j,j)}),\\
\mathcal A_3 &\eqdef& \frac{1}{n} \sum_{j=1}^n \sum_{\nu = 1}^{2} \E \xi_{j \nu} \widetilde f_{j \nu}  H(v)  [ \varphi(T_{n,h}) - \varphi(\widetilde T_{n,h}^{(j,j)})], \\
\mathcal A_4 &\eqdef& \frac{1}{n} \sum_{j=1}^n \sum_{\nu = 1}^{2} \E \xi_{j \nu} [ f_{j \nu} - \widetilde  f_{j \nu} ] H(v) \varphi( T_{n,h}).
\end{eqnarray*}
Moreover, it is easy to check that \( \mathcal A_1 = 0\).

\subsubsection{Bound for \( \mathcal A_2 \) } 
Taking conditional expectation and applying H\"older's inequality it is straightforward to check that
\begin{eqnarray*}
\mathcal A_2 \le \E^\frac{p-1}{p} |T_{n,h}|^p \E^\frac1p \bigg(\frac1n \sum_{j=1}^n \E_{j,j}( |\xi_{j \nu} \widetilde f_{j \nu}||H(v) - \widetilde H^{(j,j)}(v)|   \bigg )^p.
\end{eqnarray*}
Moreover, Young's inequality implies that
\begin{eqnarray}\label{A_2 bound 1}
\mathcal A_2 \le \rho \E|T_{n,h}|^p + C^p \E  \bigg(\frac1n \sum_{j=1}^n \E_{j,j}( |\xi_{j \nu} \widetilde f_{j \nu}||H(v) - \widetilde H^{(j,j)}(v)|   \bigg )^p.
\end{eqnarray}
Let us denote for simplicity
\begin{eqnarray*}
\mathcal B_2 \eqdef \frac1n \sum_{j=1}^n \E \big[ \E_{j,j}( |\xi_{j \nu} \widetilde f_{j \nu}||H(v) - \widetilde H^{(j,j)}(v)|  \big]^p ,
\end{eqnarray*} 
To estimate the r.h.s. of ~\eqref{A_2 bound 1} it is enough to bound \(\mathcal B_2\).  
We may use Lemma~\ref{approximation1} to estimate the difference \(H(v) - \widetilde H^{(j, j)}(v) \). We get
\begin{eqnarray}\label{B_2 bound}
\mathcal B_2  \le \frac{C^p p^p }{n}  \sum_{j=1}^n \sum_{\alpha=1}^{m} \sum_{k=1}^{K_v} \frac{1}{\imag s(z,v_k)}  \E^\frac12 \big [\E_{j,j}( |\xi_{j \nu}| |\Lambda_n^{(\alpha)}(v_k) - \widetilde \Lambda_n^{(\alpha,j,j)}(v_k)| I(v) \big]^{2p},
\end{eqnarray}
where \(v_k \eqdef v s^k, k \geq 0\). We also used the fact that \( K_v^p \le p^p \) and \( \widetilde f_{j \nu} \) is \( \mathfrak M^{(j, j)}\)-measurable.  We fix \(j, \alpha\) and \(k\) and study
\begin{eqnarray*}
 \E \big [\E_{j,j}( |\xi_{j \nu} || \Lambda_n^{(\alpha)}(v_k) - \widetilde \Lambda_n^{(\alpha,j,j)}(v_k)| I(v) \big]^{2p}.
\end{eqnarray*}
Applying Lemma~\ref{lem: important lemma 1} we get 
\begin{eqnarray*}
	\E^\frac12 \big [\E_{j,j}( |\xi_{j \nu} | \Lambda_n^{(\alpha)}(v_k) - \widetilde \Lambda_n^{(\alpha, j,j)}(v_k)| I(v) \big]^{2p} \le \frac{C^p \mathcal E^{p}(\kappa p)}{(nv)^{2p} }. 
\end{eqnarray*}
Since \( \imag s(z,v) \geq (nv)^{-1} \) for \(w \in \mathcal D\), the last inequality and~\eqref{B_2 bound} imply
\begin{eqnarray*}
\mathcal A_2 \le \rho \E|T_{n,h}|^p + \frac{C^p p^p \mathcal E^p(\kappa p)}{(nv)^p}. 
\end{eqnarray*}
\subsubsection{Bound for \( \mathcal A_3 \) } 
Applying Taylor's formula
\begin{eqnarray*}
|\varphi(T_{n,h}) - \varphi(\widetilde T_{n,h}^{(j,j)})|\le p | \widetilde T_{n,h}^{(j,j)} + \theta (T_{n,h} - \widetilde T_{n,h}^{(j,j)})   |^{p-2} | T_{n,h} - \widetilde T_{n,h}^{(j,j)}|  
\end{eqnarray*}
It is easy to check that
\begin{eqnarray}\label{T_n tilde 1}
\widetilde T_{n,h}^{(j,j)} =  \widetilde T_n^{(j,j)} \widetilde H^{(j,j)} - \E_{j,j} [(T_n - \widetilde T_n^{(j,j)})H]
\end{eqnarray}
and
\begin{eqnarray}\label{T_n tilde 2}
T_{n,h} - \widetilde T_n^{(j,j)} \widetilde H^{(j,j)} = (T_n - \widetilde T_n^{(j,j)})\widetilde H^{(j,j)} + T_n(H - \widetilde H^{(j,j)})
\end{eqnarray}
Hence,
\begin{eqnarray*}
|\varphi(T_{n,h}) - \varphi(\widetilde T_{n,h}^{(j,j)})|&\le& p |\widetilde T_{n,h}^{(j,j)}|^{p-2} | T_{n,h} - \widetilde T_{n,h}^{(j,j)}|  \\
&+& p^{p-2} |T_n|^{p-2} |H - \widetilde H^{(j,j)}|^{p-2} | T_{n,h} - \widetilde T_{n,h}^{(j,j)}|\\
&+& p^{p-2} |T_n - \widetilde T_n^{(j,j)}|^{p-2} | T_{n,h} - \widetilde T_{n,h}^{(j,j)}|H^{p-2}(v) \\
&+& p^{p-2} \E_{j,j}^{p-2} [|T_n - \widetilde T_n^{(j,j)}|H] | T_{n,h} - \widetilde T_{n,h}^{(j,j)}|.
\end{eqnarray*}
We obtain that \( \mathcal A_3 \lesssim  \mathcal A_{31} + \ldots +\mathcal A_{34}\), where
\begin{eqnarray*}
\mathcal A_{31} &\eqdef& \frac{p}{n} \sum_{j=1}^n \sum_{\nu = 1}^{2} \E |\xi_{j \nu} \widetilde f_{j \nu}| |\widetilde T _{n,h}|^{p-2} | T_{n,h} - \widetilde T_{n,h}^{(j,j)}| H(v), \\
\mathcal A_{32} &\eqdef& \frac{p^{p-1}}{n} \sum_{j=1}^n \sum_{\nu = 1}^{2} \E |\xi_{j \nu} \widetilde f_{j \nu}| |T_n|^{p-2} |H - \widetilde H^{(j,j)}|^{p-2} | T_{n,h} - \widetilde T_{n,h}^{(j,j)}| H(v), \\
\mathcal A_{33} &\eqdef& \frac{p^ {p-1}}{n} \sum_{j=1}^n \sum_{\nu = 1}^{2} \E |\xi_{j \nu} \widetilde f_{j \nu}|  |T_n - \widetilde T_n^{(j,j)}|^{p-2} | T_{n,h} - \widetilde T_{n,h}^{(j,j)}|  H^{p-1}(v), \\
\mathcal A_{34} &\eqdef& \frac{p^ {p-1}}{n} \sum_{j=1}^n \sum_{\nu = 1}^{2} \E |\xi_{j \nu} \widetilde f_{j \nu}|  \E_{j,j}^{p-2} [|T_n - \widetilde T_n^{(j,j)}|H] | T_{n,h} - \widetilde T_{n,h}^{(j,j)}|   H(v) 
\end{eqnarray*}

It follows from these representations that \( \mathcal A_{31} \le \mathcal A_{311} + \ldots + \mathcal A_{313} \), where
\begin{eqnarray*}
\mathcal A_{311} &\eqdef& \frac{p}{n} \sum_{j=1}^n \sum_{\nu = 1}^{2} \E |\xi_{j \nu} \widetilde f_{j \nu}| |\widetilde T_{n,h}^{(j,j)}|^{p-2} | T_{n} - \widetilde T_n^{(j,j)}| I(v), \\
\mathcal A_{312} &\eqdef& \frac{p}{n} \sum_{j=1}^n \sum_{\nu = 1}^{2} \E |\xi_{j \nu} \widetilde f_{j \nu}| |\widetilde T_{n,h}^{(j,j)}|^{p-2} |T_{n,h}| |H - \widetilde H^{(j,j)}|, \\
\mathcal A_{313} &\eqdef& \frac{p}{n} \sum_{j=1}^n \sum_{\nu = 1}^{2} \E |\xi_{j \nu} \widetilde f_{j \nu}| |\widetilde T _{n,h}^{(j,j)}|^{p-2}  \E_{j,j} [(T_n - \widetilde T_n^{(j,j)})H] I(v),
\end{eqnarray*}
Let us consider the term \( \mathcal A_{311}\). We may apply Lemma~\ref{T bound}  and bound this term by the sum of three terms:
\begin{eqnarray*}
\mathcal A_{311} &\le&   \frac{p \imag s}{n} \sum_{\alpha=1}^{2m}\sum_{j=1}^n \sum_{\nu = 1}^{2} \E |\xi_{j \nu} \widetilde f_{j \nu}| |\widetilde T _{n,h}|^{p-2} | \Lambda_{n} - \widetilde \Lambda_{n}^{(\alpha, j,j)}| I(v)   \\
&+& \frac{p}{n (nv)^2} \sum_{j=1}^n \sum_{\nu = 1}^{2} \E |\xi_{j \nu} \widetilde f_{j \nu}|  |\widetilde T _{n,h}|^{p-2} \bigg[ \frac{\imag^2 \RR_{j j}}{|\RR_{j j}|^2}   + \frac{\imag^2 \RR_{j_{m+1}, j_{m+ 1}}^{(j)}}{|\RR_{j_{m+1}, j_{m+1}}^{(j)}|^2} \bigg]I(v) \\
&+& \frac{p}{n (nv)^2} \sum_{j=1}^n \sum_{\nu = 1}^{2} \E |\xi_{j \nu} \widetilde f_{j \nu}|  |\widetilde T _{n,h}|^{p-2} \E_{j,j}\bigg[ \frac{\imag^2 \RR_{jj}}{|\RR_{jj}|^2} +  \frac{\imag^2 \RR_{j_{m+1}, j_{m+ 1}}^{(j)}}{|\RR_{j_{m+1}, j_{m+1}}^{(j)}|^2} \bigg]I(v) \\
&\eqdef& \mathcal A_{311}^{(1)} + \mathcal A_{311}^{(2)} + \mathcal A_{311}^{(3)}.
\end{eqnarray*}
The last two terms, \( \mathcal A_{311}^{(2)},  \mathcal A_{311}^{(3)}\), may be easily bounded  as follows
\begin{eqnarray*}
	\mathcal A_{311}^{(j)} \le \frac{C^p p^\frac p2 \mathcal E^p(\kappa p)}{(nv)^p} + \rho \E|T_{n,h}|^p, \quad j = 2, 3.
\end{eqnarray*}
Indeed, one may apply H\"older's inequality and Young's inequality. For the estimation of \(  \mathcal A_{311}^{(1)} \) we first use Young's inequality and get
\begin{eqnarray*}
\mathcal A_{311}^{(1)} \le \rho \E|T_{n,h}|^p + \frac{C^p p^\frac p2 \imag^\frac p2 s }{n}\sum_{j=1}^n \sum_{\alpha=1}^{2m} \sum_{\nu = 1}^{2} \E^\frac12 \big [\E_{j,j} (|\xi_{j \nu}| | \Lambda_{n}^{(\alpha)} - \widetilde \Lambda_{n}^{(\alpha, j, j)}| I(v)) \big]^p. 
\end{eqnarray*}  
Applying Lemma~\ref{lem: important lemma 1} we get
\begin{eqnarray*}
	\mathcal A_{311}^{(1)} \le \rho \E|T_{n,h}|^p + \frac{C^p p^\frac p2 \mathcal E^p(\kappa p)}{(nv)^p}. 
\end{eqnarray*}
It is easy to see that similarly one may estimate the term \( \mathcal A_{313} \). To finish estimation of \(\mathcal A_{31}\) it remains to estimate \(\mathcal A_{312}\). Applying Lemma~\ref{approximation1} we obtain
\begin{eqnarray*}
\mathcal A_{312} &\lesssim& \frac{p}{n} \sum_{j=1}^n \sum_{\nu = 1}^{2} \sum_{\alpha=1}^{2m} \sum_{k=0}^{K_v} \frac{1}{\imag s(z,v_k)} \E  |\xi_{j \nu} \widetilde f_{j \nu}| |\widetilde T_{n,h}^{(j,j)}|^{p-2} |T_{n,h}| | \Lambda_n^{(\alpha)}(v_k) - \widetilde \Lambda_n^{(\alpha, j,j)}(v_k)| I(v). 
\end{eqnarray*}
Applying 
Let us denote 
\begin{eqnarray*}
I_{j, k, \alpha, \nu} \eqdef \mathcal \E  |\xi_{j \nu} \widetilde f_{j \nu}| |\widetilde T_{n,h}^{(j,j)}|^{p-2} |T_{n,h}| | \Lambda_n^{(\alpha)}(v_k) - \widetilde \Lambda_n^{(\alpha, j, j)}(v_k)| I(v). 
\end{eqnarray*}
Using the Cauchy-Schwartz inequality we obtain
\begin{eqnarray*}
I_{j, k, \alpha, \nu} &\le& \mathcal \E  |\widetilde T_{n,h}^{(j,j)}|^{p-2} |\widetilde f_{j \nu}| I(v) \E_{j,j}[|\xi_{j \nu} | |T_{n,h}| | \Lambda_n^{(\alpha)}(v_k) - \widetilde \Lambda_n^{(\alpha, j,j)}(v_k)| I(v)] \\
&\le&  \E  |\widetilde T_{n,h}^{(j,j)}|^{p-2} |\widetilde f_{j \nu}| I(v) \E_{j,j}^\frac12[|\xi_{j \nu} I(v)]^2 \E_{j,j}^\frac14 |T_{n,h}|^4 \E_{j,j}^\frac14[| \Lambda_n^{(\alpha)}(v_k) - \widetilde \Lambda_n^{(\alpha, j, j)}(v_k)| I(v)]^4 \\
&\le& \E^\frac{p-1}{p} |T_{n,h}|^p \E^\frac1{2p} \E_{j,j}^p[|\xi_{j \nu} I(v)]^2 \E^\frac1{4p} \E_{j,j}^p[| \Lambda_n^{(\alpha)}(v_k) - \widetilde \Lambda_n^{(\alpha, j,j)}(v_k)| I(v)]^4.
\end{eqnarray*}
It is easy to show that
\begin{eqnarray}\label{lambda diff} 
| \Lambda_n^{(\alpha)}(v_k) - \Lambda_n^{(\alpha, j,j)}(v_k) | &\le& \frac{1}{n} \bigg[ \delta_{\alpha, 1} |\RR_{j j}(v_k)| +\delta_{\alpha, m+1} |\RR_{j _{m+1} j_{m+1}}^ {(j)}(v_k)|  \nonumber \\
&+&  \frac{1}{v_k} \frac{\imag  \RR_{j j}(v_k) }{|\RR_{j j}(v_k)| } +  \frac{1}{v_k} \frac{\imag  \RR_{j_{m+1} j_{m+1}}^{(j)}(v_k) }{|\RR_{j_{m+1} j_{m+1}}^{(j)}(v_k)| }  \bigg].
\end{eqnarray}
We may use Lemmas~\ref{e2}--\ref{e4-e5}  and Young's inequality to get
\begin{eqnarray*}
\mathcal A_{312} \le  \rho \E|T_{n,h}|^p + \frac{C^p p^{p} \mathcal E^p(\kappa p)}{(nv)^p}. 
\end{eqnarray*}
Let us consider \(\mathcal A_{32}\). Applying we may may etimate it by the sum of the following terms
\begin{eqnarray*}
\mathcal A_{321} &\eqdef& \frac{p^{p-1}}{n} \sum_{j=1}^n \sum_{\nu = 1}^{2} \E |\xi_{j \nu} \widetilde f_{j \nu}| |T_n|^{p-2} |H - \widetilde H^{(j,j)}|^{p-2} | T_{n} - \widetilde T_{n}^{(j,j)}| \widetilde H^{(j,j)} H(v), \\
\mathcal A_{322} &\eqdef& \frac{p^{p-1}}{n} \sum_{j=1}^n \sum_{\nu = 1}^{2} \E |\xi_{j \nu} \widetilde f_{j \nu}| |T_n|^{p-1} |H - \widetilde H^{(j,j)}|^{p-1} H(v), \\
\mathcal A_{321} &\eqdef& \frac{p^{p-1}}{n} \sum_{j=1}^n \sum_{\nu = 1}^{2} \E |\xi_{j \nu} \widetilde f_{j \nu}| |T_n|^{p-2} |H - \widetilde H^{(j,j)}|^{p-1}  \E_{j,j} [(T_n - \widetilde T_n^{(j,j)})H] H(v).
\end{eqnarray*}
All three terms may be bounded similarly. We turn our attention to the first term only. To deal with it we use the following inequality
\begin{eqnarray*}
	|T_{n}|^{p-2}I(v) \le C^p \imag^{p-2} s(z,w),
\end{eqnarray*}
which may be deduced from equation~\eqref{linear system}.  Hence,
\begin{eqnarray*}
\mathcal A_{321} &\le & \frac{C^p p^{p-1} \imag^{p-2} s(v)}{n} \sum_{j=1}^n \sum_{\nu = 1}^{2} \E |\xi_{j \nu} \widetilde f_{j \nu}| |H - \widetilde H^{(j,j)}|^{p-2} | T_{n} - \widetilde T_{n}^{(j,j)}| \widetilde H^{(j,j)} H(v), 
\end{eqnarray*}
Using Lemma~\ref{approximation1} we get
\begin{eqnarray*}
\mathcal A_{321} \le  \frac{C^p p^{2p-1} \imag^{p-2} s( v)}{n} \sum_{j=1}^n \sum_{\nu = 1}^{2} \sum_{\alpha=1}^{2m}\sum_{k=0}^{K_v}\frac{\E |\xi_{j \nu} \widetilde f_{j \nu}| |\Lambda_n^{(\alpha)}(v_k) - \widetilde \Lambda_n^{(\alpha, j,j)}(v_k)|^{p-2} | T_{n} - \widetilde T_{n}^{(j,j)}| I(v)}{s^k \imag^{p-2} s(v_k)}. 
\end{eqnarray*}
It follows from~\eqref{lambda diff} and Lemma~\ref{lem: important lemma 1} that
\begin{eqnarray*}
\mathcal A_{321} \le \frac{C^p p^{2p} \mathcal E^p(\kappa p)}{(nv)^p}.
\end{eqnarray*}
Let us consider \(\mathcal A_{33}\). Applying~\eqref{T_n tilde 1}--\eqref{T_n tilde 2} we get, where 
\begin{eqnarray*}
\mathcal A_{331} &=& \frac{p^ {p-1}}{n} \sum_{j=1}^n \sum_{\nu = 1}^{2} \E |\xi_{j \nu} \widetilde f_{j \nu}|  | T_{n} - \widetilde T_{n}^{(j,j)}|^{p-1}  H(v), \\
\mathcal A_{332} &=& \frac{p^ {p-1}}{n} \sum_{j=1}^n \sum_{\nu = 1}^{2} \E |\xi_{j \nu} \widetilde f_{j \nu}|  | T_{n} - \widetilde T_{n}^{(j,j)}|^{p-2} |H - \widetilde H^{(j,j)}|  |T_{n,h}|, \\  
\mathcal A_{333} &=& \frac{p^ {p-1}}{n} \sum_{j=1}^n \sum_{\nu = 1}^{2} \E |\xi_{j \nu} \widetilde f_{j \nu}|| T_{n} - \widetilde T_{n}^{(j,j)}|^{p-2}  \E_{j,j} [(T_n - \widetilde T_n^{(j,j)})H] H(v).
\end{eqnarray*}
We estimate \( \mathcal A_{331} \) only. All other terms may be bounded similarly. We get
\begin{eqnarray*}
\mathcal A_{331} &\le& \frac{C^p p^ {p-1} \imag^{p-2} s(v) }{n (nv)^{p-2}} \sum_{j=1}^n \sum_{\nu = 1}^{2} \E |\xi_{j \nu} \widetilde f_{j \nu}|  | T_{n} - \widetilde T_{n}^{(j,j)}|  H(v),
\end{eqnarray*}
Applying Lemmas~\ref{lem: important lemma 1} and~\ref{T bound}  we get
\begin{eqnarray*}
\mathcal A_{331} \le \frac{C^p p^{p} \mathcal E^p(\kappa p)}{(nv)^p}.
\end{eqnarray*}
The term \(\mathcal A_{34}\) may be estimated similarly to  \(\mathcal A_{33}\). We omit the details.
Collecting all bounds we obtain the following estimate for \(\mathcal A_3\):
\begin{eqnarray*}
\mathcal A_3 \le \rho \E|T_{n,h}|^p + \frac{C^p p^{2p} \mathcal E^p(\kappa p)}{(nv)^p}. 
\end{eqnarray*}

\subsubsection{Bound for \( \mathcal A_4 \) }  Recall that
\begin{eqnarray*}
\mathcal A_4 &=& \frac{1}{n} \sum_{j=1}^n \sum_{\nu = 1}^{2} \E \xi_{j \nu} [ f_{j \nu} - \widetilde  f_{j \nu} ] H(v) \varphi( T_{n,h}).
\end{eqnarray*}
We may estimate it as follows
\begin{eqnarray*}
\mathcal A_{41} &\eqdef& \frac{e}{n} \sum_{j=1}^n \sum_{\nu = 1}^{2} \E |\xi_{j \nu}| | f_{j \nu} - \widetilde  f_{j \nu} | | \widetilde T_{n,h}^{(j,j)}|^{p-1} H(v), \\
\mathcal A_{42} &\eqdef& \frac{p^{p-1}}{n} \sum_{j=1}^n \sum_{\nu = 1}^{2} \E |\xi_{j \nu}| | f_{j \nu} - \widetilde  f_{j \nu} || T_{n,h} -\widetilde T_{n,h}^{(j,j)}|^{p-1} H(v). \\
\end{eqnarray*}
We may choose 
\begin{eqnarray*}
\widetilde  f_{j 1} &\eqdef& -\frac{1}{-w-m_n^{(m+2, j,j)}(z,w)+ \frac{|z|^2}{w + m_n^{(m, j,j)}(z,w)} }, \\
\widetilde f_{j2} &\eqdef& - \frac{|z|^2\widetilde  f_{j 1}}{[w + m_n^{(m,j,j)}(z,w)]^2}
\end{eqnarray*}
To estimate the difference \(  f_{j \nu} - \widetilde  f_{j \nu}\) we may apply representation~\eqref{rj} and inequality~\eqref{lambda diff}. Repeating all arguments from the previous section we may conclude the bound
\begin{eqnarray*}
	\mathcal A_4 \le \rho \E|T_{n,h}|^p + \frac{C^p p^{2p} \mathcal E^p(\kappa p)}{(nv)^p}. 
\end{eqnarray*}

Collecting now all bounds above we get the claim of the theorem.
\end{proof}

\subsection{Auxiliary lemmas}
We finish this section by several important lemmas. 
\begin{lemma}\label{approximation1}
Let \(v_k = v s^k, k \geq 0\). The following inequality holds
\begin{eqnarray*}
|H(\Lam_n) - H(\widetilde \Lam_n^{(j_\alpha, j_\alpha)})  |  
&\le& \frac{1}{\tau} \sum_{\beta=1}^{2m} \sum_{k=0}^{K_v} \frac{1}{\imag s(z,v_k)} | \Lambda_n^{(\beta)}(v_k) - \widetilde \Lambda_n^{(\beta, j_\alpha,j_\alpha)}(v_k)| I(v).
\end{eqnarray*}
The same is true if one replaces \( \widetilde \Lam_n^{(j_\alpha, j_\alpha)}\) by  \(\Lam_n^{(j_\alpha, j_\alpha)}\).
\end{lemma}
\begin{proof}The proof follows from the simple inequality \( |\prod_{j=1}^n  a_j - \prod_{j=1}^n b_j| \le \sum_{j=1}^n |a_j - b_j| \) and direct calculations. 
\end{proof}

Denote \( I(v, v') \eqdef I(v) I(v')\). 
\begin{lemma}\label{lem: important lemma 1}
Let \(w = u+iv \in \mathcal D, w' = u + i v' \in \mathcal D \). Moreover, we assume that \( v' \geq v\). Let \(g_{1j_\alpha}(w,w'), g_{2j_\alpha}(w,w')\) be some positive r.v. such that \( \E|g_{kj_\alpha}|^q < \infty, \, k = 1, 2\), for \(1 \le q \le C \log n\). Then for any \(\alpha = 1, \ldots, m\) and \(j = 1, \ldots, n\)
\begin{eqnarray*}
\max_{ 1 \le \beta \le 2m }\E_{j_\alpha,j_\alpha}\big[ |\xi_{j_\alpha \nu}(v)|  |\Lambda_n^{(\beta)}(v')  - \widetilde \Lambda_n^{(\beta, j_\alpha, j_\alpha)}(v')| g_{1j_\alpha}(v,v')  I(v,v')   \big] \lesssim \frac{A_{j_\alpha}^{1/2}(v) B_{j_\alpha}^{1/2}(v') g_{2j_\alpha}(v,v')}{(nv)^{1/2}(nv')^{3/2}}, 
\end{eqnarray*}
where
\begin{eqnarray*}
A_{j_\alpha}(v) &\eqdef& \max \bigg\{ \imag s(v), \E_{j_\alpha,j_\alpha}^\frac14 [\imag^4 \RR_{j_{m+[\alpha+1]}, j_{ m+ [\alpha+1]}}^{(j_\alpha)}  I(v)] \bigg\}, \\
B_{j_\alpha}(v') &\eqdef& \max \bigg\{ \imag s(v'), \E_{j_\alpha,j_\alpha}^\frac{1}{4\beta} [\imag^{4\beta} \RR_{j_{m+\alpha}, j_{m+\alpha}}^{(j_\alpha)}  I(v')], \E_{j_\alpha,j_\alpha}^\frac{1}{4\beta} [\imag^{4\beta} \RR_{j_{m+[\alpha+1]}, j_{ m+ [\alpha+1]}}^{(j_\alpha)}  I(v')], \\
&&\qquad \quad \E_{j_\alpha,j_\alpha}^\frac{1}{4\beta} [\imag^{4\beta} \RR_{j_{[\alpha-1]}, j_{[\alpha-1]}}^{(j_{m+\alpha})}  I(v')], \E_{j_\alpha,j_\alpha}^\frac{1}{4\beta} [\imag^{4\beta} \RR_{j_\alpha j_{\alpha}}^{(j_{m+\alpha})}  I(v')]   \bigg\}.
\end{eqnarray*}
\end{lemma}
\begin{proof}
We start from the representation for \( \Lambda_n^{(\beta)}  - \widetilde \Lambda_n^{(\beta, j_\alpha, j_\alpha)}\).  We rewrite it as follows
\begin{eqnarray*}
\Lambda_n^{(\beta)}  - \widetilde \Lambda_n^{(\beta, j_\alpha, j_\alpha)} = \Lambda_n^{(\beta)}  - \widetilde \Lambda_n^{(\beta, j_\alpha)} + \widetilde \Lambda_n^{(\beta, j_\alpha)} - \widetilde \Lambda_n^{(\beta, j_\alpha, j_\alpha)}. 
\end{eqnarray*}
Let us introduce the following notations: 
\begin{eqnarray*}
I_1 &\eqdef& \E_{j_\alpha,j_\alpha}\big[ |\xi_{j_\alpha \nu}(v)|  |\Lambda_n^{(\beta)}(v')  - \widetilde \Lambda_n^{(\beta, j_\alpha)}(v')| |g_{1j_\alpha}(v,v')|  I(v, v')    \big], \\
I_2 &\eqdef& \E_{j_\alpha,j_\alpha}\big[ |\xi_{j_\alpha \nu}(v)|  |\widetilde \Lambda_n^{(\beta,j_\alpha)}(v')  - \widetilde \Lambda_n^{(\beta, j_\alpha,j_\alpha}(v')| |g_{1j_\alpha}(v,v')|  I(v,v')  \big].
\end{eqnarray*}
We start from \(I_1\). We first mention that \(
\Lambda_n^{(\beta)}  - \widetilde \Lambda_n^{(\beta, j_\alpha)} = m_n^{(\beta)} - m_n^{(\beta, j_\alpha)}  - \E_{j_\alpha} (m_n^{(\beta)} - m_n^{(\beta, j_\alpha)} )\).
Let us consider \( m_n^{(\alpha)} - m_n^{(\alpha, j)}\) and rewrite it as follows
\begin{eqnarray*}
m_n^{(\beta)} - m_n^{(\beta, j_\alpha)} = \frac{1}{n} \big(  \delta_{\alpha, \beta}\RR_{j_\alpha j_\alpha} + \sum_{l \in \T \setminus J_{\beta}} ( \RR_{l_\beta l_\beta} - \RR_{l_\beta l_\beta}^{(j_\alpha)} )  \big).
\end{eqnarray*}
Writing down the decomposition for the diagonal entries of resolvent we get
\begin{eqnarray*}
m_n^{(\beta)} - m_n^{(\beta, j_\alpha)} &=& \frac{1}{n} \big(  \delta_{\beta \alpha}\RR_{j_\alpha j_\alpha} +  \sum_{l \in \T \setminus J_\beta} \RR_{j_\alpha j_\alpha}^{-1} [\RR_{j_\alpha l_\beta}]^2  \big) \nonumber\\
&=& \frac{1}{n}  \bigg (   \delta_{\beta \alpha}  + \sum_{l \in \T \setminus J_\beta}  \bigg[ \frac{1}{\sqrt n} \sum_{k=1}^n X_{jk}^{(\alpha)} \RR_{k_{m+[\alpha+1]}, l_\beta} ^{(j_\alpha)} - z \RR_{j_{m+1\alpha}, l_\beta}^{(j_\alpha)} \bigg ] ^2  \bigg ) \RR_{j_\alpha j_\alpha}\nonumber\\
&=& \frac{1}{n}  ( \delta_{\beta \alpha} + \eta_{j_\alpha}) \RR_{j_\alpha j_\alpha}, \label{eq: m_n repr 1}
\end{eqnarray*}
where \( \eta_{j_\alpha} \eqdef \eta_{j_\alpha 0} + \ldots + \eta_{j_\alpha 3} \) and
\begin{eqnarray*}
\eta_{j_\alpha 0} &\eqdef&  z^2 \sum_{l \in \T \setminus J_\beta}   [\RR_{j_{m+\alpha}, l_\beta}^{(j_\alpha)}]^2 - \frac{1}{n} \sum_{k=1}^n \mathcal R_{k k}^{(j_\alpha)},  \quad 
\eta_{j_\alpha 1} \eqdef \frac{1}{n} \sum_{k \neq k' } X_{jk}^{(\alpha)} X_{jk'}^{(\alpha)}  \mathcal R_{k k'}^{(j_\alpha)} , \\
\eta_{j_\alpha 2} &\eqdef& \frac{1}{n} \sum_{k =1 }^n [X_{jk}^{(1)}]^2 - 1]  \mathcal R_{k k}^{(j_\alpha)}, \qquad\qquad\qquad
\eta_{j_\alpha 3} \eqdef -\frac{2 z }{n} \sum_{k =1 }^n X_{jk}^{(\alpha)}   \mathcal R_{k j}^{(j_\alpha)}. 
\end{eqnarray*}
Here \( \mathcal R_{k, k'}^{(j_\alpha)} \eqdef    \sum_{l \in \T \setminus J_\beta}   \RR_{k_{m+[\alpha+1]}, l_\beta} ^{(j_\alpha)}  \RR_{k_{m+[\alpha+1]}', l_\beta} ^{(j_\alpha)} \).
Using this representation we write
\begin{eqnarray*}
\Lambda_n^{(\beta)}  - \widetilde \Lambda_n^{(\beta, j_\alpha)}  = \frac{1}{n}(\delta_{\beta \alpha}+ \eta_{j_\alpha 0}) [\RR_{j_\alpha j_\alpha} - \E_{j_\alpha}(\RR_{j_\alpha j_\alpha})]  + \frac{1}{n} \widetilde \eta_{j_\alpha} \RR_{j_\alpha j_\alpha} - \frac{1}{n} \E_{j_\alpha}(\widetilde \eta_{j_\alpha} \RR_{j_\alpha j_\alpha}),
\end{eqnarray*}
where  \( \widetilde \eta_{j_\alpha} \eqdef \eta_{j_\alpha 1} + \ldots + \eta_{j_\alpha 3}\) and \(g_{1, j_\alpha}\). Moreover, it is straightforward to check that
\begin{eqnarray*}
\frac{1}{n}(\delta_{\beta \alpha}+ |\eta_{j_\alpha 0}(v')|) I(v') \le \frac{C}{nv'} (\imag s(z,v') + \imag \RR_{j_{m+\alpha}, j_{m+\alpha}}^{(j_\alpha)}(v'))I(v'). 
\end{eqnarray*}
Introduce the following approximation for \(\RR_{j_\alpha j_\alpha} \):
\begin{eqnarray*}
Q_{j_\alpha} \eqdef -\frac{1}{-w'-m_n^{(m+[\alpha+1], j_\alpha)}(z,v')+ \frac{|z|^2}{w' + m_n^{([\alpha-1], j_\alpha)}(z,v')} }.
\end{eqnarray*}
Then \(\RR_{j_\alpha j_\alpha} - Q_{j_\alpha} = \xi_{j_\alpha 1} f_{j_\alpha 1} + \xi_{j_\alpha 2} f_{j_\alpha 2}\). We conclude from the previous facts that
\begin{eqnarray*}
I_1  &\le& \frac{1}{n}  \E_{j_\alpha,j_\alpha}\big [(\delta_{1\alpha}+ |\eta_{j_\alpha 0}(v')|) |\xi_{j_\alpha \nu}(v) \xi_{j_\alpha1}(v')| g_{1j_\alpha}'(v,v')  I(v ,v')   \big] \\
&+& \frac{1}{n}  \E_{j_\alpha,j_\alpha}\big [(\delta_{1\alpha}+ |\eta_{j_\alpha0}(v')|) |\xi_{j_\alpha \nu}(v) \xi_{j_\alpha 2}(v')| g_{1j_\alpha}'(v,v')  I(v ,v')   \big] \\
&+& \frac{1}{n}  \E_{j_\alpha,j_\alpha}\big [(\delta_{1\alpha}+ |\eta_{j_\alpha0}(v')|) |\xi_{j_\alpha \nu}(v)| \E_{j_\alpha}(|\xi_{j_\alpha 1}(v') f_{j_\alpha1}(v')|)  g_{1j_\alpha}'(v,v')  I(v ,v')   \big] \\
&+& \frac{1}{n}  \E_{j_\alpha,j_\alpha}\big [(\delta_{1\alpha}+ |\eta_{j_\alpha0}(v')|) |\xi_{j_\alpha \nu}(v)| \E_{j_\alpha}(|\xi_{j_\alpha 2}(v') f_{j_\alpha2}(v')|)  g_{1j_\alpha}'(v,v')  I(v ,v')   \big] \\
&+& \frac{1}{n}  \E_{j_\alpha,j_\alpha}\big[ |\xi_{j_\alpha \nu}(v) \widetilde \eta_{j_\alpha}(v')| g_{1j_\alpha}'(v,v')  I(v ,v')   \big] \\
&+& \frac{1}{n}  \E_{j_\alpha,j_\alpha}\big[ |\xi_{j_\alpha \nu}(v)|  \E_{j_\alpha}(|\widetilde \eta_{j_\alpha}(v') \RR_{j_\alpha j_\alpha}(v')|)  g_{1j_\alpha}(v,v')  I(v ,v')   \big].
\end{eqnarray*}
Here  \(g_{1, j_\alpha}\) is some positive function for bounded moments ap to the order \(C \log n\). 

All terms in the upper bound for \(I_1\) may be estimated directly. We show how to deal with one term only, all other terms may be estimated similarly. For example, we estimate the second last term. 
H\"older's inequality implies   
\begin{eqnarray*}
\frac{1}{n}  \E_{j_\alpha,j_\alpha}\big[ |\xi_{j_\alpha \nu}(v) \widetilde \eta_{j_\alpha}(v')| g_{1j_\alpha}'(v,v')  I(v ,v')   \big]  &\le& \frac{1}{n} \E_{j_\alpha,j_\alpha}^\frac12(|\xi_{j_\alpha \nu}(v)|^2I(v)) \E_{j_\alpha,j_\alpha}^\frac1{2\beta}(|\widetilde \eta_{j_\alpha}(v')|^{2\beta}  I(v') ) \\
&\times& \E_{j_\alpha,j_\alpha}^\frac{\beta-1}{2\beta}( |g_{1j_\alpha}'(v,v')|^\frac{2\beta}{\beta-1}  I(v ,v')).
\end{eqnarray*}
Here, \( \beta \eqdef (4+\delta)/4 \). It remains to apply Lemmas~\ref{lem: eta 2}--\ref{lem: eta 4,5} to get the desired bound as stated in lemma. 

Let us consider \(I_2\). It is easy to check that
\begin{eqnarray*}
\widetilde \Lambda_n^{(\beta,j_\alpha)}  - \widetilde \Lambda_n^{(\beta, j_\alpha,j_\alpha} = \E_{j_\alpha} (m_n^{(\beta)} - m_n^{(\beta, j_{m+\alpha}) })) - \E_{j_\alpha, j_\alpha}(m_n^{(\beta)} - m_n^{(\beta, j_{m+\alpha})}) 
\end{eqnarray*}
Similarly to~\eqref{eq: m_n repr 1} we may show that
\begin{eqnarray*}
m_n^{(\beta)} - m_n^{(\beta, j_{m+\alpha})} &=& \frac{1}{n} \big(  \delta_{\beta , m+ \alpha} \RR_{j_{m+\alpha} j_{m+\alpha}} +  \sum_{l \in \T \setminus J_\beta} \RR_{j_{m+\alpha}, j_{m+\alpha}}^{-1} [\RR_{j_{m+\alpha} l_\beta}]^2  \big) \\
&=& \frac{1}{n}  \bigg (   \delta_{\beta, m+ \alpha}  + \sum_{l \in \T \setminus J_\beta}  \bigg[ \frac{1}{\sqrt n} \sum_{k=1}^n X_{jk}^{([\alpha-1])} \RR_{k_{[\alpha-1]}, l_\beta} ^{(j_{m+\alpha})} - \overline z \RR_{j_{\alpha}, l_\beta}^{(j_{m+\alpha})} \bigg ] ^2  \bigg ) \RR_{j_{m+\alpha}, j_{m+\alpha}}\\
&=& \frac{1}{n}  \bigg (   \delta_{\beta, m+ \alpha}  + \sum_{l \in \T \setminus J_\beta}  \bigg[ \frac{1}{\sqrt n} \sum_{k=1}^n X_{jk}^{([\alpha-1])} \RR_{k_{[\alpha-1]}, l_\beta} ^{(j_\alpha, j_{\alpha})} - \overline z \RR_{j_{\alpha}, l_\beta}^{(j_{m+\alpha})} \bigg ] ^2  \bigg ) \RR_{j_{m+\alpha}, j_{m+\alpha}}\\
&+&  \frac{1}{n} \sum_{l \in \T \setminus J_\beta} \Theta_{l}  \RR_{j_{m+\alpha}, j_{m+\alpha}} 
=\widehat \eta_{j_\alpha } +  \frac{1}{n} \sum_{l \in \T \setminus J_\beta} \Theta_{l}  \RR_{j_{m+\alpha}, j_{m+\alpha}}.
\end{eqnarray*}
Here, \( \widehat \eta_{j_{m+\alpha}} \eqdef \widehat \eta_{j_{m+\alpha}, 0} + \ldots + \widehat \eta_{j_{m+\alpha} 4} \), 
\begin{eqnarray*}
\widehat \eta_{j_{m+\alpha}, 0} &\eqdef&  \overline z^2 \sum_{l \in \T \setminus J_\beta}   \E_{j_\alpha}[\RR_{j_{\alpha} l_\beta}^{(j_{m+\alpha})}]^2 - \frac{1}{n} \sum_{k=1}^n \mathcal R_{k k}^{(j_\alpha j_\alpha)},  \quad \widehat \eta_{j_{m+\alpha}, 1} \eqdef  \overline z^2 \sum_{l \in \T \setminus J_\beta}  [[\RR_{j_{\alpha} l_\beta}^{(j_{m+\alpha})}]^2 - \E_{j_\alpha}[\RR_{j_{\alpha} l_\beta}^{(j_{m+\alpha})}]^2 ], \\
\widehat \eta_{j_{m+\alpha}, 2} &\eqdef& \frac{1}{n} \sum_{k \neq k' } X_{jk}^{([\alpha-1])} X_{jk'}^{([\alpha-1])}  \mathcal R_{k k'}^{(j_\alpha j_\alpha)}, \\
\widehat \eta_{j_{m+\alpha}, 3} &\eqdef& \frac{1}{n} \sum_{k =1 }^n [X_{jk}^{([\alpha-1])}]^2 - 1]  \mathcal R_{k k}^{(j_\alpha, j_\alpha)}, \\
\widehat \eta_{j_{m+\alpha}, 4} &\eqdef& -\frac{2 \overline z }{n^{3/2}} \sum_{k =1 }^n X_{jk}^{([\alpha-1])}    \sum_{l \in \T \setminus J_\beta}  \RR_{k_{[\alpha-1]}, l_\beta} ^{(j_\alpha, j_\alpha)} \RR_{j_{\alpha} l_\beta}^{(j_{m+\alpha})}.
\end{eqnarray*}
and \( \mathcal R_{k k'}^{(j_\alpha)} \eqdef    \sum_{l \in \T \setminus J_\beta}   \RR_{k_{[\alpha-1]}, l_\beta} ^{(j_\alpha, j_\alpha)}  \RR_{k_{[\alpha-1]}', l_\beta} ^{(j_\alpha, j_\alpha)} \). 
Moreover, 
\begin{eqnarray*}
	\Theta_{l}  &\eqdef& \bigg[\frac{1}{\sqrt n} \sum_{k=1}^n X_{jk}^{([\alpha-1])} [\RR_{k_{[\alpha-1]}, l_\beta} ^{(j_{m+\alpha})} - \RR_{k_{[\alpha-1]}, l_\beta} ^{(j_\alpha, j_{\alpha})}  ]  \bigg]^2 \\
	&+& 2 \bigg[\frac{1}{\sqrt n} \sum_{k=1}^n X_{jk}^{([\alpha-1])} [\RR_{k_{[\alpha-1]}, l_\beta} ^{(j_{m+\alpha})} - \RR_{k_{[\alpha-1]}, l_\beta} ^{(j_\alpha, j_{\alpha})}   ] \bigg] 
	\bigg[ \frac{1}{\sqrt n} \sum_{k=1}^n X_{jk}^{([\alpha-1])} \RR_{k_{[\alpha-1]}, l_\beta} ^{(j_\alpha, j_{\alpha})} - \overline z  \RR_{j_{\alpha}, l_\beta}^{(j_{m+\alpha})} \bigg ].
\end{eqnarray*}
It is easy to see that \( \widehat \eta_{j_{m+\alpha}, 0}  \) is \(\mathfrak M^{(j_\alpha, j_\alpha)}\)-measurable.  Hence, 
\begin{eqnarray*}
\widetilde \Lambda_n^{(\beta, j_\alpha)}  - \widetilde \Lambda_n^{(\beta, j_\alpha, j_\alpha)}  &=& \frac{1}{n}(\delta_{\beta, m+\alpha}+ \widehat \eta_{j_{m+\alpha}, 0}) [\E_{j_\alpha} \RR_{j_{m+\alpha}, j_{m+\alpha}} - \E_{j_{\alpha}, j_\alpha} \RR_{j_{m+\alpha} j_{m+\alpha}}] \\
&+& \frac{1}{n} \E_{j_\alpha}\overline\eta_{j_{m+\alpha}} \RR_{j_{m+\alpha}, j_{m+\alpha}} - \frac{1}{n} \E_{j_\alpha, j_\alpha}( \overline\eta_{j_{m+\alpha}}  \RR_{j_{m+\alpha}, j_{m+\alpha}}) \\
&+&  \frac{1}{n} \sum_{l \in \T \setminus J_\beta} \E_{j_\alpha} \Theta_{l}  \RR_{j_{m+\alpha}, j_{m+\alpha}} - \frac{1}{n} \sum_{l \in \T \setminus J_\beta} \E_{j_\alpha, j_\alpha} \Theta_{l}  \RR_{j_{m+\alpha}, j_{m+\alpha}}.
\end{eqnarray*}
Here \( \overline\eta_{j_{m+\alpha}} \eqdef \widehat \eta_{j_{m+\alpha},1} + \ldots + \widehat \eta_{j_{m+\alpha},4} \). All this facts imply the following bound for \(I_2\):
\begin{eqnarray*}
I_2  &\le& \frac{1}{n}  \E_{j_\alpha,j_\alpha}\big [|\xi_{j_\alpha \nu}(v)|  (\delta_{\beta, m+\alpha}+ |\widehat \eta_{j_{m+\alpha}, 0}(v')|)   \E_{j_\alpha}\big( |\xi_{j_{m+\alpha}, 1}(v') f_{j_{m+\alpha}, 1}(v')| \big)    g_{1j_\alpha}'(v,v')  I(v ,v')   \big] \\
&+& \frac{1}{n}  \E_{j_\alpha,j_\alpha}\big [|\xi_{j_\alpha \nu}(v)| (\delta_{\beta, m+\alpha}+ |\widehat \eta_{j_{m+\alpha}, 0}(v')|)    \E_{j_\alpha}\big(  |\xi_{j_{m+\alpha}, 2}(v') f_{j_{m+\alpha}, 2}(v')| \big)    g_{1j_\alpha}'(v,v')  I(v ,v')   \big]\\
&+& \frac{1}{n}  \E_{j_\alpha,j_\alpha}\big [|\xi_{j_\alpha \nu}(v)|  (\delta_{\beta, m+ \alpha}+ |\widehat \eta_{j_{m+\alpha}, 0}(v')|)   g_{1j_\alpha}'(v,v')  I(v ,v')   \big]  \E_{j_\alpha, j_\alpha }\big[|\xi_{j_{m+\alpha}, 1}(v') f_{j_{m+\alpha}, 1}(v')|I(v')\big]  \\
&+& \frac{1}{n}  \E_{j_\alpha,j_\alpha}\big [|\xi_{j_\alpha \nu}(v)|  (\delta_{\beta, m+\alpha}+ |\widehat \eta_{j_{m+\alpha}, 0}(v')|)   g_{1j_\alpha}'(v,v')  I(v ,v')   \big]  \E_{j_\alpha, j_\alpha }\big[|\xi_{j_{m+\alpha}, 2}(v') f_{j_{m+\alpha}, 2}(v')| I(v')\big]  \\
&+& \frac{1}{n}  \E_{j_\alpha,j_\alpha}\big[ |\xi_{j_\alpha \nu}(v) |  \E_{j_\alpha}( |\overline\eta_{j_{m+\alpha}}(v') \RR_{j_{m+\alpha}, j_{m+\alpha}}(v')|) g_{1j_\alpha}'(v,v')  I(v ,v')   \big] \\
&+& \frac{1}{n}  \E_{j_\alpha,j_\alpha}\big[ |\xi_{j_\alpha \nu}(v)| g_{1j_\alpha}(v,v')  I(v ,v')   \big] \E_{j_\alpha, j_\alpha}\big[|\overline\eta_{j_{m+\alpha}}(v') \RR_{j_{m+\alpha}, j_{m+\alpha}}(v')| I(v') \big]\\
&+& \frac{1}{n}  \sum_{l \in \T \setminus J_\beta}  \E_{j_\alpha,j_\alpha}\big[ |\xi_{j_\alpha \nu}(v)|  \E_{j_\alpha}\big( |\Theta_{l}(v')  \RR_{j_{m+\alpha}, j_{m+\alpha} }(v') |\big)   g_{1j_\alpha}(v,v')  I(v ,v')   \big]\\
&+& \frac{1}{n}  \sum_{l \in \T \setminus J_\beta}  \E_{j_\alpha,j_\alpha}\big[ |\xi_{j_\alpha \nu}(v)| g_{1j_\alpha}(v,v')  I(v ,v')   \big] \E_{j_\alpha, j_\alpha}\big[ |\Theta_{l}(v')  \RR_{j_{m+\alpha}, j_{m+\alpha} } (v') |I(v')\big].
\end{eqnarray*}
One may proceed similarly to the estimate of \(I_1\). We only consider the second last term.  It is straightforward to check that
\begin{eqnarray*}
|\Theta_{l}|  &\le&     \bigg[\frac{1}{\sqrt n} \sum_{k=1}^n X_{jk}^{([\alpha-1])} \RR_{k_{[\alpha-1]}, j_\alpha} ^{(j_{m+\alpha})}  \bigg]^2|\RR_{j_\alpha l_\beta}^{(j_{m+\alpha})}|^2 |\RR_{j_{\alpha} j_\alpha}^{(j_{m+\alpha)}}|^{-2} \\
&+& 2 \bigg|\frac{1}{\sqrt n} \sum_{k=1}^n X_{jk}^{([\alpha-1])} \RR_{k_{[\alpha-1]}, j_\alpha} ^{(j_{m+\alpha})} \bigg| \bigg| \frac{1}{\sqrt n} \sum_{k=1}^n X_{jk}^{([\alpha-1])} \RR_{k_{[\alpha-1]}, l_\beta} ^{(j_\alpha, j_{\alpha})} - \overline z \RR_{j_{\alpha}, l_\beta}^{(j_{m+\alpha})} \bigg | |\RR_{j_\alpha, l_\beta} ^{(j_{m+\alpha})}||\RR_{j_{m+\alpha} j_\alpha}^{(j_{m+\alpha)}}   |^{-1}.
\end{eqnarray*}
Let us introduce the following notations:
\begin{eqnarray*}
\zeta_{j_\alpha, 1} &\eqdef& \frac{1}{\sqrt n} \sum_{k=1}^n X_{jk}^{([\alpha-1])} \RR_{k_{[\alpha-1]}, j_\alpha} ^{(j_{m+\alpha})}, \quad \zeta_{j_\alpha, 2} \eqdef  \frac{1}{\sqrt n} \sum_{k=1}^n X_{jk}^{([\alpha-1])} \RR_{k_{[\alpha-1]}, l_\beta} ^{(j_\alpha, j_{\alpha})}.
\end{eqnarray*}
Using these notations we may obtain
\begin{eqnarray*}
\frac{1}{n}  \sum_{l \in \T \setminus J_\beta}  |\Theta_{l}| &\le& \frac{|\zeta_{j_\alpha, 1}|^2}{nv}  \imag \RR_{j_{\alpha} j_\alpha}^{(j_{m+\alpha)}}  |\RR_{j_{\alpha} j_\alpha}^{(j_{m+\alpha)}} |^{-2} + 2\frac{|\zeta_{j_\alpha, 1} |}{\sqrt{nv}}   \bigg | \frac{1}{n} \sum_{l \in \T \setminus J_\beta} |\zeta_{j_\alpha, 2}|^2\bigg|^{1/2}    \imag^{1/2}\RR_{j_\alpha, j_\alpha} ^{(j_{m+\alpha})}  | \RR_{j_{\alpha} j_\alpha}^{(j_{m+\alpha)}} |^{-1} \\
&+&  \frac{2|z| |\zeta_{j_\alpha, 1}|}{nv} \imag^{1/2}\RR_{j_{[\alpha-1]}, j_{[\alpha-1]}}^{(j_{m+\alpha})} \imag^{1/2}\RR_{j_\alpha j_\alpha} ^{(j_{m+\alpha})}       |\RR_{j_{\alpha} j_\alpha}^{(j_{m+\alpha)}} |^{-1}.
\end{eqnarray*}
Now we may estimate the second last term in the bound for \(I_2\). Using the previous inequality it may be estimated as the sum of three terms. We will estimate the second term only. It is easy to see that \( \zeta_{j_\alpha, 2} \) is \(\mathfrak M^{(j_\alpha)}\)-measurable. Hence,
\begin{eqnarray*}
&&2 \E_{j_\alpha,j_\alpha}\bigg[ |\xi_{j_\alpha \nu}(v)| \big(\frac{1}{n} \sum_{l \in \T \setminus J_\beta}|\zeta_{j_\alpha, 2}|^2 \big)^{1/2} \\
&&\qquad\quad\times \E_{j_\alpha}[|\zeta_{j_\alpha, 1}|   \imag^{1/2}\RR_{j_\alpha, j_\alpha} ^{(j_{m+\alpha})} | \RR_{j_{\alpha} j_\alpha}^{(j_{m+\alpha)}} |^{-1}  |\RR_{j_{m+\alpha}, j_{m+\alpha} } |]   |g_{1j_\alpha}(v,v')|  I(v ,v')   \bigg] \\
&&\le \E_{j_\alpha,j_\alpha}^{1/2\beta}\big[ |\xi_{j_\alpha \nu}(v)| I(v)\big]^{2\beta}  \E_{j_\alpha,j_\alpha}^\frac{2\beta-1}{2\beta}\big[ |g_{1j_\alpha}'(v,v')|^\frac{2\beta}{2\beta-1}  I(v ,v')\big] \\
&&\qquad\quad\times \E_{j_\alpha,j_\alpha}^{1/2}\bigg[ \frac{1}{n} \sum_{l \in \T \setminus J_\beta}|\zeta_{j_\alpha, 2}(v')|^2  \E_{j_\alpha}\big(|\zeta_{j_\alpha, 1}(v') |^2 \big)   \E_{j_\alpha}^{1/2} \big(\imag^2 \RR_{j_\alpha j_\alpha} ^{(j_{m+\alpha})}(v')\big)   \E_{j_\alpha}^{1/2} \big(g'(v,v')\big) I(v') \bigg]\\
&& \le \frac{1}{\sqrt{nv'}} \E_{j_\alpha,j_\alpha}^{1/2\beta}\big[ |\xi_{j_\alpha \nu}(v) I(v)|^{2\beta} \big]    \E_{j_\alpha,j_\alpha}^{1/4}\big[ |\zeta_{j_\alpha, 2}(v')|^4 I(v')\big]  \E_{j_\alpha,j_\alpha}^{1/4\beta}\big[ |\zeta_{j_\alpha, 1}(v')|^{4\beta} I(v') \big]  g_{2 j_\alpha}(v,v').
\end{eqnarray*}
Applying now Rosenthal's inequality to \(\zeta_{j_\alpha, k}, k = 1,2\), we conclude the bound as required by the statement of the lemma. 
\end{proof}

\begin{lemma}\label{T bound} 
For any \( w \in \mathcal D\) the following inequality holds:
\begin{eqnarray*}
\max_{1 \le \alpha \le m} |T_n^{(\alpha)} - T_n^{(\alpha, j_\alpha,j_\alpha)}|I(v) \lesssim \frac{\imag s }{nv}.
\end{eqnarray*}
Moreover, 
\begin{eqnarray*}
\max_{1 \le \alpha \le m} |T_n^{(\alpha)} - \widetilde T_n^{(\alpha, j_\alpha,j_\alpha)}|I(v) & \lesssim & \imag s(z,w)  \max_{1 \le \beta \le 2m}| \Lambda_n^{(\beta)} - \widetilde \Lambda_n^{(\beta,j_\alpha,j_\alpha)}| I(v)   \\
&+& \frac{1}{(nv)^2} \max_{1 \le \alpha \le m}  \bigg[\frac{\imag^2 \RR_{j_\alpha j_\alpha}}{|\RR_{j_\alpha j_\alpha}|^2}   + \frac{\imag^2 \RR_{j_{m+\alpha}, j_{m+ \alpha}}^{(j_\alpha)}}{|\RR_{j_{m+\alpha}, j_{m+\alpha}}^{(j_\alpha)}|^2} \bigg] I(v)\\
&+& \frac{1}{(nv)^2} \E_{j_\alpha,j_\alpha} \bigg[\max_{1 \le \alpha \le m}   \bigg( \frac{\imag^2 \RR_{j_\alpha j_\alpha}}{|\RR_{j_\alpha j_\alpha}|^2}  +  \frac{\imag^2 \RR_{j_{m+\alpha}, j_{m+ \alpha}}^{(j_\alpha)}}{|\RR_{j_{m+\alpha}, j_{m+\alpha}}^{(j_\alpha)}|^2} \bigg) I(v) \bigg].
\end{eqnarray*}
\end{lemma}
\begin{proof}
Both statements are consequence of the equation for \( \Lam_n\). 
\end{proof}

 \section{Acknowledgement}
A. Naumov and A. Tikhomirov would like to thank Laszlo Erd\"os for his hospitality in IST Austria and helpful discussions on the subject of this paper. 

F. G\"otze was supported by CRC 1283 "Taming uncertainty and profiting from randomness and low regularity in analysis, stochastics and their applications". A.~Naumov and A. Tikhomirov were supported by the Russian Academic Excellence Project 5-100 and Russian Science Foundation grant 18-11-00132 (results of section 3 were obtained under support of the grant).

\appendix
\section{Inequalities for linear and quadratic forms}

In this section we present some inequalities for linear and quadratic forms.

\subsection{Estimations of \(\ve_{j_\alpha}\), for \(j\in\mathbb T\)}
In this section we estimate the moments of \(\ve_{ j_\alpha, \nu}^{(\J, \K)}\) for \(\J, \K \in \mathbb T\) and \(\nu=1, \ldots 5\) and \(\eta_{\nu j}^{(\J,\K)}, \nu = 1, \ldots, 5\). In what follows for any \(\J \subset \T\) and \(j \in \TJ\) we denote \(\widetilde \J \eqdef \J \cup \{j\}\). We also introduce \(\sigma\)-algebra \(\mathfrak M^{(\J, \K)} \eqdef \sigma\{X_{kl}, k \in \J, l \in \K\}\), \(\J, \K \subset \T\), and denote 
\begin{eqnarray}\label{eq: sigma alg}
\E^*(\cdot) \eqdef \E(\cdot\big|\mathfrak M^{(\widetilde \J, \K)}).
\end{eqnarray}
\begin{lemma}\label{e1}
For any \(j \in \T \setminus J_\alpha\) 
\begin{eqnarray*}
|\widetilde \ve_{j1}^{(\J, \K)}|\le  (nv)^{-1}. 
\end{eqnarray*}	
\end{lemma}
\begin{proof}
Since
\begin{eqnarray*}
\RR_{k_[\alpha+1]+m, k_[\alpha+1]+m,}^{(\J, \K)} - \RR_{k_[\alpha+1]+m,  k_[\alpha+1]+m,}^{(\widetilde \J, \K)} =  [\RR_{k_[\alpha+1]+m, k_[\alpha+1]+m,}^{(\J, \K)}]^2 /\RR_{j_\alpha j_\alpha}^{(\J, \K)}
\end{eqnarray*}
we get that
\begin{eqnarray*}
|\widetilde \ve_{j1}^{(\J, \K)}| \le \frac{1}{nv} \imag \RR_{j_\alpha j_\alpha}^{(\J, \K)}|\RR_{j_\alpha j_\alpha}^{(\J, \K)}|^{-1} \le (nv)^{-1}. 
\end{eqnarray*}
Thus lemma \ref{e1} is proved.
\end{proof}

\begin{lemma}\label{e2}
There exist a positive constant \(C\) such that for any \(j \in \T \setminus J_\alpha\) and all \(p \geq 2\)
\begin{eqnarray*}
\E^*|\widetilde \ve_{j_\alpha, 2}^{(\J, \K)}|^p &\le&  \frac{C^p p^p}{(n v)^\frac p2} \imag^\frac p2 m_n^{(\widetilde \J, \K)}   +\mu_p \frac{C^p p^\frac{3p}{2}}{n^p v^{\frac p2}} \sum_{l \in \T\setminus K_{[\alpha+1]}}\imag^\frac p2 \RR_{l_{[\alpha +1 ] + m }, k_{[\alpha +1 ] + m }}^{(\widetilde \J, \K)} \\
&+&\mu_p^2 \frac{C^p p^{2p}}{n^p} \sum_{l, k \in \T\setminus K_{[\alpha+1]}} |\RR_{l_{[\alpha +1 ] + m }, k_{[\alpha +1 ] + m }}^{(\widetilde \J, \K)}|^p.
\end{eqnarray*}	
\end{lemma}
\begin{proof}
	Applying moment inequality for quadratic forms, e.g.~\cite{GineLatalaZinn2000}[Proposition~2.4] or~\cite{GotTikh2003}[Lemma A.1], and Lemma~\ref{lem: res inequalities} we get the proof.
\end{proof}
\begin{lemma}\label{e3}
	There exist a positive constant \(C\) such that for any \(j \in \T \setminus J_\alpha\) and all \(p \geq 2\)
	\begin{eqnarray*}
		\E^*|\widetilde \ve_{j_\alpha, 3}^{(\J, \K)}|^p  \le \frac{C^p p^{\frac p2}}{
			n^{\frac p2}} \Big(\frac1n\sum_{l \in \T\setminus K_{[\alpha+1]} }|\RR^{(\widetilde \J, \K)}_{l_{[\alpha + 1] + m},l_{[\alpha + 1] + m}}|^2\Big)^{\frac p2}+
		\mu_{2p} \frac{C^p p^p}{n^p}\sum_{l \in \T\setminus K_{[\alpha+1]}} |\RR^{(\widetilde \J, \K)}_{l_{[\alpha + 1] + m},l_{[\alpha + 1] + m}}|^p.
	\end{eqnarray*}
\end{lemma}
\begin{proof}
The proof follows from the Rosenthal type inequality for linear forms, e.g.~\cite{Rosenthal1970}[Theorem~3] and~\cite{JohnSchecttmanZinn1985}[Inequality~(A)] and Lemma~\ref{lem: res inequalities}.
\end{proof}
\begin{lemma}\label{e4-e5}
There exist a positive constant \(C\) such that for any \(j \in \T \setminus J_\alpha\) and all \(p \geq 2\)
\begin{eqnarray*}
\E^*|\widetilde \ve_{j_\alpha, 4}^{(\J, \K)}|^p  \le 
\frac{C^p|z|^p p^\frac p2}{(nv)^\frac p2} \imag^\frac p2 \RR^{(\widetilde \J, \K)}_{l_{[\alpha + 1] + m},l_{[\alpha + 1] + m}} +\mu_p\frac{C^p|z|^p p^p}{n^\frac p2}\sum_{l \in \T\setminus K_{[\alpha+1]}}|\RR^{(\widetilde \J, \K)}_{l_{[\alpha + 1] + m},l_{[\alpha + 1] + m}}|^p.
\end{eqnarray*}	
\end{lemma}
\begin{proof}
The proof is similar to the proof of previous lemma. 
\end{proof}
We also estimate the moments of \(\ve_{j+n,\nu}^{(\widetilde \J, \K)}\) for \(j \in \TK, \nu=1,2,3\). Similarly to~\eqref{eq: sigma alg} we define
\begin{eqnarray*}
\E^{**}(\cdot)\eqdef \E(\cdot\big|\mathfrak M^{(\widetilde \J, \widetilde \K)}).
\end{eqnarray*}
\begin{lemma}\label{e1+n}
For any \(j \in \T \setminus J_\alpha\) 
\begin{eqnarray*}
|\widehat \ve_{j_\alpha 1}^{(\widetilde \J, \K)}|^p\le  2(nv)^{-1}.
\end{eqnarray*}
\end{lemma}
\begin{proof}
Repeating the arguments of the proof of Lemma~\ref{e1} one gets the statement of this lemma. 
\end{proof}
\begin{lemma}\label{e2+n}
There exist a positive constant \(C\) such that for any \(j \in \T \setminus J_\alpha\) and all \(p \geq 2\)
\begin{eqnarray*}
\E^{**}|\widehat \ve_{j_\alpha 2}^{(\J, \K)}|^p&\le&  \frac{C^p p^\frac p2}{(nv)^\frac p2}  \imag^\frac p2 m_n^{(\widetilde \J, \widetilde \K)} 
+\mu_p\frac{C^p p^\frac{3p}2}{n^p v^\frac p2}\sum_{l  \in \T\setminus J_{[\alpha-1]} }\imag^\frac p2 \RR^{(\widetilde \J, \widetilde \K)}_{ll}\\
&+& \mu_p^2\frac{C^p p^{2p}}{n^p} \sum_{l, k \in \T\setminus J_{[\alpha-1]} } |\RR^{(\widetilde \J, \widetilde \K)}_{kl}|^p.	
\end{eqnarray*}	
\end{lemma}
\begin{proof}
Repeating the arguments of the proof of Lemma~\ref{e2} one gets the statement of this lemma. 
\end{proof}
\begin{lemma}\label{e3+n}
There exist a positive constant \(C\) such that for any \(j \in \T \setminus J_\alpha\) and all \(p \geq 2\)
\begin{eqnarray*}
\E^{**}|\widehat \ve_{j_\alpha 3}^{(\widetilde \J, \K)}|^p\le 
\frac{C^p p^\frac p2}{n^\frac p2} \Big(\frac1n\sum_{l   \in \T\setminus J_{[\alpha-1]} }|\RR^{(\widetilde \J, \widetilde \K)}_{l l}|^2\Big)^{\frac p2}+ \mu_{2p}\frac{C^p p^p}{n^p}\sum_{l  \in \T\setminus J_{[\alpha-1]} } |\RR^{(\widetilde \J, \widetilde \K)}_{l l}|^p .
\end{eqnarray*}
\end{lemma}
\begin{proof}
Repeating the arguments of the proof of Lemma~\ref{e3} one gets the statement of this lemma. 
\end{proof}

\begin{lemma} \label{lem: eta 1}
	For any \(j \in \T\)
\begin{eqnarray*}
|\eta_{j_\alpha 0}| \lesssim \frac{1}{v} (\imag m_n^{ (j_\alpha)} + |z|^2 \imag \RR_{j_{m+\alpha}, j_{m+\alpha}}^{(j_\alpha)} ).
\end{eqnarray*}
\end{lemma}
\begin{proof}
The proof follows from Lemma~\ref{lem: res inequalities}.
\end{proof}
\begin{lemma}\label{lem: eta 2}
Under conditions \(\Cond\) for any \(j \in \T \) and all \(p: 2 \le p \le 4\)
\begin{eqnarray*}
\E^* |\eta_{j_\alpha 1}|^p \lesssim  \frac{1}{(nv^3)^\frac p2} \imag^\frac p2 m_n^{(j_\alpha)}. 
\end{eqnarray*}
\end{lemma}
\begin{proof}
Applying moment inequality for quadratic forms, e.g.~\cite{GineLatalaZinn2000}[Proposition~2.4] or~\cite{GotTikh2003}[Lemma A.1], and Lemma~\ref{lem: res inequalities} we get the proof.
\end{proof}
\begin{lemma}\label{lem: eta 3}
Under conditions \(\Cond\) for any \(j \in \T\) and all \(p: 2 \le p \le \frac 1\alpha\)
\begin{eqnarray*}
\E^* |\eta_{j_\alpha, 2}|^p \lesssim \frac{1}{(nv)^\frac p2} \imag^\frac p2 m_n^{(j_\alpha)}. 
\end{eqnarray*}
\end{lemma}
\begin{proof}
The proof follows from the Rosenthal type inequality for linear forms, e.g. ~\cite{Rosenthal1970}[Theorem~3] and~\cite{JohnSchecttmanZinn1985}[Inequality~(A)] and Lemma~\ref{lem: res inequalities}.
\end{proof}
\begin{lemma}\label{lem: eta 4,5}
Under conditions \(\Cond\) for any  \(j \in \T\) and all \(p: 2 \le p \le 4\) 
\begin{eqnarray*}
\E^* |\eta_{j_\alpha3}|^p \lesssim \frac{|z|^p \imag^\frac p 2 \RR_{j_{m+\alpha}, j_{m+\alpha}}^{(j_\alpha)}}{n^\frac p2 v^p} \imag^\frac p2 m_n^{(j_\alpha)}.
\end{eqnarray*}
\end{lemma}
\begin{proof}
Similar to the proof of previous lemma.
\end{proof}
\subsection{Inequalities for resolvent matrices}

\begin{lemma}\label{lem: res inequalities}
Let \(0 \le K < n\). For all \((\J, \K) \in \mathcal J_{K}\)
\begin{eqnarray}\label{eq: res 1}
\frac{1}{mn}\Tr |\RR^{(\J, \K)}|^2 \le \frac{1}{v} \imag m_n^{(\J, \K)} + \frac{|\J|-|\K|}{2mnv}.
\end{eqnarray}
For all \(j = 1, \ldots , mn\) such that \(j \notin \J\)
\begin{eqnarray}\label{eq: res 2}
\sum_{k}{\vphantom{\sum}}^* |\RR_{jk}^{(\J, \K)}|^2 \le \frac{1}{v} \imag \RR_{jj}^{(\J, \K)}, 
\end{eqnarray}
where \(\sum_{k}{\vphantom{\sum}}^*\) is the sum over all \(k = 1, \ldots, 2mn\) such that \(k \notin \J \cup \K \).  
\end{lemma}
\begin{proof}
The proof of~\eqref{eq: res 1} follows from the following inequality 
\begin{eqnarray*}
\frac{1}{mn}\Tr |\RR^{(\J, \K)}|^2 = \frac{1}{m n v} \imag \Tr \RR^{(\J, \K)} \le \frac{1}{v} \imag m_n^{(\J, \K)} + \frac{|\J|-|\K|}{2mnv}.
\end{eqnarray*}
The bound~\eqref{eq: res 2} follows from the eigenvalue decomposition of \(\V(z)\).
\end{proof}

\section{Bounds for the  Kolmogorov distance between distribution functions  via  Stieltjes transforms}

We reformulate the following smoothing inequality proved in~\cite{GotTikh2003}[Corollary~2.3], which allows to relate   distribution functions to their Stieltjes transforms. Let  \(G(x)\) be an arbitrary distribution function  which support is an interval or union of non-intersecting intervals, say \(\mathbb J=\supp G(x)=\cup_{\alpha=1}^m\mathbb J_{\alpha}\) and \(\mathbb J_{\alpha}=[a_{\alpha},b_{\alpha}]\). Additionally we assume that \(G(x)\) has an absolutely continues density which is bounded and for any end-point \(c\) of the support \(\J\) it behaves as \(g(x)\sim (x-c)^{\frac12}\). For any \(x\in\mathbb J\) we define \(\gamma(x)\eqdef\min_{\alpha}\{|x-a_{\alpha}|, |b_{\alpha}-x|\}\).
Given \(\frac12\min_{\alpha}\{b_{\alpha}-a_{\alpha}\}>\varepsilon>0\) introduce the interval \(\mathbb J_{\alpha}^{(\varepsilon)}=\{x\in[a_{\alpha},b_{\alpha}]:\, \gamma(x)\ge\varepsilon\}\) and
\(\mathbb J'_{\varepsilon}=\cup_{\alpha=1}^m\mathbb J_{\alpha}^{(\varepsilon/2)}\).
For a distribution function \(F\) denote by \(S_F(z)\) its Stieltjes transform,
\begin{eqnarray*}
S_F(z) \eqdef \int_{-\infty}^{\infty}\frac1{x-z}dF(x).
\end{eqnarray*}
We also denote 
\begin{eqnarray}\label{a def}
a \eqdef \sqrt 2 + 1.
\end{eqnarray}
\begin{statement}\label{smoothing}
Let \(v_0>0\) and \(\frac12>\varepsilon>0\) be positive numbers such that
\begin{eqnarray}\label{avcond}
2v a\le \varepsilon^{3/2}.
\end{eqnarray}
Denote  \(v'=v/\sqrt{\gamma}\). If \(G\) denotes the  distribution function satisfying conditions above, and \(F\) is any distribution function,
there exist some absolute constants \(C_1\) and \(C_2\) such that
\begin{eqnarray*}
\Delta(F,G)&\eqdef& \sup_{x}|F(x)-G(x)|\\
&\le& 2 \sup_{x\in\mathbb J'_{\varepsilon}}\bigg|\imag\int_{-\infty}^x\big(S_F(u+i v')-S_G(u+i v')\big)du\bigg|+C_1 v
	+C_2\varepsilon^{\frac32}.
	\end{eqnarray*}	
\end{statement}

\begin{remark}\label{rem2.2}
For any \(x\in\mathbb J_{\varepsilon}\) we have \(\gamma=\gamma(x)\ge\varepsilon\)
and according to condition \eqref{avcond}, \(\frac{av}{\sqrt\gamma}\le \frac{\varepsilon}2\).
\end{remark}

\begin{lemma}\label{Cauchy}
Let \(0<v\le \frac{\varepsilon^{3/2}}{2a}\) and \(V > v\). Denote  \(v' \eqdef v/\sqrt{\gamma}\). The following inequality holds
\begin{eqnarray*}
&&\sup_{x\in\mathbb J'_{\varepsilon}}\left|\int_{-\infty}^x(\imag(S_F(u+iv')-S_G(u+iv'))du\right|\notag\\
&&\qquad\qquad\qquad\le \int_{-\infty}^{\infty}|S_F(u+iV)-S_G(u+iV)|du +\sup_{x\in\mathbb J'_{\varepsilon}}\left|\int_{v'}^V\left(S_F(x+iu)-S_G(x+iu)\right)du\right|.\end{eqnarray*}
\end{lemma}

\begin{proof}Let \(x\in \mathbb J'_{\varepsilon}\) be fixed. Let \(\gamma=\gamma(x)\). Put \(z=u+iv'\).   Since \(v'=\frac v{\sqrt{\gamma}}\le \frac{\varepsilon}{2a}\), see \eqref{avcond},  we may assume without loss of generality that \(v'\le 4\) for  \(x\in\mathbb J'_{\varepsilon}\).  Since the functions of \(S_F(z)\) and \(S_G(z)\) are analytic in the upper half-plane, it is enough to use Cauchy's theorem. We can write for \(x\in\mathbb J'_{\varepsilon}\)
\begin{eqnarray*} 
\int_{-\infty}^{x}\imag(S_F(z)-S_G(z))du=\imag\{\lim_{L\to\infty}\int_{-L}^x(S_F(u+iv')-S_G(u+iv'))du\}.
\end{eqnarray*}
By Cauchy's integral formula, we have
\begin{eqnarray*}
\int_{-L}^x (S_F(z)-S_G(z))du&=&\int_{-L}^x(S_F(u+iV)-S_G(u+iV))\,du\\
&+&\int_{v'}^V (S_F(-L+iu)-S_G(-L+iu))\,du -\int_{v'}^V (S_F(x+iu)-S_G(x+iu))\,du
\end{eqnarray*}
Denote by \(\xi\text{ (resp. }\eta)\) a random variable with distribution function \(F(x)\) (resp. \(G(x)\)). Then we have
\begin{eqnarray*} 
|S_F(-L+iu)|=|\E (\xi+L-iu)^{-1} |\le (v')^{-1}\Pb (|\xi|>L/2)+2/L,
\end{eqnarray*}
for any \(v'\le u\le V\).
Similarly,
\begin{eqnarray*}
|S_G(-L+iu)|\le {v'}^{-1}\Pb(|\eta|>L/2)+2/L.
\end{eqnarray*}
These inequalities imply that
\begin{eqnarray*}
\bigg|\int_{v'}^V(S_F(-L+iu)-S_G(-L+iu))du\bigg|\to 0\quad\text{as}\quad L\to\infty,
\end{eqnarray*}
which completes the proof. 
\end{proof}

Combining the results of Proposition \ref{smoothing} and Lemma \ref{Cauchy}, we get
\begin{corollary}\label{smoothing1}
Under the conditions of Proposition \ref{smoothing} the following inequality holds
\begin{eqnarray*}
\Delta(F,G)&\le& 2\int_{-\infty}^{\infty}|S_F(u+iV)-S_G(u+iV)|du+C_1v+C_2\varepsilon^{\frac32}\\
&+& 2 \sup_{x\in\mathbb J'_{\varepsilon}}\int_{v'}^V|S_F(x+iu)-S_G(x+iu)|du,
\end{eqnarray*}
where \(v'=\frac v{\sqrt{\gamma}}\) and \(C_1,C_2 >0\) denote absolute constants.
\end{corollary}

\subsection{Proof of Bounds for the Kolmogorov Distance}\label{kolmdistance}
\begin{proof} {\it Proposition \ref{smoothing}}.
The proof of Proposition~\ref{smoothing} is a straightforward adaptation of the proof from~\cite{GotTikh2003}[Lemma~2.1]. We include it here for the sake of completeness. First we note that
\begin{eqnarray*}
\sup_x|F(x)-G(x)|&=&\sup_{x\in\mathbb J}|F(x)-G(x)|=\max\bigg\{\sup_{x\in\mathbb J_{\varepsilon}}|F(x)-G(x)|,\\
&&\sup_{x\in[a_{\alpha},a_{\alpha}+\varepsilon]}|F(x)-G(x)|,\sup_{x\in[b_{\alpha}-\varepsilon,b_{\alpha}]}|F(x)-G(x)|, \alpha=1,\ldots,m\bigg\}.
\end{eqnarray*}
Without loss of generality we shall assume that \(a_1\le b_1<a_2\le b_2<\cdots\le a_{m}\le b_m\). 
Consider \(x\in[a_{1},a_{1}+\varepsilon]\) we have
\begin{eqnarray*}
-G(a_{\alpha}+\varepsilon)&\le& F(x)-G(x)\le F(a_1+\varepsilon)-G(a_1+\varepsilon)+G(a_1+\varepsilon)\\
&\le& \sup_{x\in\mathbb J_1^{(\varepsilon)}}|F(x)-G(x)|+
G(a_1+\varepsilon).
\end{eqnarray*}
This inequality yields
\begin{eqnarray*}
\sup_{x\in[a_1,a_1+\varepsilon]}|F(x)-G(x)|\le \sup_{x\in\mathbb J_1^{(\varepsilon)}}|F(x)-G(x)|+G(a_1+\varepsilon).
\end{eqnarray*}
Let \(x\in [a_{\alpha},a_{\alpha}+\varepsilon]\) for \(\alpha=2,\ldots,m\). Note that \(G(b_{\alpha-1})=G(a_{\alpha})\).
We may write
\begin{eqnarray*}
&&F(b_{\alpha-1})-G(b_{\alpha-1})-(G(a_{\alpha}+\varepsilon)-G(a_{\alpha}))\le F(x)-G(x)\notag\\
&&\qquad\qquad\qquad\qquad\le (F(a_{\alpha}+\varepsilon)-G(a_{\alpha}+\varepsilon))+(G(a_{\alpha}+\varepsilon)-G(a_{\alpha})).
\end{eqnarray*}
From here it follows that
\begin{eqnarray*}
\sup_{x\in[a_{\alpha},a_{\alpha}+\varepsilon}|F(x)-G(x)|&\le& \sup_{x\in\mathbb J_{\alpha}^{(\varepsilon)}}|F(x)-G(x)|+\sup_{x\in\mathbb J_{\alpha-1}^{(\varepsilon)}}|F(x)-G(x)|\notag\\&+&(G(a_{\alpha}+\varepsilon)-G(a_{\alpha})).
\end{eqnarray*}
By induction we get for any \(\alpha=1,\ldots,m\)
\begin{eqnarray*}
\sup_{x\in\mathbb J_{\alpha}}|F(x)-G(x)|\le m\max_{1\le i\le m}\sup_{x\in\mathbb J_{i}^{(\varepsilon)}}|F(x)-G(x)|+m\max_{1\le i\le m}(G(a_{i}+\varepsilon)-G(a_{i})).
\end{eqnarray*}
Similarly we get
\begin{eqnarray*}
\sup_{x\in[b_{\alpha}-\varepsilon,b_{\alpha}]}|F(x)-G(x)|\le  m\max_{1\le i\le m}\sup_{x\in\mathbb J_{i}^{(\varepsilon)}}|F(x)-G(x)|+m\max_{1\le i\le m}(G(b_{i}+\varepsilon)-G(b_{i})).
\end{eqnarray*}
Note that \(G(a_{\alpha}+\varepsilon)-G(a_{\alpha})\le C\varepsilon^{3/2}\)  with some absolute constant \(C>0\).
Combining all these relations  we get
\begin{eqnarray}\label{gav1}
\sup_x|F(x)-G(x)|\le\Delta_{\varepsilon}(F,G)+C\varepsilon^{3/2},
\end{eqnarray}
where \(\Delta_{\varepsilon}(F,G)=\sup_{x\in\mathbb J_{\varepsilon}}|F(x)-G(x)|\).
We denote \(v' \eqdef v/\sqrt{\gamma}\). For any \(x\in \mathbb J'_{\varepsilon}\)
\begin{eqnarray}\label{smoothm}
&&\bigg|\frac1{\pi}\imag\Big(\int_{-\infty}^x(S_F(u+iv')-S_G(u+iv'))du\Big)\bigg| \ge \frac1{\pi} \imag\Big(\int_{-\infty}^x(S_F(u+iv')-S_G(u+iv'))du\Big)\nonumber\\
&&=\frac1{\pi}\left[\int_{-\infty}^{x}\int_{-\infty}^{\infty}\frac{v'd(F(y)-G(y))}{(y-u)^2+{v'}^2}du\right]\nonumber = \frac1{\pi}\int_{-\infty}^x\left[\int_{-\infty}^{\infty}\frac{2v'(y-u)(F(y)-G(y))dy}{((y-u)^2+{v'}^2)^2}\right]\nonumber\\
&&=\frac1{\pi}\int_{-\infty}^{\infty}(F(y)-G(y))\left[\int_{-\infty}^x\frac{2v'(y-u)}{((y-u)^2+{v'}^2)^2}du\right]dy\nonumber\\
&&=\frac1{\pi}\int_{-\infty}^{\infty}\frac{F(x-v'y)-G(x-v'y)}{y^2+1} dy,  \qquad \text{by change of variables}.
\end{eqnarray}
Furthermore, using the definition~\eqref{a def} of \(a\) and \(\Delta(F,G)\) we note that
\begin{eqnarray}\label{int1}
\frac1{\pi} \int_{|y|>a}\frac{|F(x-v'y)-G(x-v'y)|}{y^2+1}dy\le (1-\beta)\Delta(F,G).
\end{eqnarray}
Since \(F\) is non decreasing, we have
\begin{eqnarray*}
\frac1{\pi}\int_{|y|\le a}\frac{F(x-v'y)-G(x-v'y)}{y^2+1}dy\ge \frac1{\pi}\int_{|y|\le a}\frac{F(x-v'a)-G(x-v'y)}{y^2+1}dy
\\
\ge (F(x-v'a)-G(x-v'a))\beta -\frac1{\pi}\int_{|y|\le a}|G(x-v'y)-G(x-v'a)|dy.
\end{eqnarray*}
These inequalities together imply (using a change of variables in the last step)
\begin{eqnarray} \label{smoth1}
&&\frac1{\pi}\int_{-\infty}^{\infty}\frac{F(x-v'y)-G(x-v'y)}{y^2+1}dy\ge
\beta(F(x-v'a)-G(x-v'a))\nonumber\\
&&-\frac1{\pi}\int_{|y|\le a}|G(x-v'y)-G(x-v'a)|dy-(1-\beta)\Delta(F,G) 
\ge \beta(F(x-v'a)-G(x-v'a))\nonumber \\
&&-\frac1{v'\pi}\int_{|y|\le v'a}|G(x-y)-G(x-v'a)|dy-(1-\beta)\Delta(F,G) .
\end{eqnarray}
Note that according to Remark \ref{rem2.2}, \(x\pm v'a\in\mathbb J'_{\varepsilon}\) for any \(x\in\mathbb J_{\varepsilon}\). Assume first that \(x_n\in\mathbb J_{\varepsilon}\) is a sequence  such that \(F(x_n)-G(x_n)\to\Delta_{\varepsilon}(F,G)\). Then \(x_n'\eqdef x_n+v'a\in\mathbb J'_{\varepsilon}\).
Using \eqref{smoothm} and \eqref{smoth1}, we get
\begin{eqnarray}\label{gav20}
&&\sup_{x\in\mathbb J'_{\varepsilon}}\left|\imag\int_{-\infty}^x(S_F(u+iv')-S_G(u+iv'))du\right|\nonumber\\
&&\ge \imag\int_{-\infty}^{x_n'}(S_F(u+iv')-S_G(u+iv'))du \ge\beta(F(x_n'-v'a)-G(x_n'-v'a))\nonumber\\ 
&&\quad\quad-\frac1{\pi v} \sup_{x\in\mathbb J'_{\varepsilon}}{\sqrt{\gamma}}\int_{|y|\le 2v'a}|G(x+y)-G(x)|dy-(1-\beta)\Delta(F,G) \nonumber\\
&&=\beta(F(x_n)-G(x_n)) -\frac1{\pi v} \sup_{x\in\mathbb J'_{\varepsilon}}{\sqrt{\gamma}}\int_{|y|<2v'a}|G(x+y)-G(x)|dy-(1-\beta)\Delta(F,G) .
\end{eqnarray}
Assume for definiteness  that \(y>0\). Recall that \( \varepsilon\le 2\gamma\), for any \(x\in\mathbb J_{\varepsilon}'\).
By Remark~\ref{rem2.2} with \(\varepsilon/2\) instead of \(\varepsilon\), we have \(0<y\le 2v'a\le\sqrt 2\varepsilon\), for any \(x\in\mathbb J'_{\varepsilon}\).
By conditions of Proposition, we have,
\begin{eqnarray*}
|G(x+y)-G(x)|&\le& y\sup_{u\in[x,x+y]}G'(u)\le yC\sqrt{\gamma+y}\\
&\le& C y\sqrt{\gamma+2v'a} \le C y\sqrt{\gamma+\varepsilon}\le Cy\sqrt{\gamma}.
\end{eqnarray*}
This yields after integrating in \(y\)
\begin{eqnarray}\label{gav3}
\frac1{\pi v} \sup_{x\in\mathbb J'_{\varepsilon}}{\sqrt{\gamma}}\int_{0\le y\le 2v'a}|G(x+y)-G(x)|dy \le\frac C{v} \sup_{x\in\mathbb J'_{\varepsilon}}\gamma{v'}^2\le Cv.
\end{eqnarray}
Similarly we get that
\begin{eqnarray}\label{gav3*}
\frac1{\pi v} \sup_{x\in\mathbb J'_{\varepsilon}}{\sqrt{\gamma}}\int_{0\ge y\ge -2v'a}|G(x+y)-G(x)|dy \le\frac C{v} \sup_{x\in\mathbb J'_{\varepsilon}}\gamma{v'}^2\le Cv.
\end{eqnarray}
By inequality \eqref{gav1}
\begin{eqnarray}\label{gt1}
\Delta_{\varepsilon}(F,G)\ge \Delta(F,G)-C\varepsilon^{\frac32}.
\end{eqnarray}
The  inequalities \eqref{gav20}, \eqref{gt1} and \eqref{gav3}, \eqref{gav3*} together yield as \(n\) tends to infinity
\begin{eqnarray}
&&\sup_{x\in\mathbb J'_{\varepsilon}}\left|\imag\int_{-\infty}^x(S_F(u+iv')-S_G(u+iv'))du\right| \ge(2\beta-1)\Delta(F,G)-Cv -C\varepsilon^{\frac32}
,\label{Deltaeps}
\end{eqnarray}
for some constant \(C>0\). Similar arguments may be used to prove this inequality in case  there is a sequence \(x_n\in\mathbb J_{\varepsilon}\) such \(F(x_n)-G(x_n)\to-\Delta_{\varepsilon}(F,G)\). In view of \eqref{Deltaeps} and \(2\beta-1=1/2\) this completes the proof.
\end{proof}

\def\polhk#1{\setbox0=\hbox{#1}{\ooalign{\hidewidth
			\lower1.5ex\hbox{`}\hidewidth\crcr\unhbox0}}}

\end{document}